\documentclass[a4paper, 11pt]{amsart}
\usepackage[utf8]{inputenc}
\usepackage[english]{babel}
\usepackage[T1]{fontenc}
\usepackage{amsmath, amssymb, amsfonts}
\usepackage[all]{xy}
\usepackage{enumerate}
\usepackage{graphicx}
\usepackage{geometry}
\usepackage{tikz}
\usetikzlibrary{arrows}
\usepackage{hyperref}

\title{Graphs of quantum groups and K-amenability}
\author{Pierre Fima and Amaury Freslon}
\thanks{The first author is partially supported by ANR grants OSQPI and NEUMANN}
\address{Univ. Paris Diderot, Sorbonne Paris Cité, UMR 7586, , 75013, Paris, France}
\email{pfima@math.jussieu.fr, freslon@math.jussieu.fr}
\date{\today}

\theoremstyle{plain}
\newtheorem{thm}{Theorem}[section]
\newtheorem{prop}[thm]{Proposition}
\newtheorem{cor}[thm]{Corollary}
\newtheorem{lem}[thm]{Lemma}

\theoremstyle{definition}
\newtheorem{de}[thm]{Definition}
\newtheorem{ex}[thm]{Example}

\theoremstyle{remark}
\newtheorem{rem}[thm]{Remark}

\DeclareMathOperator{\E}{E}
\DeclareMathOperator{\Id}{Id}
\DeclareMathOperator{\Irr}{Irr}
\DeclareMathOperator{\HNN}{HNN}
\DeclareMathOperator{\I}{Im}
\DeclareMathOperator{\II}{II}
\DeclareMathOperator{\KK}{KK}
\DeclareMathOperator{\Pol}{Pol}
\DeclareMathOperator{\Span}{Span}
\DeclareMathOperator{\V}{V}

\newcommand{\B}{\mathcal{B}}
\newcommand{\C}{\mathbb{C}}
\newcommand{\D}{\Delta}
\newcommand{\EE}{\mathbb{E}}
\newcommand{\F}{\mathcal{F}}
\newcommand{\Fi}{\varphi}
\newcommand{\G}{\mathbb{G}}
\newcommand{\Gr}{\mathcal{G}}
\newcommand{\HH}{\mathcal{H}}
\newcommand{\K}{\mathcal{K}}
\newcommand{\LL}{\mathcal{L}}

\newcommand{\PP}{\mathcal{P}}
\newcommand{\Q}{\mathcal{Q}}
\newcommand{\R}{\mathcal{R}}
\newcommand{\T}{\mathcal{T}}

\newcommand{\Lde}{{\rm L}^2}
\newcommand{\h}{\widehat}
\newcommand{\ii}{\imath}
\newcommand{\rev}{\overline}

\newcommand{\ot}{\otimes}

\begin{document}

\begin{abstract}
Building on a construction of J-P. Serre, we associate to any graph of C*-algebras a maximal and a reduced fundamental C*-algebra and use this theory to construct the fundamental quantum group of a graph of discrete quantum groups. This construction naturally gives rise to a quantum Bass-Serre tree which can be used to study the K-theory of the fundamental quantum group. To illustrate the properties of this construction, we prove that if all the vertex qantum groups are amenable, then the fundamental quantum group is K-amenable. This generalizes previous results of P. Julg, A. Valette, R. Vergnioux and the first author.
\end{abstract}

\maketitle

\section{Introduction}

\noindent One of the first striking application of K-theory to the theory of operator algebras was the proof by M.V. Pimsner and D.V. Voiculescu in \cite{pimsnervoiculescu1982} that the reduced C*-algebras of free groups with different number of generators are not isomorphic. It relies on an involved computation of the K-theory of these C*-algebras, which appear not to be equal. From then on, K-theory of group C*-algebras has been a very active field of research, in particular in relation with algebraic and geometric problems, culminating in the celebrated \emph{Baum-Connes conjecture} (see for example \cite{baum1994classifying}).

\vspace{0.2cm}

\noindent Along the way J. Cuntz introduced the notion of \emph{K-amenability} in \cite{cuntz1983k}. Noticing that the maximal and reduced C*-algebras of free groups have the same K-theory, he endeavoured to give a conceptual explanation of this fact based on a phenomenon quite similar to amenability, but on a K-theoretical level. Combined with a short and elegant computation of the K-theory of the maximal C*-algebras of free groups, he could therefore recover the result of M.V. Pimsner and D.V. Voiculescu in more conceptual way. K-amenability implies in particular that the K-theory of any reduced crossed-product by the group is equal to the K-theory of the corresponding full crossed-product, thus giving a powerful tool for computing K-theory of C*-algebras. The original definition of J. Cuntz was restricted to discrete groups but was later generalized to arbitrary locally compact groups by P. Julg and A. Valette in \cite{julg1984k}. In this seminal paper, they also proved that any group acting on a tree with amenable stabilizers is K-amenable. As particular cases, one gets that free products and HNN extensions of amenable groups are K-amenable. This result was later extended to groups acting on trees with K-amenable stabilizers by M.V. Pimsner in \cite{pimsner1986kk}. Let us also mention the notion of \emph{K-nuclearity} developped by G. Skandalis in \cite{skandalis1988notion} and further studied by E. Germain who proved in \cite{germain1996kk} that a free product of amenable C*-algebras is K-nuclear. However, we will concentrate in the present paper on the Julg-Valette theorem and extend it to the setting of discrete quantum groups. Let us first recall an algebraic description of groups acting on trees which will be better suited to our purpose.

\vspace{0.2cm}

\noindent A \emph{graph of groups} is the data of a graph $\Gr$ together with groups attached to each vertex and each edge in a way compatible with the graph structure. One can generalize the topological construction of the fundamental group $\pi_{1}(\Gr)$ to include the additional data of the groups and thus obtain the notion of \emph{fundamental group of a graph of groups}. The core of Bass-Serre theory, developped in \cite{serre1977arbres}, is a powerful correspondance between this construction and the structure of groups acting on trees. Two very important particular cases are amalgamated free products and HNN extensions. In both cases, Bass-Serre theory provides us with a tree on which the group acts in a canonical way.

\vspace{0.2cm}

\noindent In the context of locally compact quantum groups, K-amenability was defined and studied by R. Vergnioux in \cite{vergnioux2004k}, building on the work of S. Baaj and G. Skandalis on equivariant KK-theory for coactions of Hopf-C*-algebras \cite{baaj1989c}. In the discrete case, R. Vergnioux was able to prove that several classical characterizations still hold in the quantum setting (some of them are recalled in Theorem \ref{thm:kamenabilitydefinition}). He also proved the K-amenability of amalgamated free products of amenable discrete quantum groups. His proof used the first example of a \emph{quantum Bass-Serre tree}. This is a pair of Hilbert C*-modules over the C*-algebra of a compact quantum group endowed with actions which can be used as a "geometric object" for the study of K-theoretical properties. Similar techniques where used by the first author to prove K-amenability of HNN extensions of amenable discrete quantum groups in \cite{fima2012k}. The use of quantum Bass-Serre trees also proved crucial in the study of the Baum-Connes conjecture for discrete quantum groups by R. Vergnioux and C. Voigt in \cite{vergnioux2011k}.

\vspace{0.2cm}

\noindent In the present paper, we generalize the construction of the fundamental group to graphs of discrete quantum groups. As one could expect, this fundamental quantum group comes along with a quantum Bass-Serre tree which can be used to construct a natural KK-element. Our construction is in some sense the most general construction of a quantum Bass-Serre tree such that the "quotient" by the action of the quantum group is a classical graph. We then use techniques combining the ones of \cite{vergnioux2004k} and \cite{fima2012k} to prove that if all the vertex groups are amenable, then the resulting quantum group is K-amenable. In view of the Bass-Serre equivalence, this generalizes the result of \cite{julg1984k}. Note that this gives a large class of K-amenable discrete quantum groups and improves the aforementionned results in the quantum setting. For example, it is known by \cite{fima2012k} that an HNN extension of amenable discrete quantum groups is K-amenable, but it was not known that if we again take an HNN extension or a free product with a third amenable discrete quantum group, then the resulting quantum group will still be K-amenable.

\vspace{0.2cm}

\noindent Let us now outline the organization of the paper. In Section \ref{sec:preliminaries}, we specify some notations and conventions used all along the paper and we give some basic definitions and results concerning quantum groups and K-amenability. In section \ref{sec:graphCstar}, we associate to any graph of C*-algebras a full and a reduced fundamental C*-algebra and give some structure results. This section is rather long but contains most of the technical results of this paper. It ends with an "unscrewing" technique which can be used to prove that some properties of the vertex C*-algebras are inherited by the fundamental C*-algebras. In Section \ref{sec:fundamentalqgroup}, we use the previous results to define the fundamental quantum group of a graph of quantum groups and describe its Haar state and representation theory. Eventually, we prove in Section \ref{sec:kamenability} that the fundamental quantum group of a graph of amenable discrete quantum groups is K-amenable.  Note that one could also define \emph{graphs of von Neumann algebras} and work out similar constructions. This is outlined in the Appendix.

\section{Preliminaries}\label{sec:preliminaries}

\subsection{Notations and conventions}

In this paper all the Hilbert spaces, Hilbert C*-modules and C*-algebras are assumed to be separable. Moreover, all the C*-algebras are assumed to be unital. The scalar products on Hilbert spaces or Hilbert C*-modules are denoted by $\langle .,.\rangle$ and are supposed to be linear in the second variable. For two Hilbert spaces $H$ and $K$, $\B(H, K)$ will denote the set of bounded linear maps from $H$ to $K$ and $\B(H):=\B(H, H)$. For a C*-algebra $A$ and Hilbert $A$-modules $\HH$a nd $\mathcal{K}$, we denote by $\LL_{A}(\HH,\mathcal{K})$ the set of bounded adjointable $A$-linear operators from $\HH$ to $\mathcal{K}$ and $\mathcal{L}_A(\HH)=\mathcal{L}_A(\HH,\HH)$. 

\vspace{0.2cm}

\noindent We will also use the following terminology : if $\HH$ is a Hilbert $A$-module and $\varphi\in A^{*}$ is a state, the \emph{GNS construction} of $(\HH,\varphi)$ is the triple $(H, \pi, \eta)$, where $H$ is the Hilbert space obtained by separation and completion of $\HH$ with respect to the scalar product $\langle \xi, \eta\rangle_{H} = \varphi(\langle \xi, \eta\rangle_{\HH})$, $\eta\;:\;\HH\rightarrow H$ is the canonical linear map with dense range and $\pi : \mathcal{L}_{A}(\HH)\rightarrow \B(H)$ is the induced unital $*$-homomorphism. Note that $\pi$ and $\eta$ are faithful as soon as $\varphi$ is. Observe also that if $\mathcal{K}$ is another Hilbert $A$-module and if $(K, \rho, \xi)$ denotes the GNS construction of $(\mathcal{K}, \varphi)$, then we also have an obvious induced linear map $\mathcal{L}_{A}(\HH, \mathcal{K})\rightarrow \mathcal{B}(H, K)$ which respects the adjoint and the composition (if we take a third Hilbert $A$-module). 

\vspace{0.2cm}

\noindent If $\Gr$ is a graph in the sense of \cite[Def 2.1]{serre1977arbres}, its vertex set will be denoted $\V(\Gr)$ and its edge set will be denoted $\E(\Gr)$. For $e\in\E(\Gr)$ we denote by $s(e)$ and $r(e)$ respectively the source and range of $e$ and by $\rev{e}$ the inverse edge of $e$. An \emph{orientation} of $\Gr$ is a partition $\E(\Gr) = \E^{+}(\Gr)\sqcup \E^{-}(\Gr)$ such that $e\in \E^{+}(\Gr)$ if and only if $\rev{e}\in \E^{-}(\Gr)$.

\vspace{0.2cm}

\noindent Finally, we will always denote by $\ii$ the identity map. 

\subsection{Compact quantum groups}\label{section:CQG}

We briefly recall the main definitions and results of the theory of compact quantum groups in order to fix notations. The reader is referred to \cite{woronowicz1995compact} or \cite{maes1998notes} for details and proofs.

\begin{de}
A \emph{compact quantum group} is a pair $\G = (C(\G), \D)$ where $C(\G)$ is a unital C*-algebra and $\D : C(\G)\rightarrow C(\G)\otimes C(\G)$ is a unital $*$-homomorphism such that
\begin{equation*}
(\D\otimes \ii)\circ\D = (\ii \otimes \D)\circ \D
\end{equation*}
and the linear span of $\D(C(\G))(1\otimes C(\G))$ as well as the linear span of $\D(C(\G))(C(\G)\otimes 1)$ are dense in $C(\G)\otimes C(\G)$. 
\end{de}

\begin{thm}[Woronowicz]
Let $\G$ be a compact quantum group. There is a unique state $h\in C(\G)^{*}$, called the \emph{Haar state} of $\G$, such that for every $x\in C(\G)$,
$$
(h\otimes \ii)\circ\D(x) = h(x).1\text{ and }(\ii\otimes h)\circ\D(x) = h(x).1.
$$
\end{thm}

\noindent The Haar state need not be faithful. Let $C_{\text{red}}(\G)$ be the C*-algebra obtained by the GNS construction of the Haar state. $C_{\text{red}}(\G)$ is called the \emph{reduced C*-algebra} of the compact quantum group $\G$. By the invariance properties of the Haar state, the coproduct $\Delta$ induces a coproduct on $C_{\text{red}}(\G)$ which turns it into a compact quantum group called the \emph{reduced form} of $\G$. The Haar state on the reduced form of $\G$ is faithful by construction.

\begin{de}
Let $\G$ be a compact quantum group. A \emph{representation} of $\G$ of dimension $n$ is a matrix $(u_{i j})\in M_{n}(C(\G)) = M_{n}(\C)\otimes C(\G)$ such that for all $1\leqslant i,j\leqslant n$,
\begin{equation*}
\D(u_{i j}) = \sum_{k} u_{ik}\otimes u_{k j}.
\end{equation*}
\end{de}

\noindent A representation is called \emph{unitary} if $u=(u_{ij})\in M_{n}(\C)\otimes C(\G)$ is a unitary. An \emph{intertwiner} between two representations $u$ and $v$ of dimension respectively $n$ and $m$ is a linear map $T : \C^{n} \rightarrow \C^{m}$ such that $(T\ot \ii)u = v( T\ot\ii)$. If there exists a unitary intertwiner between $u$ and $v$, they are said to be \emph{unitarily equivalent}. A representation is said to be \emph{irreducible} if its only self-intertwiners are the scalar multiples of the identity. The \emph{tensor product} of the two representations $u$ and $v$ is the representation
\begin{equation*}
u\otimes v = u_{13}v_{23}\in M_{n}(\C)\otimes M_{m}(\C)\otimes C(\G) \simeq M_{nm}(\C)\otimes C(\G).
\end{equation*}

\begin{thm}[Woronowicz]
Every unitary representation of a compact quantum group is unitarily equivalent to a direct sum of irreducible unitary representations.
\end{thm}

\noindent Let $\Irr(\G)$ be the set of equivalence classes of irreducible unitary representations of $\G$ and, for $\alpha\in \Irr(\G)$, denote by $u^{\alpha}$ a representative of $\alpha$. The linear span of the elements $u^{\alpha}_{ij}$ for $\alpha\in \Irr(\G)$ forms a Hopf-$*$-algebra $\Pol(\G)$ which is dense in $C(\G)$. Its enveloping C*-algebra is denoted $C_{\text{max}}(\G)$ and it has a natural quantum group structure called the \emph{maximal form} of $\G$. By universality, there is a surjective $*$-homomorphism
\begin{equation*}
\lambda_{\G} : C_{\text{max}}(\G) \rightarrow C_{\text{red}}(\G)
\end{equation*}
which intertwines the coproducts.

\begin{rem}\label{rem:discreteqgroups}
The C*-algebras $C_{\text{red}}(\G)$ and $C_{\text{max}}(\G)$ should be thought of as the reduced and maximal C*-algebras of the \emph{dual discrete quantum group} $\h{\G}$. This point of view justifies the terminology "discrete quantum groups" used in the paper.
\end{rem}

\noindent $C_{\text{max}}(\G)$ admits a one-dimensional representation $\varepsilon\,:\, C_{\text{max}}(\G)\rightarrow\C$, called the \emph{trivial representation} (or the \emph{counit}) and defined by $\varepsilon(u^{\alpha}_{ij})=\delta_{ij}$ for all $\alpha\in\Irr(\G)$ and every $i,j$. The counit is the unique unital $*$-homomorphism $\varepsilon\,:\,C_{\text{max}}(\G)\rightarrow\C$ such that $(\varepsilon\ot\ii)\Delta=\ii=(\ii\ot\varepsilon)\Delta$.

\subsection{K-amenability}

\begin{de}
A compact quantum group $\G$ is said to be \emph{co-amenable} if $\lambda_{\G}$ is an isomorphism. We will equivalently say that $\h{\G}$ is \emph{amenable}.
\end{de}

\noindent Like in the classical case, co-amenability has several equivalent characterizations, we only give the one which will be needed in the sequel (see \cite[Thm 3.6]{bedos2001co} for a proof).

\begin{prop}\label{prop:amenability}
A compact quantum group $\G$ is co-amenable if and only if the trivial representation factors through $\lambda_{\G}$.
\end{prop}

\noindent K-amenability admits similar characterizations on the level of KK-theory, which were proved by R. Vergnioux in \cite[Thm 1.4]{vergnioux2004k}. We refer the reader to \cite{blackadar1998k} for the basic definitions and results concerning KK-theory.

\begin{thm}[Vergnioux]\label{thm:kamenabilitydefinition}
Let $\G$ be a compact quantum group. The following are equivalent
\begin{itemize}
\item There exists $\gamma\in \KK(C_{\text{red}}(\G), \C)$ such that $\lambda_{\G}^{*}(\gamma) = [\varepsilon]$ in $\KK(C_{\text{max}}(\G), \C)$.
\item The element $[\lambda_{\G}]$ in invertible in $\KK(C_{\text{max}}(\G), C_{\text{red}}(\G))$.
\end{itemize}
\end{thm}

\noindent In any of those two equivalent situations, we will say that $\h{\G}$ is K-amenable.

\section{Graphs of C*-algebras}\label{sec:graphCstar}

\noindent In this section we give the general construction of a maximal and a reduced fundamental C*-algebra associated to a graph of C*-algebras.

\begin{de}
A \emph{graph of C*-algebras} is a tuple
\begin{equation*}
(\Gr, (A_{q})_{q\in \V(\Gr)}, (B_{e})_{e\in \E(\Gr)}, (s_{e})_{e\in \E(\Gr)})
\end{equation*}
where
\begin{itemize}
\item $\Gr$ is a connected graph.
\item For every $q\in\V(\Gr)$ and every $e\in\E(\Gr)$, $A_{q}$ and $B_{e}$ are unital C*-algebras.
\item For every $e\in \E(\Gr)$, $B_{\rev{e}} = B_{e}$.
\item For every $e\in\E(\Gr)$, $s_{e} : B_{e} \rightarrow A_{s(e)}$ is a unital faithful $*$-homomorphism.
\end{itemize}
For every $e\in \E(\Gr)$, we set $r_{e} = s_{\rev{e}} : B_{e} \rightarrow A_{r(e)}$, $B_{e}^{s} = s_{e}(B_{e})$ and $B_{e}^{r} = r_{e}(B_{e})$.
\end{de}

\noindent The notation will always be simplified in $(\Gr, (A_{q})_{q}, (B_{e})_{e})$.

\subsection{The maximal fundamental C*-algebra}

Like in the case of free products, the definition of the maximal fundamental C*-algebra is quite obvious and simple. However, it requires the choice of a maximal subtree of the graph $\Gr$ (which is implicit in the case of free products since the graph is already a tree, see Example \ref{ex:freeproduct}).

\begin{de}
Let $(\Gr, (A_{q})_{q}, (B_{e})_{e})$ be a graph of C*-algebras and let $\T$ be a maximal subtree of $\Gr$. The \emph{maximal fundamental C*-algebra with respect to $\T$} is the universal C*-algebra generated by the C*-algebras $A_{q}$ for $q\in\V(\Gr)$ and by unitaries $u_{e}$ for $e\in\E(\Gr)$ such that
\begin{itemize}
\item For every $e\in\E(\Gr)$, $u_{\rev{e}} = u_{e}^{*}$.
\item For every $e\in\E(\Gr)$ and every $b\in B_{e}$, $u_{\rev{e}}s_{e}(b)u_{e}=r_{e}(b)$.
\item For every $e\in\E(\T)$, $u_{e}=1$.
\end{itemize}
This C*-algebra will be denoted $\pi_{1}^{\text{max}}(\Gr,(A_{q})_{q},(B_{e})_{e},\T)$.
\end{de}

\begin{rem}
It is not obvious that this C*-algebra is not $0$ (i.e. that the relations admit a non-trivial representation). With natural additional assumptions, the non-triviality will be proved by the construction of the reduced fundamental C*-algebra and it will be clear that the inclusions of $A_q$ in the maximal fundamental C*-algebra are faithful.
\end{rem}

\begin{ex}\label{ex:freeproduct}
Let $A_{0}$ and $A_{1}$ be two C*-algebras and let $B$ be a C*-algebra together with injective $*$-homomorphisms $i_{k} : B\rightarrow A_{k}$ for $k = 0, 1$. Let $\Gr$ be the graph with two vertices $p_{0}$ and $p_{1}$ and two edges $e$ and $\overline{e}$, where $s(e)=p_1$ and $r(e)=p_2$. This graph is obviously a tree. Setting $B_{e} = B$, $s_{e} = i_{0}$, $r_{e} = i_{1}$ and $A_{p_{i}} = A_{i}$ yields a graph of C*-algebras whose maximal fundamental C*-algebra with respect to $\Gr$ is the maximal free product $A_{0}\ast^{\text{max}}_{B}A_{1}$ of $A_{0}$ and $A_{1}$ amalgamated over $B$.
\end{ex}

\begin{ex}\label{ex:hnn}
Let $A$ be a C*-algebra, $B$ a C*-subalgebra of $A$ and $\theta : B\rightarrow A$ an injective $*$-homomorphism. Let $\Gr$ be a graph with one vertex $p$ and two edges $e$ and $\overline{e}$, where $e$ is a loop from $p$ to $p$. Obviously, the only maximal subtree of $\Gr$ is the graph with one vertex $p$ and no edge. Setting $B_{e} = B$, $s_{e} = \ii$, $r_{e} = \theta$ and $A_{p} = A$ yields a graph of C*-algebras whose maximal fundamental C*-algebra with respect to 
$\{p\}$ is the maximal HNN extension $\HNN^{\text{max}}(A, B, \theta)$ as defined in \cite[Rmk 7.3]{ueda2005hnn}.
\end{ex}

\noindent By construction, the maximal fundamental C*-algebra of $(\Gr, (A_{q})_{q}, (B_{e})_{e})$ satisfies the following universal property.

\begin{prop}\label{prop:universal}
Let $(\Gr, (A_{q})_{q}, (B_{e})_{e})$ be a graph of C*-algebras, let $\T$ be a maximal subtree of $\Gr$ and let $H$ be a Hilbert space. Assume that for every $q\in \V(\Gr)$, we have a representation $\rho_{q}$ of $A_{q}$ on $H$ and that for every $e\in \E(\Gr)$, we have a unitary $U_{e}\in \B(H)$ such that $U_{\overline{e}}=U_e^*$, $U_{e} = 1$ for all $e\in \E(\T)$ and for every $b\in B_{e}$,
\begin{equation*}
U_{e}^{*}\rho_{s(e)}(s_{e}(b))U_{e} = \rho_{r(e)}(r_{e}(b)).
\end{equation*}
Then, there is a unique representation $\rho$ of $\pi_{1}^{\text{max}}(\Gr,(A_{q})_{q},(B_{e})_{e},\T)$ on $H$ such that for every $e\in\E(\Gr)$ and every $q\in\V(\Gr)$,
\begin{equation*}
\rho(u_e)=U_e \text{ and } \rho|_{A_q}=\rho_q.
\end{equation*}
\end{prop}

\begin{rem}\label{rem:pathmax}
Let $p_{0}\in \V(\Gr)$. Define $\mathcal{A}$ to be the linear span of $A_{p_{0}}$ and elements of the form $a_{0}u_{e_{1}} \dots u_{e_{n}}a_{n}$ where $(e_{1}, \dots, e_{n})$ is a path in $\Gr$ from $p_{0}$ to $p_{0}$, $a_{0}\in A_{p_{0}}$ and $a_{i}\in A_{r(e_{i})}$ for $1\leqslant i\leqslant n$. Observe that $\mathcal{A}$ is a dense $*$-subalgebra of $\pi_{1}^{\text{max}}(\Gr, (A_{q})_{q}, (B_{e})_{e}, \T)$. Indeed, it suffices to show that it contains $A_{q}$ for every $q\in \V(\Gr)$ and $u_{e}$ for every $e\in \E(\Gr)$. Let $q\in \V(\Gr)$ and $a\in A_{q}$. Let $w = (e_{1}, \dots, e_{n})$ be the unique geodesic path in $\T$ from $p_{0}$ to $q$. Since $e_{i}\in \E(\T)$, we have $u_{e_{i}} = 1$ for every $1\leqslant i\leqslant n$. Hence, $a = u_{e_{1}} \dots u_{e_{n}}a u_{\overline{e}_{n}}\dots u_{\overline{e}_{1}}\in\mathcal{A}$. Now, let $e\in \E(\Gr)\setminus\E(\T)$ and let $(e_{1}, \dots, e_{n})$ (resp. $(f_{1}, \dots, f_{m})$) be the geodesic path in $\T$ from $p_{0}$ to $s(e)$ (resp. $r(e)$). Then, $u_{e} = u_{e_{1}} \dots u_{e_{n}}u_{e}u_{\overline{f}_{m}}\dots u_{\overline{f}_{1}}\in \mathcal{A}$.
\end{rem}

\noindent We will need in the sequel the following slightly more general version of the universal property.

\begin{cor}\label{cor:universal}
Let $(\Gr, (A_{q})_{q}, (B_{e})_{e})$ be a graph of C*-algebras, let $\T$ be a maximal subtree of $\Gr$ and let $p_{0}\in \V(\Gr)$. Assume that for every $q\in \V(\Gr)$, there is a Hilbert space $H_{q}$ together with a representation $\rho_{q}$ of $A_{q}$ and that, for every $e\in \E(\Gr)$, there is a unitary $U_{e} : H_{r(e)}\rightarrow H_{s(e)}$ such that $U_e^*=U_{\overline{e}}$ and, for every $b\in B_{e}$,
\begin{equation*}
U_{e}^{*}\rho_{s(e)}(s_{e}(b))U_{e} = \rho_{r(e)}(r_{e}(b)).
\end{equation*}
Then, there exists a unique representation $\rho$ of $\pi_{1}^{max}(\Gr, (A_{q})_{q}, (B_{e})_{e}, \T)$ on $H_{p_{0}}$ such that $\rho|_{A_{p_0}}=\rho_0$ and, for every path $w = (e_{1}, \dots, e_{n})$ from $p_{0}$ to $p_{0}$ in $\Gr$ and every $a_{0} \in A_{s(e_{1})}, a_{i}\in A_{r(e_{i})}$,
\begin{equation*}
\rho(a_{0}u_{e_{1}} \dots u_{e_{n}}a_{n}) = \rho_{s(e_{1})}(a_{1})U_{e_{1}}\dots U_{e_{n}}\rho_{r(e_{n})}(a_{n}).
\end{equation*}
\end{cor}

\begin{proof}
The proof amounts to a suitable application of Proposition \ref{prop:universal}. Let $p, q\in\V(\Gr)$ and let $w = (e_{1}, \dots, e_{n})$ be the unique geodesic path in $\T$ from $p$ to $q$. Set $U_{pq}=U_{e_{1}} \dots U_{e_{n}}$ and observe that $U_{pq}^{*} = U_{qp}$. For every $q\in \V(\Gr)$ and every $e\in \E(\Gr)$, we can define a representation $\pi_{q} = U_{qp_{0}}^{*}\rho_q(.)U_{qp_{0}}$ of $A_{q}$ on $H_{p_{0}}$ and a unitary $V_e\in\mathcal{B}(H_{p_0})$ by $V_{e} = U_{p_{0}s(e)}U_{e}U_{r(e)p_{0}}$. It is easily checked that these satisfy the hyptohesis of Proposition \ref{prop:universal}, yielding the result.
\end{proof}

\subsection{The reduced fundamental C*-algebra}

We now turn to the construction of the reduced fundamental C*-algebra, which is more involved. The basic idea is to build a concrete representation of the C*-algebras forming the graph together with unitaries satisfying the required relations. To be able to carry out this construction, we will need an extra assumption. From now on, we assume that for every $e\in\E(\Gr)$, there exists a conditional expectation $\EE_{e}^{s} : A_{s(e)} \rightarrow B_{e}^{s}$ and we set $\EE_{e}^{r} = \EE_{\rev{e}}^{s} : A_{r(e)} \rightarrow B_{e}^{r}$.

\subsubsection{Path Hilbert modules}

For every $e\in\E(\Gr)$ let $(\HH_{e}^{s},\pi_{e}^{s},\eta_{e}^{s})$ be the GNS construction associated to the completely positive map $s_{e}^{-1}\circ \EE_{e}^{s}$. This means that $\HH_{e}^{s}$ is the right Hilbert $B_{e}$-module obtained by separation and completion of $A_{s(e)}$ with respect to the $B_{e}$-valued inner product
\begin{equation*}
\langle x,y\rangle=s_e^{-1}\circ \EE_{e}^{s}(x^*y)\quad\text{for}\quad x,y\in A_{s(e)}.
\end{equation*}
The right action of an element $b\in B_{e}$ is given by right multiplication by $s_{e}(b)$ and the representation
\begin{equation*}
\pi_{e}^{s} : A_{s(e)} \rightarrow \LL_{B_{e}}(\HH_{e}^{s})
\end{equation*}
is induced by the left multiplication. Finally, $\eta_{e}^{s} : A_{s(e)} \rightarrow \HH_{e}^{s}$ is the standard linear map with dense range. Let $\xi_{e}^{s}$ denote the image of $1_{A}$ in $\HH_{e}^{s}$. The triple $(\HH_{\rev{e}}^{s},\pi_{\rev{e}}^{s},\eta_{\rev{e}}^{s})$ will be denoted $(\HH_{e}^{r},\pi_{e}^{r},\eta_{e}^{r})$. Although it is not necessary, we will assume, for convenience and simplicity of notations, that for every $e\in\E(\Gr)$, the conditional expectations $\EE_{e}^{s}$ are GNS-faithful (i.e. that the representations $\pi_{e}^{s}$ are faithful). This allows us to identify $A_{s(e)}$ with its image in $\LL_{B_{e}}(\HH_{e}^{s})$. We will also use, for every $a\in A_{s(e)}$, the notation $\h{a}$ for $\eta_{e}^{s}(a) \in \HH_{e}^{s}$. One should however keep in mind that $\h{a}$ may be zero for some non-zero $a$. Let us also notice that the submodule $\xi_{e}^{s}.B_{e}$ of $\HH_{e}^{s}$ is orthogonally complemented. In fact, its orthogonal complement is the closure $(\HH_{e}^{s})^{\circ}$ of $\{\h{a}\vert a\in A_{s(e)}, \EE_{e}^{s}(a)=0\}$. We thus have an orthogonal decomposition
\begin{equation*}
\HH_{e}^{s} = (\xi_{e}^{s}.B_{e}) \oplus (\HH_{e}^{s})^{\circ}
\end{equation*}
with $(\HH_{e}^{s})^{\circ}.B_{e}^{s} = (\HH_{e}^{s})^{\circ}$. Similarly, the orthogonal complement of $\xi_{e}^{r}.B_{e}$ in $\HH_{e}^{r}$ will be denoted $(\HH_{e}^{r})^{\circ}$ .

\noindent We now turn to the construction of the Hilbert C*-module which will carry our faithful representation of the fundamental C*-algebra. Let $n\geqslant 1$ and let $w = (e_{1}, \dots, e_{n})$ be a path in $\Gr$. We define Hilbert C*-modules $\K_{0}$, $\K_{n}$ and $\K_{i}$ for $1\leqslant i\leqslant n-1$ by
\begin{itemize}
\item $\K_{0} = \HH_{e_{1}}^{s}$
\item If $e_{i+1}\neq \rev{e}_{i}$, then $\K_{i} = \HH_{e_{i+1}}^{s}$
\item If $e_{i+1}= \rev{e}_{i}$, then $\K_{i} = (\HH_{e_{i+1}}^{s})^{\circ}$
\item $\K_{n} = A_{r(e_{n})}$
\end{itemize}
For $0\leqslant i\leqslant n-1$, $\K_{i}$ is a right Hilbert $B_{e_{i+1}}$-module and $\K_{n}$ will be seen as a right Hilbert $A_{r(e_{n})}$-module. We can put compatible left module structures on these Hilbert C*-modules in order to make tensor products. In fact, for $1\leqslant i \leqslant n-1$, the map
\begin{equation*}
\rho_{i} = \pi_{e_{i+1}}^{s}\circ r_{e_{i}} : B_{e_{i}} \rightarrow \LL_{B_{e_{i+1}}}(\K_{i})
\end{equation*}
yields a suitable action of $B_{e_{i}}$ on $\K_{i}$ and left multiplication by $r_{e_{n}}(b)$ for $b\in B_{e_{n}}$ induces a representation
\begin{equation*}
\rho_{n} : B_{e_{n}} \rightarrow \LL_{A_{r(e_{n})}}(\K_{n}).
\end{equation*}
We can now define a right Hilbert $A_{r(e_n)}$-module
\begin{equation*}
\HH_{w}=\K_{0}\underset{\rho_1}{\otimes} \dots \underset{\rho_n}{\otimes} \K_{n}
\end{equation*}
endowed with a faithful left action of $A_{s(e_{1})}$ which is induced by its action on $\K_0$ by left multiplication. This will be called a \emph{path Hilbert module}. Let us describe more precisely the inner product.

\begin{lem}\label{lem:innerproduct}
Let $n\geqslant 1$ and let $w = (e_{1}, \dots, e_{n})$ be a path in $\Gr$. Let $a = \h{a}_{0}\otimes \dots \otimes a_{n}$ and $b = \h{b}_{0}\otimes \dots \otimes b_{n}$ be two elements in $\HH_{w}$. Set $x_{0} = a_{0}^{*}b_{0}$ and, for $1\leqslant k\leqslant n$, set
\begin{equation*}
x_{k} = a_{k}^{*}(r_{e_{k}}\circ s_{e_{k}}^{-1}\circ \EE_{e_{k}}^s(x_{k-1}))b_{k}.
\end{equation*}
Then, $\langle a, b\rangle_{\HH_{w}} = x_{n}\in A_{r(e_{n})}$.
\end{lem}

\begin{proof}
The proof is by induction on $n$. For $n=1$, we have $w=(e)$ where $e\in\E(\Gr)$, $a=\widehat{a}_0\underset{e}{\ot} a_1$ and $b=\widehat{b}_0\underset{e}{\ot} b_1$, where $a_0,b_0\in A_{s(e)}$ and $a_1,b_1\in A_{r(e)}$. By definition of $\mathcal{H}_{w}$ we have
\begin{equation*}
\langle a,b\rangle_{\HH_w}=\langle a_1,\rho_1(s_e^{-1}\circ\EE_e^s(a_0^*b_0))b_1\rangle_{A_{r(e)}}=a_1^*(r_e\circ s_e^{-1}\circ\EE_e^s(a_0^*b_0))b_1=x_1.
\end{equation*}
Assume that the formula holds for a given $n\geq 1$. Let $w=(e_1,\ldots e_{n+1})$ be a path and fix $a = \h{a}_{0}\otimes \dots \otimes a_{n+1}, b = \h{b}_{0}\otimes \dots \otimes b_{n+1}\in\HH_w$. Write
$$\HH_w=\K_{0}\underset{\rho_1}{\otimes} \dots \underset{\rho_{n+1}}{\otimes} \K_{n+1}=\HH_w'\underset{\rho_{n+1}}{\ot} A_{r(e_{n+1})},$$
where $\HH_w'$ is the Hilbert $B_{e_{n+1}}$-module $\HH_w'=\K_{0}\underset{\rho_1}{\otimes} \dots \underset{\rho_n}{\otimes} \K_{n}$. We have,
$$\langle a,b\rangle_{\HH_w}=a_{n+1}^*r_{e_{n+1}}\left(\langle\widehat{a}_0\ot\ldots\ot a_n, \widehat{b}_0\ot\ldots\ot\widehat{b}_n\rangle_{\HH_w'}\right)b_{n+1}.$$
By definition of the inner product, we get, with $w'=(e_1,\ldots,e_n)$,
$$\langle\widehat{a}_0\ot\ldots\ot a_n, \widehat{b}_0\ot\ldots\ot\widehat{b}_n\rangle_{\HH_w'}
=s_{e_{n+1}}^{-1}\circ\EE_{e_{n+1}}^s\left( \langle\widehat{a}_0\ot\ldots\ot a_n, \widehat{b}_0\ot\ldots\ot b_n\rangle_{\HH_{w'}}\right).$$
This concludes the proof using the induction hypothesis.
\end{proof}

\noindent For any two vertices $p_{0}, q\in \V(\Gr)$, we define a right Hilbert $A_q$-module
\begin{equation*}
\mathcal{H}_{p_{0},q}=\bigoplus_{w}\HH_{w}
\end{equation*}
where the sum runs over all paths $w$ in $\Gr$ connecting $p_{0}$ with $q$. By convention, when $q=p_0$, the sum also runs over the empty path, where $\HH_{\emptyset} = A_{p_{0}}$ with its canonical Hilbert $(A_{p_{0}}, A_{p_{0}})$-bimodule structure. We equip this Hilbert C*-module with the faithful left action of $A_{p_{0}}$ which is given by the sum of its left actions on every $\HH_{w}$.

\subsubsection{The C*-algebra}

For every $e\in\E(\Gr)$ and $p\in\V(\Gr)$, we can define an operator
\begin{equation*}
u_{e}^{p} : \HH_{r(e),p} \rightarrow \HH_{s(e),p}
\end{equation*}
which "adds the edge $e$ on the left". To construct this operator, let $w$ be a path in $\Gr$ from $r(e)$ to $p$ and let $\xi\in \HH_w$.
\begin{itemize}
\item If $p = r(e)$ and $w$ is the empty path, then $u_{e}^{p}(\xi) = \xi_{e}^{s}\otimes\xi \in \HH_{(e)}$.
\item If $n = 1$, $w = (e_{1})$, $\xi = \h{a}\otimes\xi'$ with $a\in A_{s(e_{1})}$ and $\xi'\in A_{p}$, then
\begin{itemize}
\item If $e_{1} \neq \rev{e}$, $u_{e}^{p}(\xi) = \xi_{e}^{s}\otimes\xi \in \HH_{(e, e_{1})}$.
\item If $e_{1} = \rev{e}$, $u_{e}^{p}(\xi) =
\left\{\begin{array}{cccc}
\xi_{e}^{s}\otimes\xi & \in \HH_{(e, e_{1})} & \text{if} & \h{a}\in (\HH_{e_{1}}^{s})^{\circ}, \\
r_{e_{1}}\circ s_{e_{1}}^{-1}(a)\xi' & \in A_{p} & \text{if} & a\in B^{s}_{e_{1}}.
\end{array}\right.$
\end{itemize}
\item If $n\geqslant 2$, $w = (e_{1}, \dots, e_{n})$, $\xi = \h{a}\otimes\xi'$ with $a\in A_{s(e_{1})}$ and $\xi' \in \K_{1} \underset{\rho_{2}}{\otimes} \dots \underset{\rho_{n}}{\otimes} \K_{n}$, then
\begin{itemize}
\item If $e_{1} \neq \rev{e}$, $u_{e}^{p}(\xi) = \xi_{e}^{s}\otimes\xi \in \HH_{(e, e_{1}, \dots, e_{n})}$.
\item If $e_{1} = \rev{e}$, $u_{e}^{p}(\xi) =
\left\{\begin{array}{cccc}
\xi_{e}^{s}\otimes\xi & \in \HH_{(e, e_{1}, \dots, e_{n})} & \text{if} & \h{a}\in (\HH_{e}^{r})^{\circ}, \\
r_{e_{1}}\circ s_{e_{1}}^{-1}(a)\xi' & \in \HH_{(e_{2}, \dots, e_{n})} & \text{if} & a \in B_{e_{1}}^{s}.\end{array}\right.$
\end{itemize}
\end{itemize}
One easily checks that the operators $u_{e}^{p}$ commute with the right actions of $A_{p}$ on $\HH_{r(e), p}$ and $\HH_{s(e), p}$ and extend to unitary operators (still denoted $u_{e}^{p}$) in $\mathcal{L}_{A_p}(\HH_{r(e),p},\HH_{s(e),p})$ such that $(u_{e}^{p})^{*} = u_{\rev{e}}^{p}$. Moreover, for every $e\in \E(\Gr)$ and every $b\in B_{e}$, the definition implies that
\begin{equation*}
u_{\rev{e}}^{p}s_{e}(b)u_{e}^{p} = r_{e}(b)
\end{equation*}
as operators in $\LL_{A_{p}}(\HH_{r(e),p})$. Let $w=(e_1,\ldots, e_n)$ be a path in $\Gr$ and let $p\in\V(\Gr)$, we set
\begin{equation*}
u^{p}_{w} = u_{e_{1}}^p \dots u_{e_{n}}^{p} \in \LL_{A_{p}}(\HH_{r(e_{n}),p},\HH_{s(e_{1}),p}).
\end{equation*}
We are now ready to define the reduced fundamental C*-algebra.

\begin{de}
Let $(\Gr, (A_{q})_{q}, (B_{e})_{e})$ be a graph of C*-algebras and let $p, p_{0}\in \V(\Gr)$. The \emph{reduced fundamental C*-algebra rooted in $p_{0}$ with base $p$} is the C*-algebra
$$
\pi^{p}_{1}(\Gr, (A_{q})_{q}, (B_{e})_{e}, p_{0}) =  \left\langle (u^{p}_{z})^{*}A_{q}u^{p}_{w} \vert q \in \V(\Gr),w,z \text{ paths from $q$ to $p_0$ }\right\rangle
\subset \LL_{A_{p}}(\HH_{p_{0},p}).
$$
If the root $p_0$ is equal to the base $p$, we will simply call it the \emph{reduced fundamental C*-algebra in $p_0$}. We will use the shorthand notation $P_{p}(p_{0})$ (and $P(p_{0})$ when $p = p_{0}$) to denote the reduced fundamental C*-algebra in the sequel.
\end{de}

\begin{rem}
The above definition may seem unsatisfying because of the two arbitrary vertices involved. However, this will give many natural representations of the reduced C*-algebra which will be needed later on. This also gives a more tractable object when it turns to making products or computing norms.
\end{rem}

\begin{rem}\label{rem:basepoint}
Because the graph $\Gr$ is connected, the previous construction does not really depend on $p_{0}$. In fact, let $p, p_{0}, q_{0}$ be three vertices of $\Gr$ and let $\T$ be a maximal subtree in $\Gr$. If $g_{q_{0} p_{0}}$ denotes the unique geodesic path in $\T$ from $q_{0}$ to $p_{0}$, then we have an isomorphism
\begin{equation*}
\Phi^{p}_{\T, p_{0}, q_{0}} : P_{p}(p_{0}) \rightarrow P_{p}(q_{0})
\end{equation*}
which is given by
\begin{equation*}
x \mapsto u^{p}_{g_{q_{0}p_{0}}}x(u^{p}_{g_{q_{0}p_{0}}})^{*}.
\end{equation*}
Note, however, that there is no truly canonical way to identify these C*-algebras.
\end{rem}

\begin{ex}
Carrying out this construction with the graphs of Examples \ref{ex:freeproduct} and \ref{ex:hnn}, one recovers respectively the reduced amalgamated free product construction of \cite{voiculescu85symmetries} and the reduced HNN extension construction of \cite[Sec 7.4]{ueda2005hnn}.
\end{ex}

\subsubsection{The quotient map}\label{quotient}

We now investigate the link beteween the reduced fundamental C*-algebra and the maximal one. From now on, we fix two vertices $p_{0}, p\in \V(\Gr)$ and consider the C*-algebra $P_{p}(p_{0})$. Let $\mathcal{T}$ be a maximal subtree in $\Gr$. As before, given a vertex $q\in \V(\Gr)$, we denote by $g_{qp}$ the unique geodesic path in $\mathcal{T}$ from $q$ to $p$. For every $e\in \E(\Gr)$, we define a unitary operator $w^{p}_{e}\in P_p(p_0)$ by
\begin{equation*}
w^{p}_{e} = (u^{p}_{g_{s(e)p}})^{*}u^{p}_{(e, g_{r(e)p})}.
\end{equation*}
\noindent For every $q\in\V(\Gr)$, we define a unital faithful $*$-homomorphism $\pi^{p}_{q, p_{0}}: A_{q} \rightarrow P_{p}(p_{0})$ by
\begin{equation*}
\pi^{p}_{q, p_{0}}(a) = (u^{p}_{g_{qp_{0}}})^{*}a u^{p}_{g_{qp_{0}}}
\end{equation*}
Observe that the following relations hold:
\begin{itemize}
\item $w^{p}_{\rev{e}} = (w^{p}_{e})^{*}$ for every $e\in\E(\Gr)$,
\item $w^{p}_{e} = 1$ for every $e\in \E(\T)$,
\item $w_{\rev{e}}\pi^{p}_{s(e), p_{0}}(s_e(b))w_{e} = \pi^{p}_{r(e), p_{0}}(r_e(b))$ for every $e\in \E(\Gr)$, $b\in B_{e}$.
\end{itemize}
\noindent The first and the last relations are clear from the definitions. To check the second one, observe that if $e\in \E(\T)$, then the path $(\rev{g}_{s(e)p}, e, g_{r(e)p})$ is a cycle in $\mathcal{T}$. This means that either $g_{r(e)p} = (\rev{e}, g_{s(e)p})$ or $g_{s(e)p} = (e,g_{r(e)p})$. In both cases we get $w^{p}_{e} = 1$.

\vspace{0.2cm}

\noindent Thus, we can apply the universal property of Proposition \ref{prop:universal} to get a surjective $*$-homomorphism
\begin{equation*}
\lambda_{p, p_{0}}^{\T}: \pi_{1}^{\text{max}}(\Gr, (A_{q})_{q}, (B_{e})_{e}, \T) \rightarrow P_{p}(p_{0}).
\end{equation*}

\subsubsection{Reduced operators}

\noindent Like in the case of groups, we have a notion of reduced element.

\begin{de}
Let $(\Gr, (A_{q})_{q}, (B_{e})_{e})$ be a graph of C*-algebras and let $p_{0},p,q\in \V(\Gr)$. Let $a\in\mathcal{L}_{A_p}(\mathcal{H}_{q,p},\mathcal{H}_{p_0,p})$ be of the form $a = a_{0}u^{p}_{e_{1}} \dots u^{p}_{e_{n}}a_{n}$, where $w = (e_{1}, \dots,  e_{n})$ is a path in $\Gr$ from $p_{0}$ to $q$, $a_{0}\in A_{p_{0}}$ and, for $1\leqslant i\leqslant n$, $a_{i}\in A_{r(e_{i})}$. The operator $a$ is said to be \emph{reduced} (from $p_0$ to $q$) if for all $1\leqslant i\leqslant n-1$ such that $e_{i+1} = \rev{e}_{i}$, we have $\EE_{e_{i+1}}^{s}(a_{i}) = 0$.
\end{de}

\begin{rem}
Let $w=(e_{1}, \dots,  e_{n})$ be a path from $p_0$ to $p_0$. Observe that any reduced operator of the form $a = a_{0}u^{p}_{e_{1}} \dots u^{p}_{e_{n}}a_{n}$ is in $P_p(p_0)$ and that the linear span $\mathcal{R}_p(p_0)$ of $A_{p_0}$ and the reduced operators from $p_0$ to $p_0$ is a dense $*$-subalgebra of $P_p(p_0)$. Indeed, we can write $a=x_0x_1\ldots x_n$ where $x_0=a_0u^p_{(e_1,\overline{e}_1)}\in P_p(p_0)$, $x_n=u^p_{w}a_n\in P_p(p_0)$ and, for $1\leq i\leq n-1$, $x_i=u^p_{(e_1,\ldots,e_i)}a_iu^p_{(\overline{e}_i,\ldots,\overline{e}_1)}\in P_p(p_0)$. This shows that $a\in P_p(p_0)$. Now, using the relations $u^{p}_{\rev{e}}s_{e}(b)u^{p}_{e} = r_{e}(b)$ for $e\in\E(\Gr)$ and $b\in B_e$, we see that $\R_{p}(p_{0})$ is a $*$-subalgebra of $P_{p}(p_{0})$. To show that $\R_{p}(p_{0})$ is dense it suffices to show that it contains all operators of the form $(u^{p}_{z})^{*}au^{p}_{w}$, where $a\in A_q$ and $z,w$ are paths from $q$ to $p_0$. One can easily check this by induction and using the relations  $u^{p}_{\rev{e}}s_{e}(b)u^{p}_{e} = r_{e}(b)$.
\end{rem}

\begin{rem}\label{rem:redmax}
The notion of reduced operator also makes sense in the maximal fundamental C*-algebra (if we assume the existence of conditional expectations) and the linear span of $A_{p_{0}}$ and the reduced operators from $p_0$ to $p_0$ is the $*$-algebra $\mathcal{A}$ introduced in Remark \ref{rem:pathmax}, which is dense in the maximal fundamental C*-algebra.
\end{rem}

\noindent In order to simplify later computations, we now give an explicit formula for the action of a reduced operator on the Hilbert C*-module $\HH_{p_{0}, p}$ which will be used several times in the sequel. For every edge $e\in \E(\Gr)$ and every $x\in A_{r(e)}$, we set $\PP_{e}^{r}(x) = x - \EE_{e}^{r}(x)$. For the sake of simplicity, if $w = (f_{1}, \dots, f_{n})$ is a path in $\Gr$, we will use the notation $\widehat{b}_{0} \underset{f_{1}}{\otimes} \dots \underset{f_{m}}{\otimes}b_{m}$ to denote a typical element in $\HH_{w}$.

\begin{lem}\label{lem:product}
Let $(\Gr, (A_{q})_{q}, (B_{e})_{e})$ be a graph of C*-algebras and let $p_{0}, p_{1}, p_{2} \in \V(\Gr)$. Let $w=(e_{n}, \dots, e_{1})$ be a path from $p_{0}$ to $p_{1}$ and let $\mu = (f_{1}, \dots, f_{m})$ be path from $p_{1}$ to $p_{2}$. Set
\begin{equation*}
n_{0} = \max\{1\leqslant k\leqslant\min(n, m)\vert e_{i} = \rev{f}_{i}, \forall i\leqslant k\}.
\end{equation*}
If the above set is empty, set $n_{0} = 0$. Let $a = a_{n} u^{p_2}_{e_{n}} \dots u^{p_2}_{e_{1}}a_{0}$ be a reduced operator and let $b = \widehat{b}_{0} \underset{f_{1}}{\otimes} \dots \underset{f_{m}}{\otimes}b_{m}\in \HH_{p_{1}, p_{2}}$. Set $x_{0} = a_{0}b_{0}$ and, for $1\leqslant k\leqslant n_{0}$, set
$$
x_{k} = a_{k}(s_{e_{k}}\circ r_{e_{k}}^{-1}\circ \EE_{e_{k}}^r(x_{k-1}))b_{k}\quad\text{and}\quad y_{k} =\PP_{e_k}^{r}(x_{k-1}).$$
Then, the following holds :
\begin{enumerate}

\item If $n_{0} = 0$, then $a.b = \h{a}_{n}\underset{e_{n}}{\otimes} \dots \underset{e_{1}}{\otimes}\h{x}_{0}\underset{f_{1}}{\otimes} \dots \underset{f_{m}}{\otimes}b_{m}$.

\item If $n_0=n=m$, then $a.b=\sum_{k=1}^n\h{a}_{n}\underset{e_{n}}{\otimes} \dots \underset{e_{k}}{\otimes}\widehat{y}_{k}\underset{f_{k}}{\otimes} \dots \underset{f_{n}}{\otimes}b_{n}+x_n.$

\item If $n_{0} = n<m$, then $a.b = \sum_{k = 1}^{n}\h{a}_{n}\underset{e_{n}}{\otimes} \dots \underset{e_{k}}{\otimes}\widehat{y}_{k}\underset{f_{k}}{\otimes} \dots \underset{f_{m}}{\otimes}b_{m} + \h{x}_{n}\underset{f_{n+1}}{\otimes} \dots \underset{f_{m}}{\otimes}b_{m}$.

\item If $n_0=m<n$, then $a.b = \sum_{k = 1}^{m}\h{a}_{n}\underset{e_{n}}{\otimes} \dots \underset{e_{k}}{\otimes}\widehat{y}_{k}\underset{f_{k}}{\otimes} \dots \underset{f_{m}}{\otimes}b_{m} + \h{a}_{n} \underset{e_{n}}{\otimes} \dots \underset{e_{m+1}}{\otimes} x_{m}$.

\item If $1\leqslant n_0<\min\{n,m\}$, then
$$a.b = \sum_{k = 1}^{n_{0}}\h{a}_{n}\underset{e_{n}}{\otimes} \dots \underset{e_{k}}{\otimes}\widehat{y}_{k}\underset{f_{k}}{\otimes} \dots \underset{f_{m}}{\otimes}b_{m} + \h{a}_{n} \underset{e_{n}}{\otimes} \dots \underset{e_{n_{0}+1}}{\otimes} \h{x}_{n_{0}}\underset{f_{n_{0}+1}}{\otimes} \dots \underset{f_{m}}{\otimes}b_{m}.$$
\end{enumerate}
\end{lem}

\begin{proof}
To simplify the notations during the proof, we omit the superscript $p_2$.

\vspace{0.2cm}

\noindent $(1)$. If $n_0=0$ then $e_1\neq\overline{f}_1$ and we get, by definition of the $u_e$ and because $a$ is reduced, $a.b=a_nu_{e_n}\dots u_{e_1}.\widehat{x}_0\underset{f_1}{\ot}\dots\underset{f_m}{\ot} b_m
=\widehat{a}_n\underset{e_n}{\ot}\dots\underset{e_1}{\ot}\widehat{x}_0\underset{f_1}{\ot}\dots\underset{f_m}{\ot} b_m$. This proves $(1)$. Since the proof of the other cases are all the same we only prove $(5)$.

\vspace{0.2cm}

\noindent $(5)$. We need to show the following statement: for every $n_0\geqslant 1$, for every $n,m>n_0$, for every reduced operator $a=a_nu_{e_n}\dots u_{e_1}a_0$ from $p_0$ to $p_1$ and for  $b\in\mathcal{H}_{p_1,p_2}$ of the form $b=\widehat{b}_0\underset{\overline{e}_1}{\ot}\dots\underset{\overline{e}_{n_0}}{\ot}\widehat{b}_{n_0}\underset{f_{n_0+1}}{\ot}\dots\underset{f_m}{\ot}b_m$ with $f_{n_0+1}\neq \overline{e}_{n_0+1}$ we have
$$a.b = \sum_{k = 1}^{n_{0}}\h{a}_{n}\underset{e_{n}}{\otimes} \dots \underset{e_{k}}{\otimes}\widehat{y}_{k}\underset{\overline{e}_{k}}{\otimes} \dots\underset{\overline{e}_{n_0}}{\ot}\widehat{b}_{n_0}\underset{f_{n_0+1}}{\ot}\dots \underset{f_{m}}{\otimes}b_{m} + \h{a}_{n} \underset{e_{n}}{\otimes} \dots \underset{e_{n_{0}+1}}{\otimes} \h{x}_{n_{0}}\underset{f_{n_{0}+1}}{\otimes} \dots \underset{f_{m}}{\otimes}b_{m}.$$
We prove it by induction on $n_0$. If $n_0=1$, let $n,m>1$, let $a=a_nu_{e_n}\dots u_{e_1}a_0$ and let $b=\widehat{b}_0\underset{\overline{e}_1}{\ot}\widehat{b}_1\underset{f_2}{\ot}\dots\underset{f_m}{\ot}b_m$ with $f_{2}\neq \overline{e}_{2}$. Since $a$ is reduced and $f_2\neq\overline{e}_2$, we have
\begin{eqnarray*}
a.b&=&a_{n}u_{e_{n}}\dots a_1u_{e_1}.\widehat{x}_0\underset{\overline{e}_1}{\ot}\widehat{b}_1\underset{f_2}{\ot}\dots\underset{f_{m}}{\ot} b_{m}\\
&=&a_{n}u_{e_{n}}\dots a_1u_{e_1}.\widehat{\EE_{e_1}^r(x_0)}\underset{\overline{e}_1}{\ot}\widehat{b}_1\underset{f_2}{\ot}\dots\underset{f_{m}}{\ot} b_{m}
+a_{n}u_{e_{n}}\dots a_1u_{e_1}.\widehat{y}_1\underset{\overline{e}_1}{\ot}\widehat{b}_1\underset{f_2}{\ot}\dots\underset{f_{m}}{\ot} b_{m}\\
&=&a_{n}u_{e_{n}}\dots u_{e_2}.\widehat{x}_1\underset{f_2}{\ot}\dots\underset{f_{m}}{\ot} b_{m}+\widehat{a}_{n}\underset{e_{n}}{\ot}\dots\underset{e_{1}}{\ot} \widehat{y}_1\underset{\overline{e}_1}{\ot} \widehat{b}_1\underset{f_2}{\ot}\dots\underset{f_{m}}{\ot} b_{m}\\
&=&\h{a}_{n}\underset{e_{n}}{\otimes} \dots\underset{e_{2}}{\otimes}\widehat{x}_1\underset{f_2}{\ot}\dots\underset{f_{m}}{\ot} b_{m}+\widehat{a}_{n}\underset{e_{n}}{\ot}\dots\underset{e_{1}}{\ot} \widehat{y}_1\underset{\overline{e}_1}{\ot} \widehat{b}_1\underset{f_2}{\ot}\dots\underset{f_{m}}{\ot} b_{m},
\end{eqnarray*}
Assume that the statement holds for a given $n_0\geqslant 1$. Let $n,m> n_0+1$, let $a=a_nu_{e_n}\dots u_{e_1}a_0$ and let $b=\widehat{b}_0\underset{\overline{e}_1}{\ot}\dots\underset{\overline{e}_{n_0+1}}{\ot}\widehat{b}_{n_0+1}\underset{f_{n_0+2}}{\ot}\dots\underset{f_m}{\ot}b_m$ with $f_{n_0+2}\neq \overline{e}_{n_0+2}$. We have
\begin{eqnarray*}
a.b&=&a_{n}u_{e_{n}}\dots u_{e_1}.\widehat{x}_0\underset{\overline{e}_1}{\ot}\dots\underset{\overline{e}_{n_0+1}}{\ot}\widehat{b}_{n_0+1}\underset{f_{n_0+2}}{\ot}\dots\underset{f_{m}}{\ot} b_{m}\\
&=&a_{n}u_{e_{n}}\dots u_{e_2}.\widehat{x}_1\underset{\overline{e}_2}{\ot}\dots\underset{\overline{e}_{n_0+1}}{\ot}\widehat{b}_{n_0+1}\underset{f_{n_0+2}}{\ot}\dots\underset{f_{m}}{\ot} b_{m}\\
&+&
\widehat{a}_{n}\underset{e_{n}}{\ot}\dots\underset{e_{1}}{\ot} \widehat{y}_1\underset{\overline{e}_1}{\ot}\dots\underset{\overline{e}_{n_0+1}}{\ot}\widehat{b}_{n_0+1}\underset{f_{n_0+2}}{\ot}\dots\underset{f_{m}}{\ot} b_{m}\\
&=&a'.b'+\widehat{a}_{n}\underset{e_{n}}{\ot}\dots\underset{e_{1}}{\ot} \widehat{y}_1\underset{\overline{e}_1}{\ot}\dots\underset{\overline{e}_{n_0+1}}{\ot}\widehat{b}_{n_0+1}\underset{f_{n_0+2}}{\ot}\dots\underset{f_{m}}{\ot} b_{m},
\end{eqnarray*}

\noindent where $a'=a_nu_{e_n}\dots u_{e_2}$ and $b'=\widehat{x}_1\underset{\overline{e}_2}{\ot}\dots\underset{\overline{e}_{n_0+1}}{\ot}\widehat{b}_{n_0+1}\underset{f_{n_0+2}}{\ot}\dots\underset{f_{m}}{\ot} b_{m}$. We can now apply the induction hypothesis to the pair $(a',b')$. This concludes the proof.
\end{proof}

\noindent If a maximal subtree $\T$ of $\Gr$ is fixed and if $p, q\in \V(\Gr)$, then $g_{pq}$ will denote the unique geodesic path in $\T$ from $p$ to $q$. Viewing $1_{A_{p}}$, the unit of $A_p$ as an element of $\HH_{\emptyset} \subset \HH_{p, p}$, we set
\begin{equation*}
\Omega_{p}(p_{0}) = u_{g_{p_{0}p}}^{p}(1_{A_{p}}) \in \HH_{g_{p_{0}p}} \subset \HH_{p_{0},p}.
\end{equation*}
The vector $\Omega_{p}(p_{0})$ is a cyclic vector for $P_{p}(p_{0})$.

\begin{prop}\label{prop:PquasiGNS}
With the previous notations, we have
\begin{enumerate}
\item $\rev{P_{p}(p_{0})\Omega_{p}(p_{0})}=\HH_{p_{0},p}$
\item For any reduced operator $a\in P_p(p_0)$, $\langle\Omega_{p}(p_{0}), a.\Omega_{p}(p_{0})\rangle_{\HH_{p_{0},p}} = 0$.
\end{enumerate}
\end{prop}

\begin{proof}
$(1)$. Let $w = (e_{1}, \dots, e_{n})$ be a path from $p_{0}$ to $p$ and let $x = x_{0} u_{e_{1}}^{p} \dots u^{p}_{e_{n}}x_{n}$ be a reduced operator. The operator $a = x(u^{p}_{g_{p_{0}p}})^{*}$ is in $P_{p}(p_{0})$ and, since $x$ is reduced,
\begin{equation*}
a.\Omega_{p}(p_{0}) = x(u^{p}_{g_{p_{0}p}})^{*}u^{p}_{g_{p_{0}p}}.1_{A_{p}} = x.1_{A_{p}} = \h{x}_{0}\underset{e_{1}}{\otimes} \dots \underset{e_{n}}{\otimes}x_{n} \in \HH_{w}.
\end{equation*}

\noindent $(2)$. Let $w = (e_{n}, \dots, e_{1})$ be a path from $p_{0}$ to $p_{0}$ and let $a = a_{n} u^{p}_{e_{n}} \dots u^{p}_{e_{1}}a_{0}$ be a reduced operator in $P_{p}(p_{0})$. Write $g_{p_{0}p} = (f_{1}, \dots, f_{m})$ and $b=\Omega_p(p_0)=\widehat{1}\underset{f_1}{\ot}\ldots\underset{f_m}{\ot} 1$. We use the notations of Lemma \ref{lem:product}. Note that if $n_{0} = n$, then we must have $g_{p_{0}p} = (\rev{e}_{1}, \dots, \rev{e}_{n}, f_{n+1}, \dots, f_{m})$, which is impossible since $g_{p_{0}p}$ is a geodesic path and $(\rev{e}_{1}, \dots, \rev{e}_{n})$ is a cycle. Thus $n_{0} < n$. If $n_{0} = 0$, then
\begin{equation*}
a.\Omega_{p_{0}, p} = \h{a}_{n}\underset{e_{n}}{\otimes} \dots \underset{e_{1}}{\otimes}\h{a}_{0}\underset{f_{1}}{\otimes}\h{1} \dots\underset{f_{m}}{\otimes} 1 \in \HH_{(e_n,\ldots, e_1,f_1,\ldots,f_m)},
\end{equation*}
hence $\langle\Omega_{p}(p_{0}), a.\Omega_{p}(p_{0})\rangle_{\HH_{p_{0}, p}} = 0$. Assume now that $1\leqslant n_{0} < n$ and observe that, except if $w = (f_{1}, \dots, f_{m}, \rev{f}_{m}, \dots, \rev{f}_{1})$, the elements appearing in the formula of Lemma \ref{lem:product} all belong to subspaces which are orthogonal to $\HH_{g_{p_{0}p}}$ and consequently orthogonal to $\Omega_{p}(p_{0})$. Finally, if $w = (f_{1}, \dots, f_{m}, \rev{f}_{m}, \dots, \rev{f}_{1})$, then
\begin{equation*}
\langle \Omega_{p}(p_{0}), a.\Omega_{p}(p_{0}) \rangle_{\HH_{p_{0}, p}} = \left\langle\h{1}\underset{f_{1}}{\otimes}\h{1} \dots \underset{f_{m}}{\otimes} 1,\h{a}_{n}\underset{f_{1}}{\otimes} \dots \underset{f_{m}}{\otimes} x_{m}\right\rangle_{\HH_{p_{0}, p}} = 0.
\end{equation*}
since $x_{m}=a_mz$ with $z=s_{e_m}\circ r_{e_m}^{-1}\circ\EE_{e_m}^r(x_{m-1})\in B_{e_m}^s=B_{f_{m}}^{r}$ (here $e_{m} = \rev{f}_{m}$) and $\EE^{r}_{f_{m}}(a_{m}) = 0$ because $a$ is reduced.
\end{proof}

\begin{rem}
The first assertion of Proposition  \ref{prop:PquasiGNS} shows that the triple $(\mathcal{H}_{p_0,p},\text{id},\Omega_p(p_0))$ is the GNS construction of the unital completely positive map $\EE_{A_{p}}\,:\,P_p(p_0)\rightarrow A_p$ defined, for every $x\in P_p(p_0)$, by
$$\EE_{A_{p}}(x) = \langle \Omega_p(p_{0}), x.\Omega_p(p_{0})\rangle_{\HH_{p_{0}, p}}.$$
In the particular case $p = p_{0}$, $A_{p_0}\subset P(p_0)$ and $\EE_{A_{p_{0}}}$ is a conditional expectation from $P(p_0)$ to $A_{p_0}$.
\end{rem}

\subsubsection{A universal property for the reduced C*-algebra}

We end this section with a universal property for the reduced fundamental C*-algebra in the spirit of Corollary \ref{cor:universal}. Consider a graph of C*-algebras $(\Gr, (A_{p})_{p}, (B_{e})_{e})$ and fix $p_{0}\in\V(\Gr)$. The basic data to build a representation is:
\begin{itemize}
\item For every $p\in\V(\Gr)$, a right Hilbert $A_{p_{0}}$-module $\K_{p}$ with a faithful unital $*$-homomorphism $\pi_{p}: A_{p}\rightarrow\mathcal{L}_{A_{p_{0}}}(\K_{p})$.
\item For every $e\in\E(\Gr)$ a unitary $w_{e}\in\mathcal{L}_{A_{p_{0}}}(\K_{r(e)}, \K_{s(e)})$ such that $w_{e}^{*} = w_{\rev{e}}$ and, for every $b\in B_{e}$,
\begin{equation*}
w_{\rev{e}}\pi_{s(e)}(s_{e}(b))w_{e} = \pi_{r(e)}(r_{e}(b)).
\end{equation*}
\end{itemize}
Let $A$ be the closed linear span of $\pi_{p_{0}}(A_{p_{0}})$ and all elements of the form
\begin{equation*}
\pi_{s(e_{1})}(a_{0})w_{e_{1}} \dots w_{e_{n}}\pi_{r(e_{n})}(a_{n})
\end{equation*}
in $\mathcal{L}_{A_{p_{0}}}(\K_{p_{0}})$, where $n\geqslant 1$, $(e_{1}, \dots, e_{n})$ is a path in $\Gr$ from $p_{0}$ to $p_{0}$, $a_{k}\in A_{r(e_{k})}$, $1 \leqslant k \leqslant n$ and $a_{0}\in A_{p_{0}}$. Then, $A$ is a C*-algebra. The universal property requires the following crucial assumption: we assume that there exists a GNS-faithful conditional expectation
\begin{equation*}
\EE: A\rightarrow \pi_{p_{0}}(A_{p_{0}})
\end{equation*}
such that, for every reduced operator $a_{0}u^{p_0}_{e_{1}} \dots u^{p_0}_{e_{n}}a_{n}\in P(p_0)$,
\begin{equation*}
\EE(\pi_{s(e_{1})}(a_{0})w_{e_{1}} \dots w_{e_{n}}\pi_{r(e_{n})}(a_{n})) = 0.
\end{equation*}

\begin{prop}\label{prop:universalreduced}
With the hypothesis and notations above, there exists a unique $*$-isomorphism
$\pi: P(p_0)\rightarrow A$ such that $\pi(a)= \pi_{p_{0}}(a)$ for every $a\in A_{p_0}$ and
\begin{equation*}
\pi(a_{0}u^{p_0}_{e_{1}} \dots u^{p_0}_{e_{n}} a_{n}) = \pi_{s(e_{1})}(a_{0})w_{e_{1}} \dots w_{e_{n}}\pi_{r(e_{n})}(a_{n})\end{equation*}
for every reduced operator $a_{0}u^{p_0}_{e_{1}} \dots u^{p_0}_{e_{n}} a_{n}\in P(p_0)$.
\end{prop}

\begin{proof}
The uniqueness being obvious, let us prove the existence. The map $\pi$ is well defined on the linear span of $\pi_{p_{0}}(A_{p_{0}})$ and the reduced elements, and the closure of the image of this space is equal to $A$. Let $(K', \pi', \eta')$ be the GNS construction of $\pi_{p_{0}}^{-1}\circ \EE$. Since $\EE$ is GNS-faithful, $\pi'$ is faithful and we will assume that $A\subset \mathcal{L}_{A_{p_{0}}}(K')$ and $\pi' = \Id$. View $1_{A_{p_0}}\in\mathcal{H}_{p_{0}, p_{0}}$ and define an operator $V: \mathcal{H}_{p_{0}, p_{0}} \rightarrow K'$ by $V(a.1_{A_{p_0}}) = \pi_{p_{0}}(a).\eta'$ for $a\in A_{p_{0}}$ and
\begin{equation*}
V(x.1_{A_{p_0}}) = \pi_{s(e_{1})}(a_{0})w_{e_{1}} \dots w_{e_{n}}\pi_{r(e_{n})}(a_{n}).\eta'
\end{equation*}
for $x = a_{0}u^{p_0}_{e_{1}} \dots u^{p_0}_{e_{n}}a_{n}\in P(p_0)$ a reduced operator. It is easy to check that $V$ extends to a unitary in $\mathcal{L}_{A_{p_{0}}}(\mathcal{H}_{p_{0}, p_{0}}, K')$ and that $x\mapsto VxV^{*}$ is a $*$-isomorphism extending $\pi$.
\end{proof}

\subsubsection{States}

In the sequel, we will consider C*-algebras equiped with distinguished states compatible with the graph structure. These enable us to get genuine Hilbert space representations instead of Hilbert C*-modules.

\begin{de}
Let $(\Gr, (A_{q})_{q}, (B_{e})_{e})$ be a graph of C*-algebras with conditional expectations. An \emph{associated graph of states} is a family of states $\varphi_{e}\in B_{e}^*$ for every $e\in \E(\Gr)$ and $\varphi_q\in A_q^*$ for every $q\in\V(\Gr)$ such that, for every $e\in \E(\Gr)$, $\varphi_{\rev{e}} = \varphi_{e}$ and $\varphi_{s(e)} = \varphi_{e}\circ s_{e}^{-1}\circ \EE_{e}^{s}$ (hence $\varphi_{s(e)}\circ s_{e}=\varphi_e=\varphi_{r(e)}\circ r_{e}$). When a graph of C*-algebras is given with a graph of states we simply call $(\Gr, (A_{q},\varphi_q)_{q}, (B_{e},\varphi_e)_{e})$ a \emph{graph of C*-algebras with states}.
\end{de}


\begin{lem}\label{lem:states}
Let $(\Gr, (A_{q},\varphi_q)_{q}, (B_{e},\varphi_e)_{e})$ be a graph of C*-algebras with states. For every $p_{0}, p \in \V(\Gr)$ and every $a\in A_{p_{0}}$, we have $\varphi_{p}\left(\langle\Omega_{p}(p_{0}), a.\Omega_{p}(p_{0})\rangle_{\HH_{p_{0},p}}\right) = \varphi_{p_{0}}(a)$.
\end{lem}

\begin{proof}
Let $a\in A_{p_{0}}$ and set $x_{0}=a$ and, for $1\leqslant k\leqslant n$, $x_{k} = r_{e_{k}}\circ s_{e_{k}}^{-1}\circ \EE_{e_{k}}^{s}(x_{k-1})$, where $g_{p_{0}, p}=(e_{1}, \dots, e_{n})$. By Lemma \ref{lem:innerproduct} we have
$$\langle\Omega_{p}(p_{0}), a.\Omega_{p}(p_{0})\rangle_{\HH_{p_{0},p}}=\langle\widehat{1}\underset{e_1}{\ot}\ldots\underset{e_n}{\ot} 1,\widehat{a}\underset{e_1}{\ot}\ldots\underset{e_n}{\ot} 1\rangle_{\HH_{p_{0},p}}=x_n.$$

\noindent Moreover, the assumptions on the states $\varphi_{e}$ and $\varphi_{s(e)}$ imply that for $1\leqslant k\leqslant n$, $\varphi_{r(e_{k})}(x_{k})=\varphi_{s(e_{k})}(x_{k-1})$. It then follows, again by induction, that $\varphi_{p}(x_{n}) = \varphi_{p_{0}}(x_{0}) = \varphi_{p_{0}}(a)$, which concludes the proof.
\end{proof}

\begin{rem}\label{rem:states}
Proposition \ref{prop:PquasiGNS} and Lemma \ref{lem:states} imply the existence of a unique state $\varphi$ on $P_{p}(p_{0})$ such that for every $a\in A_{p_{0}}$, $\varphi(a) = \varphi_{p_{0}}(a)$ and for every reduced operator $c$, $\varphi(c) = 0$. We call this state the \emph{fundamental state}.
\end{rem}

\noindent Using states, we can investigate the dependance on $p$ of $P_{p}(p_{0})$. To do this, let us denote by $(H_{p_{0},\varphi_p},\pi_{p_{0}, \varphi_p},\eta_{p_{0},\varphi_p})$ the GNS construction of $(\HH_{p_{0},p},\varphi_{p})$, i.e. $H_{p_{0}, \varphi_p}$ is the completion of $\HH_{p_{0}, p}$ with respect to the inner product $\langle x, y\rangle = \varphi_{p}(\langle x, y\rangle_{\HH_{p_{0}, p}})$, $\pi_{p_{0}, \varphi_p}$ is the associated representation of $P_{p}(p_{0})$ and $\eta_{p_{0}, \varphi_p}$ is the canonical linear map with dense range, and set $\xi_{p_{0},\varphi_p} = \eta_{p_{0}, \varphi_p}(\Omega_{p}(p_{0}))$.

\begin{prop}\label{prop:base}
Let $(\Gr, (A_{q},\varphi_q)_{q}, (B_{e},\varphi_e)_{e})$ be a graph of C*-algebras with states. For any two vertices $p_{0}, p\in \V(\Gr)$, there is a unitary $V_{p_0,p}: H_{p_{0},\varphi_p} \rightarrow H_{p_{0},\varphi_{p_{0}}}$ such that, for every reduced operator $a_{0} u_{e_{1}}^{p} \dots e_{e_{n}}^{p} a_{n} \in P_{p}(p_{0})$,
\begin{equation*}
(V_{p_{0},p})\pi_{p_{0},\varphi_p}(a_{0} u_{e_{1}}^{p} \dots u_{e_{n}}^{p} a_{n})(V_{p_{0}, p})^{*} = \pi_{p_{0}, \varphi_{p_{0}}}(a_{0} u_{e_{1}}^{p_{0}} \dots u_{e_{n}}^{p_{0}} a_{n}).
\end{equation*}
If moreover the states $\varphi_{p_{0}}$ and $\varphi_{p}$ are faithful, then there exists a unique $*$-isomorphism from $P(p_{0})$ to $P_{p}(p_{0})$ mapping $a_{0}u_{e_{1}}^{p_{0}} \dots u_{e_{n}}^{p_{0}} a_{n}$ to $a_{0} u_{e_{1}}^{p} \dots u_{e_{n}}^{p} a_{n}$.
\end{prop}

\begin{proof}
We define $V_{p_{0}, p}: H_{p_{0}, \varphi_p} \rightarrow H_{p_{0}, \varphi_{p_{0}}}$ in the following way. For $a\in A_{p_{0}}$, we set
\begin{equation*}
V_{p_{0},p}(\pi_{p_{0},\varphi_p}(a).\xi_{p_{0},\varphi_p}) = \pi_{p_{0},\varphi_{p_{0}}}(a).\xi_{p_{0},\varphi_{p_{0}}}
\end{equation*}
and for $a_{0} u_{e_{1}}^{p} \dots u_{e_{n}}^{p} a_{n}\in P_{p}(p_{0})$ reduced, we set
\begin{equation*}
V_{p_{0}, p}(\pi_{p_{0}, \varphi_p}(a_{0} u_{e_{1}}^{p} \dots u_{e_{n}}^{p} a_{n}).\xi_{p_{0}, \varphi_p}) = \pi_{p_{0}, \varphi_{p_{0}}}(a_{0} u_{e_{1}}^{p_{0}} \dots u_{e_{n}}^{p_{0}} a_{n}) .\xi_{p_{0}, \varphi_{p_{0}}}.
\end{equation*}
According to Proposition \ref{prop:PquasiGNS}, $V_{p_{0}, p}$ has dense domain and range. Moreover, by Lemma \ref{lem:states}, we have, for every $a\in A_{p_{0}}$,
\begin{equation*}
\|\pi_{p_{0},\varphi_{p_{0}}}(a).\xi_{p_{0},\varphi_{p_{0}}}\|^{2} = \|\pi_{p_{0},\varphi_p}(a).\xi_{p_{0},\varphi_p}\|^{2}.
\end{equation*}
To prove that $V_{p_{0},p}$ is well defined and extends to a unitary operator, we only have to check that for every reduced operator $a = a_{0} u_{e_{1}}^{p_{0}} \dots u_{e_{n}}^{p_{0}} a_{n} \in P(p_{0})$ one has, 
$$\varphi_{p_{0}}(\langle 1_{A_{p_{0}}}, a^{*}a.1_{A_{p_{0}}}\rangle) = \varphi_{p}(\langle \Omega_{p}(p_{0}), b^{*}b.\Omega_{p}(p_{0})\rangle)\quad\text{with}\quad b = a_{0} u_{e_{1}}^{p} \dots u_{e_{n}}^{p} a_{n}.$$
\noindent Set $x_0=a_0^*a_0$ and, for $1\leqslant k\leqslant n$, $x_k=a_k^*(r_{e_k}\circ s_{e_k}\circ\EE_{e_k}^s(x_{k-1}))a_k$. By Lemma \ref{lem:innerproduct}, we have $\varphi_{p_{0}}(\langle \xi_{p_{0}}, a^{*}a.\xi_{p_{0}}\rangle) =\varphi_{p_0}(x_n)$. Set $y_k=\mathcal{P}^s_{e_k}(x_{k-1})$. We have, by induction,
$$b^*b=a_n^*u^p_{\overline{e}_n}\ldots u^p_{\overline{e}_1}a_0^*a_0u^p_{e_1}\ldots u^p_{e_n}a_n
=\sum_{k=1}^na_n^*u^p_{\overline{e}_n}\ldots u^p_{\overline{e}_k}y_k u^p_{e_k}\ldots u^p_{n}a_n+x_n.$$
Since each term in this sum, except $x_n$, is reduced we get, by Proposition \ref{prop:PquasiGNS},
$$\langle\Omega_{p}(p_{0}), b^*b.\Omega_{p}(p_{0})\rangle=\langle\Omega_{p}(p_{0}), x_{n}.\Omega_{p}(p_{0})\rangle.$$
It follows from Lemma \ref{lem:states} that
$$\varphi_{p}(\langle \Omega_{p}(p_{0}), b^{*}b.\Omega_{p}(p_{0})\rangle) = \varphi_{p_{0}}(x_{n})=\varphi_{p_{0}}(\langle 1_{A_{p_{0}}}, a^{*}a.1_{A_{p_{0}}}\rangle).$$
\noindent The end of the proof is routine.
\end{proof}

\begin{rem}\label{rem:GNSstates}
Proposition \ref{prop:PquasiGNS} and Lemma \ref{lem:states} imply that $(H_{p_{0},\varphi_p},\pi_{p_{0}, \varphi_p},\eta_{p_{0},\varphi_p})$ is the GNS construction of the fundamental state on $P_p(p_0)$ defined in Remark \ref{rem:states}.
\end{rem}

\subsection{Unscrewing}\label{subsec:unscrewing}

We will now give an unscrewing process allowing us to recover any fundamental C*-algebra as an inductive limit of iterations of amalgamated free products and HNN extensions. The case of a finite graph is done by induction and the general case is obtained by an inductive limit argument. As an application, we prove that several C*-algebraic properties can be induced, possibly under some extra assumptions, form the vertex algebras to the reduced fundamental C*-algebra. This should be thought of as a C*-algebraic translation of J-P. Serre's \emph{d\'{e}vissage} technique as detailed in \cite[Sec 5.2]{serre1977arbres}.

\vspace{0.2cm}

\noindent The basic principle is quite simple. We start with a non-trivial connected graph and remove an edge. If the graph becomes disconnected, the fundamental C*-algebra is an amalgamated free product of the fundamental C*-algebras of the two connected components. If the graph is still connected, the fundamental C*-algebra is an HNN extension of the fundamental C*-algebra of the remaining graph.

\vspace{0.2cm}

\noindent From now on, we fix a graph of C*-algebras with faithful states $(\Gr, (A_{q},\varphi_q)_{q}, (B_{e},\varphi_e)_{e})$. By Proposition \ref{prop:base} the reduced fundamental C*-algebra do not depend on a particular base (and the isomorphism is canonical). Hence, for the rest of this section, we always omit the superscript $p$ and we simply denote by $u_e$, $e\in\E(\Gr)$, the canonical unitaries.

\vspace{0.2cm}

\noindent Assume that the graph $\Gr$ has at least two edges $e$ and $\rev{e}$. We set $p_{1} = s(e)$ and $p_{2} = r(e)$ and let $P(p_k)$ be the reduced fundamental C*-algebra in $p_k$. Let $\Gr'$ be the graph of C*-algebras obtained from $\Gr$ by removing the edges $e$ and $\rev{e}$.

\vspace{0.2cm}

\noindent\emph{Case 1: The graph $\mathcal{G}'$ is not connected.} Let $\Gr_{k}$ be the connected component of $\Gr'$ containing $p_{k}$ and let $A_k$ be the reduced fundamental C*-algebra of the graph of C*-algebra restricted to $\Gr_k$ in $p_k$ for $k = 1, 2$. By the universal property of Proposition \ref{prop:universalreduced} we can view canonically $A_k\subset P(p_k)$, for $k=1,2$ and, $u_eA_2u_{\overline{e}}\subset P(p_1)$. Denote by $\EE_{k}$ the canonical GNS-faithful conditional expectation from $A_{k}$ to $B_e^s$ if $k=1$ or $B_e^r$ if $k=2$. Set $B = B_{e}$ and let $A_{1}\underset{B}{*} A_{2}$ be the reduced amalgamated free product with respect to the maps $s_{e}$, $r_{e}$ and the conditional expectations $\EE_{k}$.

\begin{lem}\label{lem:freeproduct}
There exists a unique $*$-isomorphism $\rho: A_{1}\underset{B}{*} A_{2}\rightarrow P(p_1)$ such that
\begin{equation*}
\rho(x) = \left\{\begin{array}{ccc}
x & \text{if} & x\in A_{1} \\
u_{e}xu_{\rev{e}} & \text{if} & x\in A_{2}
\end{array}\right.
\end{equation*}
Moreover, $\rho$ is state-preserving.
\end{lem}

\begin{proof}
Define the faithful unital $*$-homomor\-phisms $\rho_{k}: A_{k}\rightarrow P(p_1)$ by $\rho_{1}(x) = x$ and $\rho_{2}(x) = u_{e}xu_{\rev{e}}$. Setting $\imath_{1} = s_{e}$ and $\imath_{2} = r_{e}$, we have
\begin{equation*}
\rho_{1}\circ i_{1} = \rho_{2}\circ i_{2} = s_{e}: B\rightarrow P(p_1).
\end{equation*}
Observe that $P(p_1)$ is generated as a C*-algebra by $\rho_{k}(A_{k})$ for $k = 1, 2$. Let $\EE$ be the canonical GNS-faithful conditional expectation from $P(p_1)$ to $B_e^s$ i.e. $\EE=\EE_e^s\circ\EE_{A_{p_1}}$. By the universal property of the reduced amalgamated free product, it is enough to prove that
\begin{enumerate}
\item $\EE\circ\rho_{k}(x) = \rho_{k}\circ \EE_{k}(x)$ for $x\in A_k$ and $k = 1, 2$.
\item For any $n\geqslant 2$, $a_{1}, \dots, a_{n}$ with $a_{l}\in A_{l_{k}}\ominus i_{l_{k}}(B)$ and $l_{k}\neq l_{k+1}$ for all $k$, one has
$\EE(\rho_{l_{1}}(a_{1}) \dots \rho_{l_{n}}(a_{n})) = 0$.
\end{enumerate}

\noindent $(1).$ We prove it for $k = 2$ (it is obvious for $k = 1$). Let $x\in A_{2}$. We may suppose that $x$ is in $A_{p_2}$ or is a reduced operator in $A_2$. Recall that $\EE_{2} = \EE_{e}^{r}\circ \EE_{A_{p_{2}}}$. If $x = r_{e}(b)\in B_{e}^{r}$, then
\begin{equation*}
\EE(\rho_{2}(x)) = \EE(u_{e}r_{e}(b)u_{\rev{e}}) = \EE(s_{e}(b)) = s_{e}(b) \text{ and } \rho_{2}\circ \EE_{2}(x) = u_{e}r_{e}(b)u_{\rev{e}} = s_{e}(b).
\end{equation*}
If $x\in A_{p_{2}}\ominus B_{e}^{r}$, then $\rho_{2}(x) = u_{e}xu_{\rev{e}}$ is a reduced operator in $P(p_1)$ hence $\EE(\rho_{2}(x)) = 0$. Finally, if $x= a_{0}u_{e_{1}} \dots u_{e_{n}}a_{n}$ is a reduced operator in $A_{2}$ then $\EE_2(x)=0$ and,  since $e_{1}, e_{n} \notin\{e, \rev{e}\}$, $\rho_{2}(x) = u_{e}xu_{\rev{e}}$ is a reduced operator in $P(p_1)$ and $\EE(\rho_{2}(x)) = 0$.

\vspace{0.2cm}

\noindent $(2).$ Let $n\geqslant 2$ and $a_{1}, \dots, a_{n}$ with $a_{k}\in A_{l_{k}}\ominus i_{l_{k}}(B)$ and $l_{k}\neq l_{k+1}$. We may and will assume that all the $a_{k}$'s are either in $A_{p_{l_{k}}}\ominus i_{l_{k}}(B)$ or are reduced operators in $A_{l_{k}}$. Since the edges appearing in the elements $a_{k}$ differ from $e$ and $\rev{e}$, the operator $\rho_{l_{1}}(a_{1}) \dots \rho_{l_{n}}(a_{n})$ is always reduced in $P(p_1)$. Hence, $\EE(\rho_{l_{1}}(a_{1}) \dots \rho_{l_{n}}(a_{n})) = 0$.
\end{proof}

\noindent\emph{Case 2: The graph $\mathcal{G}'$ is connected.} Fix a maximal subtree $\mathcal{T}\subset \Gr'$ and let $g$ be the unique geodesic path in $\mathcal{T}$ from $p_{1}$ to $p_{2}$. Let $A$ be the reduced fundamental C*-algebra in $p_{1}$ of the graph of C*-algebras restricted to $\Gr'$, set $B = B^{s}_{e}\subset A$ and define a faithful unital $*$-homomorphism $\theta: B\rightarrow A$ by $\theta(x) = u_{g}[r_{e}\circ s_{e}^{-1}(x)]u_{g}^{*}$. Let $\EE_{1} = \EE_{e}^{s}$ be the canonical GNS-faithful conditional expectation from $A$ onto $B$ and set $\EE_{-1} = u_{g}\EE_{e}^{r}(u_{g}^{*} . u_{g})u_{g}^{*}$, which is a GNS-faithful conditional expectation from $A$ onto $\theta(B)$. Let $\HNN(A, B, \theta)$ be the reduced HNN extension with respect to the conditional expectations $\EE_{\epsilon}$, $\epsilon\in\{-1, 1\}$ and let $u\in\mathcal{U}(C)$ be the "stable letter" (see \cite{fima2012k}).

\begin{lem}\label{lem:hnnextension}
There exists a unique $*$-isomorphism $\rho: \HNN(A, B, \theta)\rightarrow P(p_1)$ such that,
\begin{equation*}
\rho(u) = u_{g}u_{\rev{e}}\quad\text{and}\quad\rho(x) = x\quad\text{for all}\quad x\in A.
\end{equation*}
Moreover, $\rho$ intertwines the canonical conditional expectations onto $B$.
\end{lem}

\begin{proof}
Set $v = u_{g}u_{\rev{e}}\in P(p_1)$ and view $A$ as a subalgebra of $P(p_1)$ in a canonical way, by the universal property of Proposition \ref{prop:universalreduced}. Let $\EE' = \EE_{e}^{s}\circ\EE_{A_{p_{1}}}$ and note that $\EE'$ is a GNS-faithful conditional expectation from $P(p_1)$ onto $B$. Observe that $vbv^{*} = \theta(b)$ for every $b\in B$ and that $P(p_1)$ is generated, as a C*-algebra, by $A$ and $v$. For $\epsilon\in \{-1, 1\}$, we set
\begin{equation*}
B_{\epsilon} = \left\{\begin{array}{ccc}
B & \text{if} & \epsilon = 1 \\
\theta(B) & \text{if} & \epsilon = -1
\end{array}\right.
\end{equation*}
By the universal property \cite[Proposition 3.2]{fima2012k} of the reduced HNN extension, it is enough to check that for all $n\geqslant 1$, $a_{0}, \dots a_{n}\in A$ and $\epsilon_{1}, \dots, \epsilon_{n}\in \{-1,1\}$ such that $a_{k}\in A\ominus B_{\epsilon_{k}}$ whenever $\epsilon_{k}\neq \epsilon_{k+1}$, one has
\begin{equation*}
\EE'(a_{0}v^{\epsilon_{1}} \dots v^{\epsilon_{n}}a_{n}) = 0.
\end{equation*}
To prove this, we may assume that $a_{k}\in A$ is either a reduced operator or lies in $A_{p_{1}}$. Let $x = a_{0}v^{\epsilon_{1}} \dots v^{\epsilon_{n}}a_{n}$ be a generic element. By suitably selecting elements $a_{k}'$ for $0\leqslant k \leqslant n$ in the following sets:
\begin{itemize}
\item $a_{0}'\in \{a_{0}, a_{0}u_{g}\}$
\item $a_{n}'\in \{a_{n}, u_{\rev{g}}a_{n}\}$
\item $a_{k}'\in \{a_{k}, u_{\rev{g}}a_{k}u_{g}\}$ if $\epsilon_{k+1}\neq\epsilon_{k}$
\item $a_{k}'\in \{a_{k}u_{g}, u_{\rev{g}}a_{k}\}$ if $\epsilon_{k+1} = \epsilon_{k}$
\end{itemize}
we may write $x = a'_{0}u_{\rev{e}}^{\epsilon_{1}} \dots u_{\rev{e}}^{\epsilon_{n}}a'_{n}$. Using the relations of the C*-algebra, we can write each $a_{k}'$ as a sum of reduced operators from $p_i$ to $p_j$, $i,j\in\{1,2\}$ and with edges only in $\Gr'$ and elements of $A_{p_{1}}$ or $A_{p_{2}}$. Hence, we may assume that all the $a_k'$ are reduced such that the edges appearing in the reduced expression are not equal to $e$ or $\overline{e}$ or $a_k'\in A_{p_1}\cup A_{p_2}$. Moreover, we will then have:
\begin{itemize}
\item If $\epsilon_{1} = -1$, $a_{0}'$ is a reduced operator in $A$ or $a_{0}\in A_{p_{1}}$.
\item If $\epsilon_{1} = 1$, $a_{0}'$ is reduced operator from $p_{1}$ to $p_{2}$ or, if $p_{1} = p_{2}$, $a_{0}\in A_{p_{1}}$.
\item If $\epsilon_{n} = 1$, $a_{n}'$ is a reduced operator in $A$ or $a_{n}'\in A_{p_{1}}$.
\item If $\epsilon_{n} = -1$, $a_{n}'$ is reduced operator from $p_{2}$ to $p_{1}$ or, if $p_{1} = p_{2}$, $a_{n}'\in A_{p_{1}}$.
\end{itemize}
and, for $1\leqslant k\leqslant n-1$,
\begin{itemize}
\item If $\epsilon_{k} = 1$ and $\epsilon_{k+1} = -1$, $a_{k}'$ is reduced in $A$ or $a_{k}\in A_{p_{1}}\ominus B_{e}^{s}$.
\item If $\epsilon_{k} = -1$ and $\epsilon_{k+1} = 1$, $a_{k}'$ is a reduced operator from $p_{2}$ to $p_{2}$ or $a_{k}\in A_{p_{2}}\ominus B_{e}^{r}$.
\item If $\epsilon_{k} = 1$ and $\epsilon_{k+1} = 1$, $a_{k}'$ is a reduced operator from $p_{1}$ to $p_{2}$ or, if $p_{1} = p_{2}$, $a_{k}'\in A_{p_{1}}$.
\item If $\epsilon_{k} = -1$ and $\epsilon_{k+1} = -1$, $a_{k}'$ is a reduced operator from $p_{2}$ to $p_{1}$ or, if $p_{1} = p_{2}$, $a_{k}'\in A_{p_{1}}$.
\end{itemize}
Summing up, we see that $x$ is always a reduced operator in $P(p_1)$. Hence, $\EE'(x) = \EE_{e}^{s}\circ\EE_{A_{p_{1}}}(x) = 0$. This concludes the proof.
\end{proof}

\noindent Combining Lemma \ref{lem:freeproduct} and Lemma \ref{lem:hnnextension}, we get the following proposition by a straightforward induction.

\begin{prop}\label{prop:unscrewingfinite}
Let $(\Gr, (A_{q})_{q}, (B_{e})_{e})$ be a \emph{finite} graph of C*-algebras. Then, the reduced fundamental C*-algebra $\pi_{1}(\Gr, (A_{p})_{p}, (B_{e})_{e})$ is isomorphic to an iteration of amalgamated free products and HNN extensions of vertex algebras amalgamated over edge algebras.
\end{prop}

\noindent Using an inductive limit argument, we can extend the previous result to arbitrary graph.

\begin{thm}\label{thm:unscrewing}
Let $(\Gr, (A_{q})_{q}, (B_{e})_{e})$ be a graph of C*-algebras. Then, the reduced fundamental C*-algebra $\pi_{1}(\Gr, (A_{q})_{q}, (B_{e})_{e})$ is isomorphic to an inductive limit of iterations of amalgamated free products and HNN extensions of vertex algebras amalgamated over edge algebras.

\end{thm}

\begin{proof}
Fix $p_{0}\in\V(\Gr)$ and let $A$ be the fundamental C*-algebra in $p_{0}$ of the graph of C*-algebras. Let $C_{f}(\Gr, p_{0})$ be the directed set of connected finite subgraphs of $\Gr$ containing $p_{0}$ ordered by inclusion. For any $\mathcal{K}\in C_{f}(\Gr, p_{0})$ let $A_{\mathcal{K}}$ be the reduced fundamental C*-algebra in $p_{0}$ of the graph of C*-algebras restricted to $\mathcal{K}$. By the universal property of Proposition \ref{prop:universalreduced}, we can view $A_{\mathcal{K}}$ as a subalgebra of $A$ in a canonical way. If $\mathcal{K}_{1}, \mathcal{K}_{2}\in C_{f}(\Gr, p_{0})$ are such that $\mathcal{K}_{1}\subset \mathcal{K}_{2}$ we can also identify in a canonical way $A_{\mathcal{K}_{1}}$ with a subalgebra of $A_{\mathcal{K}_{2}}$. This means that we have an inductive system of unital C*-algebras $(A_{\mathcal{K}})_{\mathcal{K}\in C_{f}(\Gr, p_{0})}$. Let
\begin{equation*}
A_{\infty} = \overline{\bigcup_{\mathcal{K}\in C_{f}(\Gr, p_{0})}A_{\mathcal{K}}}\subset A
\end{equation*}
be the inductive limit of this system. Our claim is that $A_{\infty}$ is in fact equal to $A$. To prove it, it is enough to prove that any reduced operator $x = a_{0}u_{e_{1}} \dots u_{e_{n}}a_{n}\in A$ lies in $A_{\infty}$. In fact, such an operator $x$ is in $A_{\mathcal{K}}$ where $\mathcal{K}$ is a finite connected subgraph of $\Gr$ containing the edges $e_{1}, \dots, e_{n}$ and the vertices $p_{0}, r(e_{1}), \dots r(e_{n-1})$, hence the result.
\end{proof}

\noindent As an application, we can give a permanence property for exactness. Other permanence properties will be proved later on for quantum groups (Section \ref{sec:fundamentalqgroup}) and for graphs of von Neumann algebras (Section \ref{sec:appendix}).

\begin{cor}\label{cor:exact}
Let $(\Gr, (A_{q})_{q}, (B_{e})_{e})$ be a graph of C*-algebras. Then, the reduced fundamental C*-algebra $\pi_{1}(\Gr, (A_{p})_{p}, (B_{e})_{e})$ is exact if and only if all the C*-algebras $A_{p}$ are exact.
\end{cor}

\begin{proof}
The only if part comes from the fact that any subalgebra of an exact C*-algebra is again exact. For the if part, we can restrict to finite graphs since an inductive limit of exact C*-algebras is again exact. We know that exactness passes to amalgamated free products (see \cite[Thm 3.2]{dykema2004exactness})  and also to HNN extensions by \cite[Sec 7.4]{ueda2005hnn}. Combining this with Theorem \ref{thm:unscrewing}, we get the result.
\end{proof}

\section{The fundamental quantum group of a graph of discrete quantum groups}\label{sec:fundamentalqgroup}

\noindent We can now use the results of Section \ref{sec:graphCstar} to define and study the fundamental quantum group of a graph of discrete quantum groups.

\subsection{Definition of the fundamental quantum group}

\begin{de}
A \emph{graph of discrete quantum groups} is a tuple
$$(\Gr, (\G_{q})_{q\in \V(\Gr)}, (\G_{e})_{e\in \E(\Gr)}, (s_{e})_{e\in \E(\Gr)})\quad\text{satisfying the following properties:}$$
\begin{itemize}
\item $\Gr$ is a connected graph.
\item For every $q\in\V(\Gr)$ and every $e\in\E(\Gr)$, $\G_{q}$ and $\G_{e}$ are compact quantum groups.
\item For every $e\in \E(\Gr)$, $\G_{\rev{e}} = \G_{e}$.
\item For every $e\in\E(\Gr)$, $s_{e}: C_{\text{max}}(\G_{e}) \rightarrow C_{\text{max}}(\G_{s(e)})$ is a unital faithful $*$-homomorphism intertwining the coproducts (i.e. $\h{\G}_{e}$ is a discrete quantum subgroup of $\h{\G}_{s(e)}$).
\end{itemize}
\end{de}

\begin{rem}
As mentionned in Remark \ref{rem:discreteqgroups}, Pontryagin duality allows us to use only compact quantum groups when dealing with discrete quantum groups. That is the reason why we call the above object a graph of discrete quantum groups, even though there are only compact quantum groups in the definition.
\end{rem}

\noindent For the remainder of this section we fix a maximal subtree $\T\subset\Gr$ and we denote by $P_{m}$ the maximal fundamental C*-algebra of the graph of C*-algebras $(\Gr, (C_{\text{max}}(\G_{q}))_{q}, (C_{\text{max}}(\G_{e}))_{e})$ with respect to $\mathcal{T}$.

\vspace{0.2cm}

\noindent From a graph of discrete quantum groups, we also obtain another graph of C*-algebras, given by the reduced C*-algebras in the following way. For every $e\in\E(\Gr)$, the map $s_e$ induces a unital faithful $*$-homomorphism $s_e\,:\,C_{\text{red}}(\G_e)\rightarrow C_{\text{red}}(\G_{s(e)})$ which intertwines the coproduct (indeed, since $s_e$ intertwines the coproduct, it maps a unitary representation onto a unitary representation. Moreover, since $s_e$ is injective, it maps an irreducible representation onto an irreducible representation. Hence, $s_e$ preserves the Haar states). Let $\varphi_{s(e)}$ and $\varphi_{e}$ denote the Haar states of $\G_{s(e)}$ and $\G_{e}$ respectively. According to \cite[Prop 2.2]{vergnioux2004k}, $\varphi_{s(e)}\circ s_{e} = \varphi_{e}$ for every $e\in\E(\Gr)$ and there exists a unique conditional expectation $\EE_{e}^{s}: C_{\text{red}}(\G_{s(e)})\rightarrow s_e(C_{\text{red}}(\G_e))$ such that $\varphi_{s(e)} = \varphi_{e}\circ s_{e}^{-1}\circ \EE_{e}^{s}$. In particular, $\EE_{e}^{s}$ is automatically faithful (hence GNS-faithful) since $\varphi_{s(e)}$ is. This conditional expectation is also characterized by the following invariance property:
\begin{equation*}
(\ii\otimes \EE_{e}^{s})\circ\D_{s(e)} = (\EE_{e}^{s}\otimes \ii)\circ\D_{s(e)} = \D_{e}\circ s_{e}^{-1}\circ \EE_{e}^{s} = \D_{s(e)}\circ \EE_{e}^{s}.
\end{equation*}
We obtain a graph of C*-algebras with faithful states $(\mathcal{G},(C_{\text{red}}(\G_{p}),h_p),((C_{\text{red}}(\G_{e}),h_e))$, where $h_p$ and $h_e$ are the Haar states on $C_{\text{red}}(\G_{p})$ and $C_{\text{red}}(\G_{e})$ respectively. Until the end of this section, we fix a vertex $p_{0}\in \V(\Gr)$ and we denote by $P = P(p_{0})$ the reduced fundamental C*-algebra of $(\Gr, (C_{\text{red}}(\G_{q}))_{q}, (C_{\text{red}}(\G_{e}))_{e})$ in $p_{0}$ and by $\varphi$ the fundamental state on $P$. For $q\in\V(\Gr)$, let us denote by $\lambda_{q}$ the canonical surjective $*$-homomorphism $C_{\text{max}}(\G_{q})\rightarrow C_{\text{red}}(\G_q)$ which is the identity on $\text{Pol}(\G_q)$. By the universal property, we have a unique $*$-homomorphism
\begin{equation*}
\lambda: P_{m} \rightarrow P
\end{equation*}
such that $\lambda(u_{e}) = u^{p_{0}}_{e}$ for every $e\in\E(\Gr)$ and $\lambda(a) = \pi^{p_{0}}_{q, p_{0}}\circ\lambda_{q}(a)$ for every $q\in\V(\Gr)$ and every $a\in A_{q}$. In particular, $\lambda$ is injective on $\text{Pol}(\G_q)$ for every $q\in\V(\Gr)$.

\vspace{0.2cm}

\noindent By the universal property of Proposition \ref{prop:universal}, there is a unique unital $*$-homomorphism
\begin{equation*}
\D_{m}: P_{m}\rightarrow P_{m}\otimes P_{m}
\end{equation*}
such that for every $e\in \E(\Gr)$, and every $q\in \V(\Gr), a\in A_{q}$
$$
\D_{m}(u_{e}) = u_{e}\otimes u_{e}\quad\text{and}\quad \D_{m}(a) = \D_{q}(a).$$
Obviously, $P_{m}$ is generated as a C*-algebra by the group-like unitaries $u_{e}$ for $e\in\E(\Gr)$ and the elements $u^{\alpha}_{ij}$ for $\alpha\in \Irr(\G_{q})$, $1\leqslant i,j\leqslant \dim(\alpha)$ and $q\in\V(\Gr)$. The conditions of \cite[Definition 2.1']{wang1995free} being satisfied by these elements, $\G = (P_{m}, \D_{m})$ is a compact quantum group called the \emph{fundamental quantum group} of the graph of quantum groups.

\vspace{0.2cm}

\noindent For $p\in\V(\Gr)$ and $e\in\E(\Gr)$, set $M_p={\rm L}^{\infty}(\G_p)$ and $N_e={\rm L}^{\infty}(\G_e)$. Observe that, for any $e\in\E(\Gr)$, the map $s_e$ induces a unital normal faithful $*$-homomorphism, still denoted $s_e$, from $N_e$ to $M_{s(e)}$ which intertwines the coproducts, the Haar states and the modular groups. Hence, we get a graph of von Neumann algebras $(\mathcal{G},(M_q)_q,(N_e)e)$ and we denote by $M$ the fundamental von Neumann algebra in $p_0$. The hypothesis of section \ref{section:relationcstarvn} being satisfied, the fundamental state $\varphi$ on $P$ is faithful and the von Neumann algebra generated by $P$ in the GNS representation of $\varphi$ is $M$.

\subsection{Representation theory and Haar state}

Let us investigate the representation theory of the fundamental quantum group. For every $e\in\E(\Gr)$, the map $s_{e}$ induces an injective map, still denoted $s_{e}$, from $\Irr(\G_{e})$ to $\Irr(\G_{s(e)})\subset \Irr(\G)$ and $u_{e}\in P_{m}$ is an irreducible representation of dimension $1$. The following definition will be convenient to describe the irreducible representations of $\G$.

\begin{de}
A unitary representation $u$ of $\G$ is said to be \emph{reduced} if
$$
u = u^{\alpha_{0}}\otimes u_{e_{1}}\otimes \dots \otimes u_{e_{n}}\otimes u^{\alpha_{n}}$$
where $n\geqslant 1$ and
\begin{itemize}
\item $(e_{1}, \dots, e_{n})$ is a path in $\Gr$ from $p_{0}$ to $p_{0}$
\item $\alpha_{0} \in \Irr(\G_{p})$ and for $1\leqslant i\leqslant n$, $\alpha_{i}\in \Irr(\G_{r(e_i)})$
\item For all $1\leqslant i\leqslant n-1$, $\alpha_{i}\notin s_{e_{i+1}}(\Irr(\G_{e_{i+1}}))$ whenever $e_{i+1}=\rev{e}_{i}$.
\end{itemize}
\end{de}

\begin{thm} We have,
\begin{enumerate}
\item For an irreducible unitary representation $u$ of $\G$ one of the following holds:
\begin{itemize}
\item $u=u_e$ for some $e\in \E(\Gr)$.
\item $u$ is an irreducible representation of $\G_{q}$ for some $q\in\V(\Gr)$.
\item $u$ is unitarily equivalent to a subrepresentation of a reduced representation.
\end{itemize}
Hence, the set $\bigcup_{q\in\V(\Gr)}\Irr(\G_{q})\cup\{u_{e}: e\in\E(\Gr)\}$ generates the representation category of $\G$.
\item The Haar state $h$ of $\G$ is given by $h = \varphi\circ\lambda$.
\item We have $C_{\text{max}}(\G)=P_m$, $C_{\text{red}}(\G)=P$, ${\rm L}^{\infty}(\G)=M$, and $\lambda = \lambda_{\G}$.
\end{enumerate}
\end{thm}

\begin{proof}
$(1)$. It is clear that the unitaries $u_{e}$ and the irreducible representations of the compact quantum groups $\G_{q}$ are irreducible representations of $\G$. Moreover, the closure of the $*$-algebra $\mathcal{A}$ generated by the coefficients of the aforementioned representations and the reduced representations contains, by definition of the conditional expectations, the linear span of $C_{\text{max}}(\G_{p_0})$ and the reduced operators from $p_0$ to $p_0$ in $P_m$, which is a dense $*$-subalgebra in $P_m$ (see Remark \ref{rem:redmax}). Since $\mathcal{A}$ is dense in $P_{m}$ and generated by coefficients of representations, it must contain the coefficients of all irreducible representations i.e. $\mathcal{A}=\text{Pol}(\G)$. This proves $(1)$.

\vspace{0.2cm}

\noindent $(2)$. Let $\Q_{m} \subset P_{m}$ be the linear span of the coefficients of all reduced representations. One can easily check that $\D_{m}(\Q_{m})\subset \Q_{m}\odot \Q_{m}$ and that $\lambda(\Q_{m})$ is spanned by reduced operators. Hence, for any $x\in \Q_{m}$, $(\varphi\otimes \ii)\circ\D_{m}(x) = (\ii\otimes \varphi)\circ\D_{m}(x) = 0 = \varphi(x).1$. Since the linear span of $\Q_{m}$ and $C_{\text{max}}(\G_{p_0})$ is a dense $*$-subalgebra of $P_{m}$, we only have to check the invariance property of $\varphi$ on $C_{\text{max}}(\G_{p_0})$, which is obvious.

\vspace{0.2cm}

\noindent $(3)$. The map $\lambda$ is surjective and $\varphi$ is faithful on $P$, hence $P$ is the reduced C*-algebra of $\G$. The universal property of Proposition \ref{prop:universal} implies that $P_{m}$ is the enveloping C*-algebra of $\Pol(\G)$, i.e. $P_{m}$ is the maximal C*-algebra of $\G$. Moreover, $\lambda = \lambda_{\G}$ since it is the identity on $\Pol(\G)$. Eventually, ${\rm L}^{\infty}(\G)=M$ because $M$ is the von Neumann algebra generated by $P$ in the GNS representation of $\varphi$.
\end{proof}

\begin{rem}
In the case of a free product \emph{without amalgamation}, the irreducible representations are exactly the reduced representations (together with the representations coming from the quantum groups). However, this fails as soon as one allows amalgamation, as shown in the example at the end of Section $2$ of \cite{vergnioux2004k}.
\end{rem}

\noindent Combining the results of this section with Examples \ref{ex:freeproduct} and \ref{ex:hnn}, we see that we recover both the amalgamated free product construction of \cite{wang1995free} and the HNN construction of \cite{fima2012k}.

\subsection{Permanence properties}

We give in this section some permanence results for approximation properties under the fundamental quantum group construction. Let us say that a unimodular discrete quantum group $\h{\G}$ is \emph{hyperlinear} if the von Neumann algebra ${\rm L}^{\infty}(\G)$ embeds into an ultraproduct $R^{\omega}$ of the hyperfinite II$_{1}$ factor $R$.

\begin{thm}
Let $(\Gr, (\G_{q})_{q}, (\G_{e})_{e})$ be a graph of discrete quantum groups.
\begin{enumerate}
\item $\h{\G}$ is exact if and only if all the vertex quantum groups (hence all the edge quantum groups) are exact.
\item $\h{\G}$ is unimodular if and only if all the vertex quantum groups (hence all the edge quantum groups) are unimodular.
\item If all the vertex quantum groups are unimodular and hyperlinear and if all the edge quantum groups are amenable, then $\h{\G}$ is hyperlinear.
\item If all the vertex quantum groups are unimodular and have the Haagerup property and if all the edge quantum groups are finite, then $\h{\G}$ has the Haagerup property.
\item If all the vertex quantum groups are unimodular and weakly amenable with Cowling-Haagerup constant $1$ and if all the edge quantum groups are finite, then $\h{\G}$ is weakly amenable with Cowling-Haagerup constant $1$.
\end{enumerate}
\end{thm}

\begin{proof}
$(1)$ follows from Corollary \ref{cor:exact} while $(2)$ follows from Propositions \ref{prop:modulargraph} and \ref{prop:quantumvonneumann}. $(3)$ and $(4)$ are proved in Corollary \ref{cor:embedable}, by using again Proposition \ref{prop:quantumvonneumann} (see also \cite[Proposition 7.13 and 7.14]{Daws2013haagerup} for the cases of amalgamated free product and HNN extensions). Finally, $(5)$ is a  straightforward consequence of Theorem \ref{thm:unscrewing} and the permanence properties proved in \cite[Chap 2]{freslon2013proprietes}.
\end{proof}

\section{K-amenability}\label{sec:kamenability}

\noindent In this section we will illustrate the construction of the fundamental quantum group by generalizing the Julg-Valette theorem \cite{julg1984k}.

\begin{thm}\label{thm:kamenabilitygraph}
The fundamental quantum group of a graph of amenable discrete quantum groups is $K$-amenable.
\end{thm}

\noindent The proof will be done in several steps. The strategy consists in using the natural representations of the reduced fundamental C*-algebra on a quantum Bass-Serre tree, i.e. an analogue of the $\ell^{2}$-spaces of vertices and edges of the Bass-Serre tree associated to a classical graph of groups. In that way, we get two representations of the reduced C*-algebra. We then build a KK-element and prove that it yields the K-amenability of the fundamental quantum group.

\vspace{0.2cm}

\noindent From now on, we fix an \emph{oriented} graph of compact quantum groups $(\Gr, (\G_{q})_{q}, (\G_{e})_{e})$ such that all the compact quantum groups $\G_{q}$ are co-amenable. We set $A_q=C_{\text{max}}(\G_q)=C_{\text{red}}(\G_q)$, $B_e=C_{\text{max}}(\G_e)=C_{\text{red}}(\G_e)$ and we use the notations of the preceding sections. For $q\in\V(\Gr)$ and $e\in\E(\Gr)$, denote by $\varepsilon_q\,:\,A_q\rightarrow\C$ and $\varepsilon_e\,:\,B_e\rightarrow\C$ the counit of $\G_q$ and $\G_e$ respectively. Since the maps $s_e$, $r_e$ are faithful and intertwine the coproducts we have, by the characterization of the counit given in Section \ref{section:CQG}, $\varepsilon_{s(e)}\circ s_e=\varepsilon_e=\varepsilon_{r(e)}\circ r_e$ for every $e\in\E(\Gr)$.

\vspace{0.2cm}

\noindent Let $\G = (P,\Delta)$ be the reduced fundamental quantum group of the graph of quantum groups, where $P=P(p_0)$ is the reduced fundamental C*-algebra at a fixed vertex $p_{0}\in\V(\Gr)$. Since the Haar states $\varphi_{q}$, $q\in\V(\Gr)$, and $\varphi_{e}$, $e\in\E(\Gr)$ form a graph of faithful states, we can identify canonically $P$ with any of the C*-algebras $P_{p}(p_{0})$ for $p\in\V(\Gr)$. We will consequently identify $P$ with its images in all the spaces $\LL_{A_{p}}(\HH_{p_{0},p})$ for $p\in\V(\Gr)$ and omit the superscripts $p$ in the unitaries $u_{e}^{p}$, i.e. we write $a_{0} u_{e_{1}} \dots u_{e_{n}} a_{n}$ for a reduced operator in $P$.

\subsection{The quantum Bass-Serre tree}

Let us fix a vertex $p\in \V(\Gr)$. For every $q\in\V(\Gr)$, let $(L_{p, q}, \pi_{p, q}, \eta_{p, q})$ be the GNS construction of $(\HH_{p, q}, \varepsilon_{q})$. Set $\xi_{p, q}^{L} = \eta_{p, q} (\Omega_{q}(p))$. The "$\ell^{2}$-space of vertices" relative to $p$ is the Hilbert space
\begin{equation*}
L_{p} = \bigoplus_{q\in\V(\Gr)}L_{p, q}
\end{equation*}
on which $P( p )$ ($\simeq P_q( p )$) acts by $\pi_{p} = \oplus_{q} \pi_{p, q}$. For $p=p_0$ we write $L=L_{p_0}$, $\pi=\pi_{p_0}$ and $\xi^L_{p_0}=\xi^L_{p_0,p_0}$.

\vspace{0.2cm}

\noindent For every $p \in \V(\Gr)$ and every $f \in \E(\Gr)$, let $(K_{p, f}, \rho_{p, f}, \eta_{p, f})$ be the GNS construction of $(\HH_{p, s(f)}, \varepsilon_{s(f)}\circ\EE_{f}^{s})$ and set $\xi_{p, f}^{K} = \eta_{p, f}(\Omega_{s(f)}(q))$. Let $\E^{+}(\Gr)$ be the set of positive vertices corresponding to the orientation. Then, the "$\ell^{2}$-space of \emph{positive} edges" relative to $p$ is the Hilbert space
\begin{equation*}
K_{p} = \bigoplus_{f\in\E^{+}(\Gr)} K_{p, f}
\end{equation*}
on which $P( p )$ ($\simeq P_{s(f)}( p )$) acts by $\rho_{p} = \oplus_{f\geqslant 0} \rho_{p, f}$. For $p=p_0$ we write $K=K_{p_0}$ and $\rho=\rho_{p_0}$.

\noindent Let us give some relations between the norms of these Hilbert spaces.

\begin{lem}\label{lem:jvisometry}
Let $w = (e_{1}, \dots, e_{n})$ be a path in $\Gr$  from $p$ to $q$ and let $a= \h{a}_{0}\underset{e_{1}}{\otimes} \dots \underset{e_{n}}{\otimes} a_{n} \in \HH_{w}$. Then,
\begin{eqnarray*}
\|\eta_{p, q}(a)\|^{2}_{L_{p, q}} & = & \vert\varepsilon_{q}(a_{n})\vert^{2}\|\eta_{p, e_{n}}(\h{a}_{0}\underset{e_{1}}{\otimes} \dots \underset{e_{n-1}}{\otimes} a_{n-1})\|^{2}_{K_{p, e_{n}}} \\
& = & \vert\varepsilon_{q}(a_{n})\vert^{2}\|\eta_{p, \rev{e}_{n}}(\h{a}_{0}\underset{e_{1}}{\otimes} \dots \underset{e_{n-1}}{\otimes} \widehat{a}_{n-1}\underset{e_{n}}{\otimes} 1)\|^{2}_{K_{p, \rev{e}_{n}}}.
\end{eqnarray*}
\end{lem}

\begin{proof}
Let $x_0=a_0^*a_0$ and, for $1\leqslant k\leqslant n$, $x_k=a_k^*(r_{e_k}\circ s_{e_k}^{-1}\circ \EE_{e_k}^s(x_{k-1}))a_k$. Set $a'=\h{a}_{0}\underset{e_{1}}{\otimes} \dots \underset{e_{n-1}}{\otimes} \widehat{a}_{n-1}\underset{e_{n}}{\otimes} 1\in\mathcal{H}_{p,s(\overline{e}_n)}$ and $a''=\h{a}_{0}\underset{e_{1}}{\otimes} \dots \underset{e_{n-1}}{\otimes} a_{n-1}\in\mathcal{H}_{p,s(e_n)}$. By Lemma \ref{lem:innerproduct}, we have
\begin{equation*}
\langle a,a\rangle_{\mathcal{H}_{p,q}}=x_n,\quad \langle a',a'\rangle_{\mathcal{H}_{p,s(e_n)}}=x'_{n}\quad\text{and}\quad\langle a'',a''\rangle_{\mathcal{H}_{p,s(\overline{e}_n)}}=x_{n-1},
\end{equation*}
where $x'_n=r_{e_n}\circ s_{e_n}^{-1}\circ\EE_{e_n}^s(x_{n-1})$. Hence, we get
\begin{equation*}
\|\eta_{p, q}(a)\|^{2}_{L_{p, q}}=\varepsilon_q(x_n)=\varepsilon_q(a_n^*(r_{e_n}\circ s_{e_n}^{-1}\circ \EE_{e_n}^s(x_{n-1}))a_n)
=\vert\varepsilon_{q}(a_{n})\vert^{2}\varepsilon_{r(e_n)}\circ r_{e_n}\circ s_{e_n}^{-1}\circ \EE_{e_n}^s(x_{n-1})
\end{equation*}
and
\begin{equation*}
\|\eta_{p, \overline{e}_{n}}(a')\|^{2}_{K_{p, \overline{e}_{n}}}=\varepsilon_{s(\overline{e}_n)}\circ\EE_{\overline{e}_n}^s(x_n')=\varepsilon_{r(e_n)}\circ r_{e_n}\circ s_{e_n}^{-1}\circ \EE_{e_n}^s(x_{n-1}).
\end{equation*}
This proves the second equality. Moreover, using the relation $\varepsilon_{r(e_n)}\circ r_{e_n}=\varepsilon_{s(e_n)}\circ s_{e_n}$, we have
\begin{equation*}
\varepsilon_{r(e_n)}\circ r_{e_n}\circ s_{e_n}^{-1}\circ \EE_{e_n}^s(x_{n-1})=\varepsilon_{s(e_n)}\circ\EE_{e_n}^s(x_{n-1})=\|\eta_{p, e_{n}}(a'')\|^{2}_{K_{p, e_{n}}},
\end{equation*}
proving the first equality.
\end{proof}

\begin{rem}\label{rem:lastspacetrivial}
In $L_{p, q}$ we have $\eta_{p, q}(\widehat{a}_{0}\otimes \dots \otimes a_{n}) = \varepsilon_{q}(a_{n})\eta_{p, q}(\widehat{a}_{0}\otimes \dots \otimes 1)$. Indeed, the formula is obvious for $n=0$ and, for $n\geqslant 1$, it follows by Lemma \ref{lem:jvisometry} that
$$||\eta_{p, q}(\widehat{a}_{0}\otimes \dots \otimes a_{n})- \varepsilon_{q}(a_{n})\eta_{p, q}(\widehat{a}_{0}\otimes \dots \otimes 1) ||^2
= ||\eta_{p, q}(\widehat{a}_{0}\otimes \dots \otimes (a_{n}-\varepsilon_q(a_n)))||^2=0.$$
\end{rem}

\begin{rem}\label{rem:lastspaceK}
An easy computation shows that in $K_{p,f}$ we have, for all $n\geq 0$,
$$\eta_{p,f}(\widehat{a}_{0}\otimes \dots \otimes a_{n}b) =\varepsilon_{s(f)}(b)\eta_{q,f}(\widehat{a}_{0}\otimes \dots \otimes a_{n})\quad\text{for all}\,\,b\in B_f^s.$$

\noindent Actually, the following holds for the spaces $K_{p, f}$:
\begin{equation*}
\eta_{p, f}(a_{0}\otimes \dots \otimes a_{n}) = \varepsilon_{s(f)}\circ\EE_{f}^{s}(a_{n})\eta_{p, f}(a_{0}\otimes \dots \otimes 1) + \eta_{p, f}(a_{0}\otimes \dots \otimes \PP_{f}^{s}(a_{n})).
\end{equation*}
\end{rem}

\subsection{The Julg-Valette operator}

We now define an operator $\F: L \rightarrow K$ which will give our KK-element. In the classical situation, the operator $\F$ associates to each vertex $p$ the last edge in the unique geodesic path from $p_{0}$ to $p$. Here, the construction is similar but we use the orientation we fixed on the graph $\Gr$ instead of the ascending orientation. Set $\F(\xi_{p_{0}}^{L}) = 0$ and for every path $w = (e_{1}, \dots, e_{n})$ from $p_{0}$ to $q$ and every $x = \h{x}_{0}\underset{e_{1}}{\otimes} \dots \underset{e_{n}}{\otimes} x_{n} \in \HH_{w}$, set
\begin{equation*}
\F(\eta_{p_{0}, q}(x)) = \left\{\begin{array}{ccc}
\varepsilon_{r(e_{n})}(x_{n})\eta_{p_{0}, e_{n}}(\h{x}_{0}\underset{e_{1}}{\otimes} \dots \underset{e_{n-1}}{\otimes} x_{n-1}) \in K_{p_{0}, e_{n}} & \text{if} & e_{n}\in \E^{+}(\Gr) \\
\varepsilon_{r(e_{n})}(x_{n})\eta_{p_{0}, \rev{e}_{n}}(\h{x}_{0}\underset{e_{1}}{\otimes} \dots \underset{e_{n-1}}{\otimes} \h{x}_{n-1} \underset{e_{n}}{\otimes} 1) \in K_{p_{0}, \rev{e}_{n}} & \text{if} & e_{n}\notin \E^{+}(\Gr) \\
\end{array}\right.
\end{equation*}

\noindent The operator $\F$ is called the \emph{Julg-Valette operator} associated to the graph of quantum groups. For every vertex $q\in \V(\Gr)$, let $\HH^{+}_{p_{0}, q}$ (resp. $\HH^{-}_{p_{0}, q}$) be the direct sum of all path Hilbert modules $\HH_{w}$ where $w$ is a non-empty path from $p_{0}$ to $q$ such that the last edge of $w$ is positive (resp. negative) and let $L_{p_{0}, q}^{+}$ (resp. $L_{p_{0}, q}^{-}$) be the closure of $\eta_{p_{0}, q}(\HH^{+}_{p_{0}, q})$ (resp.  $\eta_{p_{0}, q}(\HH^{-}_{p_{0}, q})$). We have a decomposition
\begin{equation*}
L_{p_{0}, q} = \left\{\begin{array}{lcl}
L_{p_{0}, q}^{+}\oplus L_{p_{0}, q}^{-}\oplus\C.\xi_{p_{0}}^{L} & \text{if} & p_{0} = q,\\
L_{p_{0}, q}^{+}\oplus L_{p_{0}, q}^{-} & \text{if} & p_{0}\neq q.
\end{array}\right.
\end{equation*}
Summing up over all $q$'s gives a similar decomposition $L = L^{+}\oplus L^{-}\oplus\C.\xi_{p_{0}}^{L}$ where,
\begin{equation*}
L^{+} = \bigoplus_{q\in\V(\Gr)}L_{p_{0}, q}^{+}\text{ and } L^{-} = \bigoplus_{q\in\V(\Gr)}L_{p_{0}, q}^{-}.
\end{equation*}
It is easily seen from Lemma \ref{lem:jvisometry} that $\F$ is an isometry on $L^{+}$ and $L^{-}$ and hence extends to a bounded operator so that, $\F(L^{+})$ and $\F(L^{-})$ being orthogonal, $\F$ is isometric on the orthogonal complement of $\xi_{p_{0}}^{L}$. Since $\F$ is also surjective on that space, we have $\F\F^{*} = \Id_{K}$. This also implies that $\F^{*}\F = \Id_{L} - p_{\xi_{p_{0}}^{L}}$, where $p_{\xi_{p_{0}}^{L}}$ is the orthogonal projection onto $\C.\xi_{p_{0}^{L}}$. In short, we have proven that $\F$ is unitary modulo compact operators. In order to get a KK-element, we now have to prove that $\F$ commutes with the representations $\pi$ and $\rho$ of $P$ up to the compact operators.

\begin{lem}
With the notations above, the following hold:
\begin{enumerate}
\item For every $a\in A_{p_{0}}$ we have $\F\circ\pi(a) = \rho(a)\circ\F$.
\item For every reduced operator $a = a_{n} u_{e_{n}} \dots u_{e_{1}}a_{0}\in P$, $\I(\F\circ\pi(a) - \rho(a)\circ \F) = X_{a}$, where
\begin{eqnarray*}
X_{a} & = & \underset{1\leqslant k\leqslant n}{\Span}\left(\left\{\eta_{p_{0}, e_{k}}(\h{a}_{n}\underset{e_{n}}{\otimes} \dots \underset{e_{k+1}}{\otimes} a_{k}): e_{k}\geqslant 0\right\}\right) \\
& \bigoplus & \underset{1\leqslant k\leqslant n}{\Span}\left(\left\{ \eta_{p_{0},\rev{e}_{k}}(\h{a}_{n}\underset{e_{n}}{\otimes} \dots \underset{e_{k+1}}{\otimes}\h{a}_{k}\underset{e_{k}}{\otimes} 1): e_{k}\leqslant 0\right\}\right).
\end{eqnarray*}
\item For every $a\in P$ the operator $\F\circ\pi(a)-\rho(a)\circ\F$ is compact.
\end{enumerate}
\end{lem}

\begin{proof}
The proof of $(1)$ is obvious. Let us prove $(2)$. Let $a = a_{n} u_{e_{n}} \dots u_{e_{1}}a_{0}$, where $n\geqslant 1$ and $(e_{n}, \dots, e_{1})$ is a path in $\Gr$ from $p_{0}$ to $p_{0}$. Observe that $(\F\circ\pi(a)-\rho(a)\circ\F)\xi^{L}_{p_{0}}$ is equal to
\begin{equation*}
\F\circ\pi(a)\xi^{L}_{p_{0}} = \varepsilon_{p_{0}}(a_{0})\left\{
\begin{array}{lcl}
\eta_{p_{0}, e_{1}}(\h{a}_{n}\underset{e_{n}}{\otimes} \dots \underset{e_{2}}{\otimes} a_{1})\in X_{a} & \text{if} & e_{1}\geqslant 0,\\
\eta_{p_{0}, \rev{e}_{1}}(\h{a}_{n}\underset{e_{1}}{\otimes} \dots \underset{e_{2}}{\otimes} \h{a}_{1}\underset{e_{1}}{\otimes} 1)\in X_{a} & \text{if} & e_{1}\leqslant 0.\end{array}\right.
\end{equation*}
Take now $\xi\in L$, of the form $\eta_{p_{0}, q}(b)$ for some $b = \h{b}_{0}\underset{f_{1}}{\otimes} \dots \underset{f_{m}}{\otimes} 1 \in \HH_{p_{0}, q}$, where $m\geqslant 1$ and $(f_{1}, \dots, f_{m})$ is a path from $p_{0}$ to $q$. Set $b' = b_{0}$ if $m = 1$ and, $b' = \h{b}_{0}\underset{f_{1}}{\otimes} \dots \underset{f_{m-1}}{\otimes} b_{m-1}$ if $m > 1$. We have
\begin{equation}\label{eqrank1}
(\F\circ\pi(a)-\rho(a)\circ\F)\xi = \left\{\begin{array}{lcl}
\F\eta_{p_{0}, q}(a.b) - \eta_{p_{0}, f_{m}}(a.b') & \text{if} & f_{m} \geqslant 0, \\
\F\eta_{p_{0}, q}(a.b) - \eta_{p_{0}, \rev{f}_{m}}(a.b) & \text{if} & f_{m}\leqslant 0
\end{array}\right.
\end{equation}
Using the notations of Lemma \ref{lem:product}, we have an integer $n_{0}$ associated to the pair $(a, b)$. Denote by $n_{0}'$ the integer associated to the pair $(a, b')$ (assume $m\geqslant 2$, otherwise it is easy to conclude). Observe that
\begin{equation*}
n_{0}' = 
\left\{\begin{array}{lcl}
n_{0} & \text{if} & n_{0} < m, \\
m-1 & \text{if} & n_{0} = m.
\end{array}\right.
\end{equation*}

\vspace{0.2cm}

\noindent\textbf{Case 1: $\xi\in L^{-}$.} This means that $f_{m}\leqslant 0$.

\vspace{0.2cm}

\noindent If $n < m$, Lemma \ref{lem:product} implies that $a.b\in\HH_{p_{0}, q}^{-}$ and, by definition of $\F$, $(\F\circ\pi(a) - \rho(a)\circ\F)\xi = 0$.

\vspace{0.2cm}

\noindent Assume $n\geqslant m$. If $n_{0} < m$, Lemma \ref{lem:product} again implies that $a.b\in\HH_{p_{0}, q}^{-}$ and we get, as before, $(\F\circ\pi(a) - \rho(a)\circ\F)\xi = 0$. If $n_{0} = m$, $e_{m} = \rev{f}_{m} \geqslant 0$ and, with the notations of Lemma \ref{lem:product}, we have
\begin{equation*}
a.b = \sum_{k = 1}^{m}\h{a}_{n}\underset{e_{n}}{\otimes} \dots \underset{e_{k}}{\otimes}\h{y}_{k}\underset{f_{k}}{\otimes} \dots \underset{f_{m}}{\otimes} 1 + \h{a}_{n} \underset{e_{n}}{\otimes} \dots \underset{e_{m+1}}{\otimes} x_{m},
\end{equation*}
where $x_{m} = a_{m}z$ and $z\in B_{e_{m}}^{s} = B_{f_{m}}^{r}$. On the one hand we have, since $f_{m}\leqslant 0$,
\begin{equation*}
\F\eta_{p_{0}, q}(a.b) = \sum_{k = 1}^{m}\eta_{p_{0}, \rev{f}_{m}}(\h{a}_{n}\underset{e_{n}}{\otimes} \dots \underset{e_{k}}{\otimes}\h{y}_{k}\underset{f_{k}}{\otimes} \dots \underset{f_{m}}{\otimes} 1) + \varepsilon_{q}(x_{m})\mathcal{F}\eta_{p_{0}, q}(\h{a}_{n}\underset{e_{n}}{\otimes} \dots \underset{e_{m+1}}{\otimes} 1)
\end{equation*}
and $\F\eta_{p_{0}, q}(\h{a}_{n}\underset{e_{n}}{\otimes} \dots \underset{e_{m+1}}{\otimes} 1) = \left\{\begin{array}{lcl}
\eta_{p_{0}, e_{m+1}}(\h{a}_{n}\underset{e_{n}}{\otimes} \dots \underset{e_{m+2}}{\otimes}a_{m+1})\in X_{a} & \text{if} & e_{m+1}\geqslant 0,\\
\eta_{p_{0}, \rev{e}_{m+1}}(\h{a}_{n}\underset{e_{n}}{\otimes} \dots \underset{e_{m+1}}{\otimes} 1) \in X_{a} & \text{if} & e_{m+1}\leqslant 0.
\end{array}\right.$

\vspace{0.2cm}

\noindent On the other hand, since $z\in B_{f_{m}}^{r}$ we get
\begin{equation*}
\eta_{p_{0}, \rev{f}_m}(a.b) = \sum_{k = 1}^{m}\eta_{p_{0},\rev{f}_{m}}(\h{a}_{n}\underset{e_{n}}{\otimes} \dots \underset{e_{k}}{\otimes}\h{y}_{k}\underset{f_{k}}{\otimes} \dots \underset{f_{m}}{\otimes} 1) + \varepsilon_{q}(z)\eta_{p_{0},\rev{f}_{m}}(\h{a}_{n}\underset{e_{n}}{\otimes} \dots \underset{e_{m+1}}{\otimes} a_{m})
\end{equation*}
and $\eta_{p_{0}, \rev{f}_{m}}(\h{a}_{n}\underset{e_{n}}{\otimes} \dots \underset{e_{m+1}}{\otimes} a_{m}) = \eta_{p_{0}, e_{m}}(\h{a}_{n}\underset{e_{n}}{\otimes} \dots \underset{e_{m+1}}{\otimes} a_{m}) \in X_{a}$. Hence, by Equation (\ref{eqrank1}), we get
\begin{equation*}
(\F\circ\pi(a) - \rho(a)\circ\F) \xi = \varepsilon_{q}(x_{m})\F\eta_{p_{0}, q}(\h{a}_{n}\underset{e_{n}}{\otimes} \dots \underset{e_{m+1}}{\otimes} 1) - \varepsilon_{q}(z)\eta_{p_{0}, \rev{f}_{m}}(\h{a}_{n}\underset{e_{n}}{\otimes} \dots \underset{e_{m+1}}{\otimes} a_{m}) \in X_{a}.
\end{equation*}

\vspace{0.2cm}

\noindent\textbf{Case 2: $\xi\in L^{+}$.} This means that $f_{m}\geqslant 0$.

\vspace{0.2cm}

\noindent If $n < m$, we have $n_{0}' = n_{0}\leqslant n < m$. Lemma \ref{lem:product} applied to the pairs $(a, b)$ and $(a, b')$ and the definition of $\F$ imply that $(\F\circ\pi(a) - \rho(a)\circ\F) \xi = 0$.

\vspace{0.2cm}

\noindent Assume $n\geqslant m$. If $n_{0} < m$ then $n_{0}' = n_{0}$ and Lemma \ref{lem:product} again implies $(\F\circ\pi(a) - \rho(a)\F) \xi = 0$. If $n_{0} = m$ then $n_{0}' = m-1$, with the notations of Lemma \ref{lem:product} we have
\begin{equation*}
a.b= \sum_{k = 1}^{m}\h{a}_{n}\underset{e_{n}}{\otimes} \dots \underset{e_{k}}{\otimes}\h{y}_{k}\underset{f_{k}}{\otimes} \dots \underset{f_{m}}{\otimes} 1 + \h{a}_{n} \underset{e_{n}}{\otimes} \dots \underset{e_{m+1}}{\otimes} x_{m}
\end{equation*}
and
\begin{equation*}
a.b' = \sum_{k = 1}^{m-1}\h{a}_{n}\underset{e_{n}}{\otimes} \dots \underset{e_{k}}{\otimes}\h{y}_{k}\underset{f_{k}}{\otimes} \dots \underset{f_{m-1}}{\otimes}b_{m-1} + \h{a}_{n} \underset{e_{n}}{\otimes} \dots \underset{e_{m}}{\otimes} x_{m-1}.
\end{equation*}
where $x_{m-1}\in A_{r(e_{m})}$ and $y_{m} = x_{m-1}-\EE_{e_{m}}^r{(}x_{m-1})$. On the one hand we have, since $f_{m}\geqslant 0$,
\begin{eqnarray*}
\F\eta_{p_{0}, q}(a.b) & = & \sum_{k = 1}^{m-1}\eta_{p_{0}, f_{m}}(\h{a}_{n}\underset{e_{n}}{\otimes} \dots \underset{e_{k}}{\otimes}\h{y}_{k}\underset{f_{k}}{\otimes} \dots \underset{f_{m-1}}{\otimes}b_{m-1})+\eta_{p_{0}, f_{m}}(\h{a}_{n}\underset{e_{n}}{\otimes} \dots \underset{e_{m}}{\otimes}y_{m})\\
& + & \varepsilon_{q}(x_{m})\F\eta_{p_{0}, q}(\h{a}_{n}\underset{e_{n}}{\otimes} \dots \underset{e_{m+1}}{\otimes} 1).
\end{eqnarray*}
As before, we see easily that $\F\eta_{p_{0}, q}(\h{a}_{n}\underset{e_{n}}{\otimes} \dots \underset{e_{m+1}}{\otimes} 1)\in X_{a}$. On the other hand, 
\begin{equation*}
\eta_{p_{0}, f_{m}}(a.b') = \sum_{k = 1}^{m-1}\eta_{p_{0}, f_{m}}(\h{a}_{n}\underset{e_{n}}{\otimes} \dots \underset{e_{k}}{\otimes}\h{y}_{k}\underset{f_{k}}{\otimes} \dots \underset{f_{m-1}}{\otimes}b_{m-1}) + \eta_{p_{0}, f_{m}}(\h{a}_{n}\underset{e_{n}}{\otimes} \dots \underset{e_{m}}{\otimes}x_{m-1}).
\end{equation*}
Hence, by Equation (\ref{eqrank1}), we have
\begin{equation*}
(\F\circ\pi(a) - \rho(a)\circ\F) \xi =
\varepsilon_{q}(x_{m})\F\eta_{p_{0}, q}(\h{a}_{n}\underset{e_{n}}{\otimes} \dots \underset{e_{m+1}}{\otimes} 1)
+\eta_{p_{0}, f_{m}}(\h{a}_{n}\underset{e_{n}}{\otimes} \dots \underset{e_{m}}{\otimes}(y_{m}-x_{m-1})).
\end{equation*}

\noindent Since $y_{m} - x_{m-1} = -\EE_{e_{m}}^{r}(x_{m-1}) = -\EE_{f_{m}}^{s}(x_{m-1})$,
\begin{equation*}
\eta_{p_{0}, f_{m}}(\h{a}_{n}\underset{e_{n}}{\otimes} \dots \underset{e_{m}}{\otimes}(y_{m}-x_{m-1})) = -\varepsilon_{s(f_{m})}\circ\EE_{f_{m}}^{s}(x_{m-1})\eta_{p_{0}, \rev{e}_{m}}(\h{a}_{n}\underset{e_{n}}{\otimes} \dots \underset{e_{m}}{\otimes} 1)\in X_{a}.
\end{equation*}
Hence, $(\F\circ\pi(a) - \rho(a)\circ\F)\xi\in X_{a}$. 

\vspace{0.2cm}

\noindent $(3)$ now follows easily from $(1)$ and $(2)$ since $(\F\circ\pi(a) - \rho(a)\circ\F)$ has finite rank for every $a$ in the linear span of $A_{p_{0}}$ and the reduced operators, which is dense in $P$.
\end{proof}

\noindent We thus have constructed a KK-element $\gamma = [(L, K, \F)]\in \KK(C_{\text{red}}(\G), \C)$. Let us describe a triple which is equal to $\lambda_{\G}^{*}(\gamma) - [\varepsilon]$. Set
\begin{equation*}
\widetilde{K} = K\oplus \C.\Omega
\end{equation*}
for some norm-one vector $\Omega$ and endow it with $\widetilde{\rho}$, which is the direct sum of the representation $\rho\circ\lambda_{\G}$ and the trivial representation. Let $F: L\rightarrow \widetilde{K}$ be the operator defined by $F = \F$ on the orthogonal complement of $\C.\xi_{p_{0}}^{L}$ and $F(\xi_{p_{0}}^{L}) = \Omega$, which is unitary. Then, the triple $(L, \widetilde{K}, F)$ (where $L$ is endowed with the representation $\pi\circ\lambda_{\G}$, that we still denote by $\pi$) defines a class $\widetilde{\gamma}\in \KK(C_{\text{max}}(\G), \C)$ which is equal to $\lambda_{\G}^{*}(\gamma) - [\varepsilon]$.

\begin{rem}\label{rem:commutation}
Note that we have $F\circ\pi(a)\circ F^{*} = \widetilde{\rho}(a)$ for every $a\in A_{p_0}$.
\end{rem}

\subsection{The homotopy}

We now want to construct an homotopy in $\KK(C_{\text{max}}(\G), \C)$ from $\widetilde{\gamma}$ to $0$. This will be done using deformations of both representations $\pi$ and $\widetilde{\rho}$, conjugating them by some path of unitaries. We thus first have to define these unitaries. Recall that if $e\in\E(\Gr)$ and $q\in\V(\Gr)$, the unitary
\begin{equation*}
u_{e}^{q}\in\mathcal{L}_{A_{q}}\left(\mathcal{H}_{r(e), q}, \mathcal{H}_{s(e), q}\right)
\end{equation*}
induces in a canonical way an operator in $\B\left(L_{r(e), q}, L_{s(e), q}\right)$ still denoted $u_{e}^{q}$. This procedure respects composition and adjoints, hence $u_{e}^{q}$ is unitary for every $e\in\E(\Gr)$ and $q\in\V(\Gr)$. We also define, for $e,f\in\E(\Gr)$, the unitaries $u_{e}^{f}\in \B\left(K_{r(e), f},K_{s(e), f}\right)$ which are induced by the unitaries
\begin{equation*}
u_{e}^{s(f)}\in\mathcal{L}_{A_{s(f)}}\left(\mathcal{H}_{r(e), s(f)}, \mathcal{H}_{s(e), s(f)}\right).
\end{equation*}

\subsubsection{Deformation of $\widetilde{\rho}$}

For $q\in\V(\Gr)$, set $\widetilde{K}_{q} = K_{q}\oplus \C.\Omega$ endowed with the sum $\widetilde{\rho}_q$ of $\rho_q\circ\lambda_{\G}$ and the trivial representation. For $e\in \E^{+}(\Gr)$, set
\begin{equation*}
u_{e}^{K} = \Id_{\C.\Omega}\oplus\bigoplus_{f\in\E(\Gr)^+}u_{e}^{f}: \widetilde{K}_{r(e)}\rightarrow \widetilde{K}_{s(e)}\quad\text{and}\quad u_{\rev{e}}^{K} = (u_{e}^{K})^{*}.
\end{equation*}
\begin{rem}
For every $e\in \E^{+}(\Gr), b\in B_{e}$ we have $\widetilde{\rho}_{s(e)}(s_{e}(b)) = u^K_{e}\widetilde{\rho}_{r(e)}(r_{e}(b))(u^K_{e})^{*}$. Indeed, since the relation holds on $K_{r(e)}$, it suffices to check it on $\Omega$:
\begin{eqnarray*}
u^K_{e}\widetilde{\rho}_{r(e)}(r_{e}(b))(u^K_{e})^{*}.\Omega &=& u^K_{e}\widetilde{\rho}_{r(e)}(r_{e}(b)).\Omega 
 = \varepsilon_{r(e)}(r_{e}(b))u^K_{e}.\Omega 
 =  \varepsilon_{s(e)}(s_{e}(b))\Omega\\
 &=&  \widetilde{\rho}_{s(e)}(s_{e}(b)).\Omega.
\end{eqnarray*}
\end{rem}
\noindent For $e\in \E^{+}(\Gr)$, we define a unitary $v_e\,:\,\widetilde{K}_{r(e)}\rightarrow \widetilde{K}_{s(e)}$ in the following way:

\vspace{0.2cm}

\noindent\emph{Case 1: e is not a loop.} This means that $s(e) \neq r(e)$. In that case we set $v_{e} = u_{e}^{K}$.

\vspace{0.2cm}

\noindent\emph{Case 2: e is a loop}. Let $p=r(e)=s(e)$. If $f\in\E(\Gr)^+$ and $f\neq e$, $v_{e}$ acts on $K_{p, f}$ by $u_{e}^{f}$ as before. Let $\mathcal{L}_{e}$ be the closure of the linear span of elements $\eta_{p, e}(b)\in K_{p, e}$ for $b = \h{1}\underset{e}{\otimes}b_{1}\underset{e_{2}}{\otimes} \dots \underset{e_{n}}{\otimes}b_{n}$ with $n\geqslant 1$, i.e. the span of "words starting with $e$" and denote by $\mathcal{R}_{e}$ the orthogonal complement of $\mathcal{L}_{e}$ in $K_{p, e}$. In the same way, we define spaces $\mathcal{L}_{\rev{e}}$ and $\mathcal{R}_{\rev{e}}$. Obviously, $u_{e}^{r(e)}$ sends $\mathcal{R}_{\rev{e}}$ (resp. $\mathcal{L}_{\overline{e}}$) bijectively onto $\mathcal{L}_{e}$ (resp. $\mathcal{R}_e$). Moreover,
\begin{equation*}
u_{e}^{r(e)}\left(\mathcal{L}_{\rev{e}}\ominus \C.\eta_{p, e}(\h{1}\underset{\rev{e}}{\otimes} 1)\right) = \mathcal{R}_{e}\ominus \C.\xi^{K}_{p, e}.
\end{equation*}
Let $v_{e}$ be equal to $u^{r(e)}_{e}$ on these spaces and extend it by setting
\begin{itemize}
\item $v_{e}(\Omega) = \xi^{K}_{p, e}$
\item $v_{e}\left(\eta_{p, e}(\h{1}\underset{\rev{e}}{\otimes} 1)\right) = \Omega$
\end{itemize}
Summing up, we have again defined a unitary operator 
\begin{equation*}
v_{e}: \widetilde{K}_{r(e)}=\widetilde{K}_{p}\rightarrow\widetilde{K}_{p}= \widetilde{K}_{s(e)}.
\end{equation*}

\begin{lem}
For every $e\in \E^{+}(\Gr), b\in B_{e}$ we have $\widetilde{\rho}_{s(e)}(s_{e}(b)) = v_{e}\widetilde{\rho}_{r(e)}(r_{e}(b))v_{e}^{*}$.
\end{lem}

\begin{proof}
Since $v_e=u_e^K$ whenever $e$ is not a loop and since the equation is known to hold for $u_e^K$, we may and will assume that $e$ is a loop. Let $p=s(e)=r(e)$. Since the equation is already known to hold for the $u_{e}^{f}$'s and since $v_e$ only differs from the $u_e^{f}$'s on a finite-dimensional subspace, we can restrict our attention to finitely many vectors.
\vspace{0.2cm}

\noindent We have to check the equality on $\Omega$ and $\xi_{p, e}^{K}$:
\begin{eqnarray*}
v_{e}\widetilde{\rho}_{r(e)}(r_{e}(b))v_{e}^{*}.\Omega & = &v_{e}\widetilde{\rho}_{r(e)}(r_{e}(b)).\eta_{p, e}(\h{1}\underset{\rev{e}}{\otimes}1) 
 = \varepsilon_{s(e)}(s_{e}(b))v_{e}.\eta_{p, e}(\h{1}\underset{\rev{e}}{\otimes} 1)\\
&=&  \varepsilon_{s(e)}(s_{e}(b)) \Omega
 =  \widetilde{\rho}_{s(e)}(s_{e}(b)).\Omega
\end{eqnarray*}
using $\h{r_{e}(b)}\underset{\rev{e}}{\otimes} 1 = \h{1}\underset{\rev{e}}{\otimes}s_{e}(b)$. Similarly,
\begin{eqnarray*}
v_{e}\widetilde{\rho}_{r(e)}(r_{e}(b))v_{e}^{*}.\xi_{p, e}^{K}&= &v_{e}\widetilde{\rho}_{r(e)}(r_{e}(b)).\Omega 
 = \varepsilon_{r(e)}(r_{e}(b))v_{e}.\Omega 
 =  \varepsilon_{s(e)}(s_{e}(b)) \xi_{p, e}^{K}\\
&=& \widetilde{\rho}_{s(e)}(s_{e}(b)).\xi_{p, e}^{K}.
\end{eqnarray*}
\end{proof}

\noindent Up to now, $v_{e}$ is only defined for positive edges $e$. If $e$ is a negative edge, we set $v_{e} = (v_{\rev{e}})^{*}$. Let us deform these operators. Because the unitaries $u_{e}^{K}$ satisfy the same relations as the $v_e$'s, we have $(u^K_{e})^{*}v_{e}\in\mathcal{B}\left(\widetilde{K}_{r(e)}\right) \cap \widetilde{\rho}_{r(e)}(B_{e}^{r})'$ for every $e\in\E(\Gr)$. Let $h_{e}$ be a positive element in $\mathcal{B}\left(\widetilde{K}_{r(e)}\right)\cap \widetilde{\rho}_{r(e)}(B_{e}^{r})'$ such that $(u_{e}^{K})^{*}v_{e} = \exp(ih_{e})$. For any $t\in\mathbb{R}$, set
\begin{equation*}
v_{e}^{t} = u_{e}^{K}\exp(ith_{e}).
\end{equation*}
Applying Corollary \ref{cor:universal} with our given fixed maximal subtree $\T$, the vertex $p_{0}$, the representations $\widetilde{\rho}_{q}$ and the unitaries $v_{e}^{t}$ we get, for all $t\in\mathbb{R}$, a representation
\begin{equation*}
\widetilde{\rho}_{t}: C_{\text{max}}(\G)\rightarrow \mathcal{B}(\widetilde{K})
\end{equation*}
such that, for any path $w = (e_{1},\dots, e_{n})$ in $\Gr$ from $p_{0}$ to $p_{0}$ and any $a_{0}\in A_{s(e_{1})}$, $a_{i}\in A_{r(e_{i})}$, we have
\begin{equation*}
\widetilde{\rho}_t(a_{0}u_{e_{1}} \dots u_{e_{n}}a_{n}) = \widetilde{\rho}_{s(e_{1})}(a_{0})v_{e_{1}}^{t} \dots v_{e_{n}}^{t} \widetilde{\rho}_{r(e_{n})}(a_{n}),
\end{equation*}
where $u_{e}$ denotes the canonical unitary in $C_{\text{max}}(\G)$ for every $e\in \E(\Gr)$.

\subsubsection{Deformation of $\pi$}

The deformation of $\pi$ uses unitaries $w_{e} = L_{r(e)}\rightarrow L_{s(e)}$ which are quite similar to the $v_{e}$'s. For every $e\in \E(\Gr)$, the sum of the operators $u_{e}^{q}$ is denoted $u_{e}^{L}$. It is unitary and satisfies, for every $e\in\E(\Gr)$, $(u_{e}^{L})^{*}=u_{\overline{e}}^L$ and, for every $e\in\E(\Gr)$ and $b\in B_{e}$,
\begin{equation*}
u_{e}^{L}\pi_{r(e)}(r_{e}(b))(u_{e}^{L})^{*} = \pi_{s(e)}(s_{e}(b)).
\end{equation*}
Let $e\in\E(\Gr)$. We define the unitary $w_e$ in the following way:

\vspace{0.2cm}

\noindent \noindent\emph{Case 1: If $e$ is not a loop}. We set $w_{e} = u_{e}^{q}$ on $L_{r(e), q}$ for every $q\notin \{s(e), r(e)\}$. Observe that
\begin{equation*}
u_{e}^{r(e)}\left(\xi^{L}_{r(e), r(e)}\right) = \eta_{s(e), r(e)}(\h{1}\underset{e}{\otimes} 1)\text{ and }u_{e}^{r(e)}\left(\eta_{r(e), s(e)}(\h{1}\underset{\rev{e}}{\otimes} 1)\right) = \xi^{L}_{s(e), s(e)}.
\end{equation*}
Thus, we can define a unitary $w_{e}: L_{r(e), r(e)}\oplus L_{r(e), s(e)}\rightarrow L_{s(e), r(e)}\oplus L_{s(e), s(e)}$ by setting
\begin{equation*}
\left\{\begin{array}{ccc}
w_{e}\left(\xi^{L}_{r(e), r(e)}\right) & = & \xi^{L}_{s(e), s(e)} \\
w_{e}\left(\eta_{r(e), s(e)}(\widehat{1}\underset{\rev{e}}{\otimes} 1)\right) & = & \eta_{s(e), r(e)}(\widehat{1}\underset{e}{\otimes} 1)
\end{array}\right.
\end{equation*}
and $w_{e} = u_{e}^{r(e)}\oplus u_{e}^{s(e)}$ on the orthogonal complement of the above vectors.
\vspace{0.2cm}

\noindent\emph{Case 2: If $e$ is a loop}. We set $w_{e} = u_{e}^{L}$.

\vspace{0.2cm}

\noindent In both cases, we get a unitary $w_{e} = L_{r(e)}\rightarrow L_{s(e)}$ satisfying $w_e^*=w_{\overline{e}}$.

\begin{lem}
For every $e\in \E^{+}(\Gr), b\in B_{e}$, we have $\pi_{s(e)}(s_{e}(b)) = w_{e}\pi_{r(e)}(r_{e}(b))w_{e}^{*}$.
\end{lem}

\begin{proof}
We may and will assume that $e$ is not a loop and we only have to check the equality on $\xi^{L}_{s(e),s(e)}$ and $\eta_{s(e), r(e)}(\h{1}\underset{e}{\otimes} 1)$. We have
\begin{eqnarray*}
w_{e}\pi_{r(e)}(r_{e}(b))w_{e}^{*}.\xi^{L}_{s(e),s(e)} & = & w_{e}\pi_{r(e)}(r_{e}(b)).\xi^{L}_{r(e),r(e)}
 =  \varepsilon_{r(e)}(r_{e}(b))w_{e}.\xi^{L}_{r(e),r(e)}\\
 &=&  \varepsilon_{s(e)}(s_{e}(b))\xi^{L}_{s(e),s(e)} 
 =  \pi_{s(e)}(s_{e}(b)).\xi^{L}_{s(e),s(e)}
\end{eqnarray*}
and,
\begin{eqnarray*}
w_{e}\pi_{r(e)}(r_{e}(b))w_{e}^{*}.\eta_{s(e), r(e)}(\h{1}\underset{e}{\otimes} 1) & = & w_{e}\pi_{r(e)}(r_{e}(b)).\eta_{r(e), s(e)}(\h{1}\underset{\rev{e}}{\otimes} 1) 
 =  w_{e}.\eta_{r(e), s(e)}( \h{r_{e}(b)}\underset{\rev{e}}{\otimes} 1) \\
& = & \varepsilon_{s(e)}(s_{e}(b))w_{e}.\eta_{r(e), s(e)}(\h{1}\underset{\rev{e}}{\otimes} 1)
 =  \varepsilon_{s(e)}(s_{e}(b))\eta_{s(e), r(e)}(\h{1}\underset{e}{\otimes} 1) \\
 &= & \pi_{s(e)}(s_{e}(b)).\eta_{s(e), r(e)}(\h{1}\underset{e}{\otimes} 1).
\end{eqnarray*}
\end{proof}

\noindent As for $v_{e}$, we can find $k_{e}\in \mathcal{B}(L_{r(e)}) \cap \pi_{r(e)}(B_{e}^{r})'$ such that $(u_{e}^{L})^{*}w_{e} = \exp(ik_{e})$ and set
\begin{equation*}
w_{e}^{t} = u_{e}^{L}\exp(itk_{e}).
\end{equation*}
Again, using the universal property of Corollary \ref{cor:universal} yields a representation $\pi_{t}$ of the maximal fundamental C*-algebra for all $t\in \mathbb{R}$ satisfying
\begin{equation*}
\pi_{t}(a_{0}u_{e_{1}} \dots u_{e_{n}}a_{n}) = \pi_{s(e_{1})}(a_{0})w_{e_{1}}^{t}\dots w_{e_{n}}^{t} \pi_{r(e_{n})}(a_{n}).
\end{equation*}

\subsubsection{Deformation of the triple}

We will now prove that the representations above yield a degenerate triple at $t = 1$.

\begin{lem}\label{lem:degenerate}
For every $x\in C_{\text{max}}(\G)$ we have $F\pi_{1}(x) F^{*} = \widetilde{\rho}_{1}(x)$.
\end{lem}

\begin{proof}
By Remark \ref{rem:redmax}, it suffices to prove the Lemma for $x\in A_{p_0}$ and $x=a_nu_{e_n}\dots u_{e_1}a_0$ a reduced operator in $C_{\text{max}}(\G)$, where $w=(e_n,\dots,e_1)$ is a path from $p_0$ to $p_0$. When $x\in A_{p_0}$, the relation follows from Remark \ref{rem:commutation}, since $\pi_1(x)=\pi(x)$ and $\widetilde{\rho}_1(x)=\widetilde{\rho}(x)$.
\vspace{0.2cm}

\noindent\emph{Claim. Let $x=a_nu_{e_n}\dots u_{e_1}a_0\in C_{\text{max}}(\G)$ be a reduced operator from $p_0$ to $p_0$. One has
\vspace{0.2cm}
\begin{enumerate}
\item  $[F\pi_{1}(x) F^{*}]\Omega = \widetilde{\rho}_{1}(x)\Omega$.
\vspace{0.2cm}
\item  $[F\pi_{1}(x) F^{*}]\xi^K_{p_0,f} = \widetilde{\rho}_{1}(x)\xi^K_{p_0,f}$ for every $f\in\E(\Gr)^+$.
\end{enumerate}}

\vspace{0.2cm}

\noindent\emph{Proof of the claim.}$(1)$. We have
$$[F\pi_{1}(x) F^{*}]\Omega = \varepsilon_{p_0}(a_0)F\pi_{s(e_n)}(a_n)w_{e_n}\dots w_{e_1}\xi^L_{p_0},\,\,
\widetilde{\rho}_{1}(x)\Omega=\varepsilon_{p_0}(a_0)\widetilde{\rho}_{s(e_n)}(a_n)v_{e_n}\dots v_{e_1}\Omega.$$
\noindent \emph{Case 1: $e_1$ is a loop.} In that case we have, since $x$ is reduced,
$$[F\pi_{1}(x) F^{*}]\Omega =\varepsilon_{p_0}(a_0)F(\eta_{p_0,p_0}( \widehat{a}_n\underset{e_n}{\otimes}\dots\underset{e_1}{\otimes} 1))=\varepsilon_{p_0}(a_0)\left\{\begin{array}{lcl}
\eta_{p_0,e_1}( \widehat{a}_n\underset{e_n}{\otimes}\dots\underset{e_{2}}{\otimes} a_{1})&\text{if}&e_1\geqslant 0,\\
\eta_{p_0,\overline{e}_1}( \widehat{a}_n\underset{e_n}{\otimes}\dots\underset{e_{1}}{\otimes} 1)&\text{if}&e_1\leqslant 0.
\end{array}\right.$$
\noindent If $e_1$ is positive, then $v_{e_1}.\Omega=\xi^K_{s(e_1),e_1}$ and if $e_1$ is negative, then $v_{e_1}.\Omega=\eta_{r(\overline{e}_1),\overline{e}_1}(\widehat{1}\underset{e_1}{\otimes} 1)$. Because $x$ is reduced, we get
$$\widetilde{\rho}_{1}(x).\Omega=\varepsilon_{p_0}(a_0)\left\{\begin{array}{lcl}
\eta_{p_0,e_1}( \widehat{a}_n\underset{e_n}{\otimes}\dots\underset{e_{2}}{\otimes} a_{1})&\text{if}&e_1\geqslant 0,\\
\eta_{p_0,\overline{e}_1}( \widehat{a}_n\underset{e_n}{\otimes}\dots\underset{e_{1}}{\otimes} 1)&\text{if}&e_1\leqslant 0.
\end{array}\right.$$
\noindent \emph{Case 2: $e_1$ is not a loop.} Define $k=\min\{l:s(e_l)=r(e_l)\}$ if this set is non-empty and $k=0$ otherwise. Since $v_e\Omega=\Omega$ when $e$ is not a loop and since $e_l$ is not a loop for $l\leqslant k-1$ we get, for $k\geqslant 1$,
$$\widetilde{\rho}_1(x).\Omega=\varepsilon_{p_0}(a_0)\left(\prod_{l=1}^{k-1}\varepsilon_{s(e_l)}(a_l)\right)\widetilde{\rho}_{s(e_n)}(a_n)v_{e_n}\dots \widetilde{\rho}_{s(e_k)}(a_{k})v_{e_k}.\Omega,$$
and, since $w_e.\xi^L_{r(e),r(e)}=\xi^L_{s(e),s(e)}$ whenever $e$ is not a loop,
$$[F\pi_{1}(x) F^{*}]\Omega =\varepsilon_{p_0}(a_0)\left(\prod_{l=1}^{k-1}\varepsilon_{s(e_l)}(a_l)\right)F\pi_{s(e_n)}(a_n)w_{e_n}\dots \pi_{s(e_k)}(a_{k})w_{e_k}.\xi^L_{r(e_k),r(e_k)}.$$
For $k=0$ the proof is easy since $\widetilde{\rho}_1(x).\Omega=\varepsilon_{p_0}(a_0)\left(\prod_{l=1}^n\varepsilon_{s(e_l)}(a_l)\right)\Omega=[F\pi_{1}(x) F^{*}]\Omega$. When $k\geqslant 1$, $e_k$ is a loop and the end of the computation is the same as in case 1.
\vspace{0.2cm}

\noindent $(2)$. Let $f\in\E(\Gr)^+$. When $s(f)\neq p_0$, let $(f_1,\dots,f_m)$ be the unique geodesic path in $\T$ from $p_0$ to $s(f)$. Recall that if $s(f)\neq p_0$, then we have $\xi^K_{p_0,f}=\eta_{p_0,f}(\widehat{1}\underset{f_1}{\otimes}\dots\underset{f_m}{\otimes} 1)$. Hence,
$$F^*\xi^K_{p_0,f}=\left\{\begin{array}{lcl}
\eta_{p_0,r(f)}(\widehat{1}\underset{f_1}{\otimes}\dots\underset{f_m}{\otimes}\widehat{1}\underset{f}{\otimes} 1)&\text{if}&s(f)\neq p_0\,\,\,\text{and}\,\,f_m\neq\overline{f},\\
\eta_{p_0,r(f_m)}(\widehat{1}\underset{f_1}{\otimes}\dots\underset{f_m}{\otimes} 1)&\text{if}&s(f)\neq p_0\,\,\,\text{and}\,\,f_m=\overline{f},\\
\eta_{p_0,r(f)}(\widehat{1}\underset{f}{\otimes} 1)&\text{if}&s(f)=p_0.
\end{array}\right.$$
Assume that $s(f) \neq p_{0}$ and set $n_{0} = \max\{k, e_{k} = \rev{f}_{k}\}$  (set $n_{0} = 0$ if this set is empty). Observe that if $n_{0} < m$, then 
\begin{eqnarray*}
\widetilde{\rho}_{1}(x)\xi_{p_{0}, f}^{K} & = & \widetilde{\rho}_{s(e_{n})}(a_{n})u^{K}_{e_{n}} \dots u^{K}_{e_{1}}\widetilde{\rho}_{r(e_{1})}(a_{0})\xi_{p_{0}, f}^{K} \\{}
[F\pi_{1}(x)F^{*}]\xi_{p_{0}, f}^{K} & = & [F\pi_{s(e_{n})}(a_{n})u^{L}_{e_{n}}\dots u^{L}_{e_{1}}\pi_{r(e_{1})}(a_{0})F^{*}]\xi_{p_{0}, f}^{K}
\end{eqnarray*}

\vspace{0.2cm}

\noindent Hence, with $a=a_nu_{e_n}\ldots u_{e_1}a_0$, $b=\widehat{1}\underset{f_1}{\otimes}\dots\underset{f_m}{\otimes} 1$ and $b'=\widehat{1}\underset{f_1}{\otimes}\dots\underset{f_m}{\otimes}\widehat{1}\underset{f}{\otimes} 1$ we have
$$\widetilde{\rho}_{1}(x)\xi_{p_0,f}^K=\eta_{p_0,f}(a.b)\quad\text{and}\quad [F\pi_1(x)F^*]\xi_{p_0,f}^K=
\left\{\begin{array}{lcl}
F\eta_{p_0,r(f)}(a.b')&\text{if}&f_m\neq\overline{f},\\
F\eta_{p_0,r(f_m)}(a.b)&\text{if}&f_m=\overline{f}.
\end{array}\right.$$
Now, using Lemma \ref{lem:product} to decompose $a.b$ and $a.b'$ as a sum of reduced tensors and since $n_0<m$, it is easy to see that $[F\pi_1(x)F^*]\xi_{p_0,f}^K=\widetilde{\rho}_{1}(x)\xi_{p_0,f}^K$ in both cases. Hence we may and will assume for the rest of the proof that $n_0=m$.

\vspace{0.2cm}

\noindent Set $a'=a_{m-1}u_{e_{m-1}}\dots a_1u_{e_1}a_0$. We have
$$\widetilde{\rho}_1(x).\xi^K_{p_0,f}=\widetilde{\rho}_{s(e_n)}(a_n)v_{e_n}\dots v_{e_m}.\eta_{r(e_m),f}(a'.b)$$
and
$$[F\pi_1(x)F^*]\xi^K_{p_0,f}=F\pi_{s(e_n)}(a_n)w_{e_n}\dots w_{e_m}.\left\{
\begin{array}{lcl}
\eta_{r(e_m),r(f)}(a'.b')&\text{if}&f_m\neq\overline{f},\\
\eta_{r(e_m),r(f_m)}(a'.b)&\text{if}&f_m=\overline{f}.\end{array}\right.$$

\noindent Since $e_m=\overline{f}_m$, we have
\begin{equation*}
\begin{array}{lclcl}
w_{e_{m}}.\eta_{s(f_{m}),r(f)}(\h{1}\underset{f_{m}}{\otimes}\h{1}\underset{f}{\otimes} 1) &= &\eta_{s(f), r(f)}(\h{1}\underset{f}{\otimes} 1) & \text{if} & f_{m}\neq \rev{f}, \\
w_{e_{m}}.\eta_{s(f_{m}),r(f)}(\h{1}\underset{f_{m}}{\otimes} 1) &=& \eta_{s(f),r(f)}(\h{1}\underset{f}{\otimes} 1) & \text{if} & f_{m} = \rev{f}.
\end{array}
\end{equation*}

\noindent Hence, using this computation and Lemma \ref{lem:product} to decompose $a'.b$ and $a'.b'$ as sums of reduced tensors, we see that the difference $[F\pi_1(x)F^*]\xi^K_{p_0,f}-\widetilde{\rho}_1(x).\xi^K_{p_0,f}$ is equal to
$$
F\pi_{s(e_n)}(a_n)w_{e_n}\ldots w_{e_{m+1}}\pi_{s(f)}(a_m).\eta_{s(f),r(f)}(\widehat{1}\underset{f}{\ot} 1)
-\widetilde{\rho}_{s(e_n)}(a_n)v_{e_n}\dots v_{e_{m+1}}\widetilde{\rho}_{s(f)}(a_m).\xi^K_{s(f),f}.
$$
Note that if $s(f) = p_{0}$, this formula is still valid with $m = 0$ and the rest of the proof also applies. If $e_{m+1} \neq \rev{f}$ or if $\EE_s^f(a_{m})=0$, it is easy to see, because $a$ is reduced, that the difference is $0$. Otherwise, we can replace $a_{m}$ by $1$ since it acts trivially on both sides via the counit. Hence, we assume that $a_m=1$ and $e_{m+1}=\overline{f}$.

\vspace{0.2cm}

\noindent If $f$ is a loop, we have, using $e_{m+1}=\overline{f}$ (in particular $e_{m+1}$ is a negative loop)
$$\widetilde{\rho}_{s(e_n)}(a_n)v_{e_n}\dots v_{e_{m+1}}.\xi^K_{s(f),f}=\varepsilon_{s_{r(f)}}(a_{m+1})\widetilde{\rho}_{s(e_n)}(a_n)v_{e_n}\dots v_{e_{m+2}}.\Omega$$
$$F\pi_{s(e_n)}(a_n)w_{e_n}\ldots w_{e_{m+1}}.\eta_{s(f),r(f)}(\h{1}\underset{f}{\otimes} 1)=\varepsilon_{s_{r(f)}}(a_{m+1})F\pi_{s(e_n)}(a_n)w_{e_n}\ldots w_{e_{m+2}}\xi^L_{r(f),r(f)}.$$
It is easy to check, by induction and since $a$ is reduced, that these two expressions are equal. If now $f$ is not a loop, we have, since $e_{m+1}=\overline{f}$ (in particular $e_{m+1}$ is not a loop)
$$\widetilde{\rho}_{s(e_n)}(a_n)v_{e_n}\dots v_{e_{m+1}}.\xi^K_{s(f),f}=\widetilde{\rho}_{s(e_n)}(a_n)v_{e_n}\dots\widetilde{\rho}_{s(e_{m+1})}(a_{m+1}).\eta_{s(e_{m+1}),f}(\widehat{1}\underset{e_{m+1}}{\ot} 1)$$
$$F\pi_{s(e_n)}(a_n)w_{e_n}\ldots w_{e_{m+1}}.\eta_{s(f),r(f)}(\h{1}\underset{f}{\otimes} 1)$$
$$=F\pi_{s(e_n)}(a_n)w_{e_n}\ldots\pi_{s(e_{m+1})}(a_{m+1}).\eta_{s(e_{m+1}),r(e_{m+1})}(\h{1}\underset{e_{m+1}}{\otimes} 1).$$
Again, it is easy to check by induction (using the fact that $a$ is reduced and $e_{m+1}\leqslant 0$) that these two expressions are equal.$\hfill{\Box}$

\vspace{0.2cm}

\noindent\emph{End of the proof of Lemma \ref{lem:degenerate}.} It follows from the first assertion of the claim and from Remark \ref{rem:commutation} that, for $a,b\in C_{\text{max}}(\G)$ either reduced (from $p_0$ to $p_0$) or in $A_{p_0}$,
\begin{eqnarray*}
[F\pi_{1}(a) F^{*}]\left(\widetilde{\rho}_{1}(b).\Omega\right)&=&[F\pi_{1}(a) F^{*}]( [F\pi_{1}(b) F^{*}]\Omega)
=[F\pi_{1}(ab)F^{*}].\Omega
=\widetilde{\rho}_{1}(ab).\Omega\\
&=&\widetilde{\rho}_{1}(a)(\widetilde{\rho}_{1}(b).\Omega).
\end{eqnarray*}
A similar computation, using the second assertion of the claim, shows that we also have
\begin{equation*}
[F\pi_{1}(a) F^{*}](\widetilde{\rho}_{1}(b).\xi^K_{p_0,f})=\widetilde{\rho}_{1}(a)(\widetilde{\rho}_{1}(b).\xi^K_{p_0,f})
\end{equation*}
for every $f\in\E(\Gr)^+$. Hence, for every $x\in C_{\text{max}}(\G)$ we have $[F\pi_1(x)F^*]\xi=\widetilde{\rho}_1(x)\xi$ for every $\xi\in\mathcal{V}$, where $\mathcal{V}\subset\widetilde{K}$ is the linear span of $\widetilde{\rho}_1(C_{\text{max}}(\G)).\Omega$ and $\widetilde{\rho}_1(C_{\text{max}}(\G)).\xi^K_{p_0,f}$ for $f\in\E(\Gr)^+$. Thus, it suffices to show that $\mathcal{V}$ is dense in $\widetilde{K}$.

\vspace{0.2cm}

\noindent Let $f\in \E^{+}(\Gr)$, let $w = (e_{n}, \dots, e_{1})$ be a path from $p_{0}$ to $s(f)$, set $x = \h{x}_{n} \underset{e_{n}}{\otimes} \dots \underset{e_{1}}{\otimes} x_{0}$ and $x^{\sharp} = x_{n}u_{e_{n}} \dots u_{e_{1}}x_{0}\in C_{\text{max}}(\G)$. Let $g=(f_1,\dots,f_l)$ be the unique geodesic path in $\T$ from $p_{0}$ to $s(f)$. Set $u_g=u_{f_1}\dots u_{f_l}$, $v_g=v_{f_1}\dots v_{f_l}$ and observe that, by definition of $\widetilde{\rho}_1$, we have, $\widetilde{\rho}_1(x^{\sharp}u_g^*)=\widetilde{\rho}_{s(e_n)}(x_n)v_{e_n}\dots v_{e_1}\widetilde{\rho}_{r(e_1)}(x_0)v_g^*$. Moreover, since $f_i\in\E(\mathcal{T})$ for all $i$, none of the $f_i$'s are loops, hence $v_{g}^*(\Omega) = \Omega$ and $v_{g}^*(\xi^{K}_{p_{0}, f}) = \xi^{K}_{s(f), f}$.

\vspace{0.2cm}

\noindent Assume that $e_1\neq\overline{f}$ or $e_1=\overline{f}$ and $\EE^r_{e_1}(x_0)=0$ then,
$$v_{e_1}\widetilde{\rho}_{r(e_1)}(x_0)\xi^K_{s(f),f}=u^K_{e_1}\widetilde{\rho}_{r(e_1)}(x_0)\xi^K_{s(f),f}.$$
Since $x^{\sharp}$ is reduced, we deduce that
$$\widetilde{\rho}_1(x^{\sharp}u_g^*).\xi^K_{p_0,f}=\widetilde{\rho}_{s(e_n)}(x_n)v_{e_n}\dots v_{e_1}\widetilde{\rho}_{r(e_1)}(x_0).\xi^K_{s(f),f}=\eta_{p_{0}, f}(x)\in\mathcal{V}.$$
\noindent Observe that the previous computation also works for $e_1=\overline{f}$ and any $x_0$ whenever $e_1$ is not a loop since in this case we have $v_{e_1}=u_{e_1}^K$.

\vspace{0.2cm}

\noindent Assume now that $e_1=\overline{f}$, $x_0\in B^r_{e_1}$ and $e_1$ is a loop. Then $v_{e_1}(\Omega)=v_f^*(\Omega)=\eta_{s(e_1),f}(\widehat{1}\underset{f}{\ot} 1)$. Since $x^{\sharp}$ is reduced, we get

\begin{eqnarray*}
\widetilde{\rho}_{1}(x^{\sharp}u_{g}^*).\Omega
&=&\widetilde{\rho}_{s(e_n)}(x_n)v_{e_n}\dots v_{e_1}\widetilde{\rho}_{r(e_1)}(x_0).\Omega\\
&=&\varepsilon_{r(e_1)}(x_0)\widetilde{\rho}_{s(e_n)}(x_n)v_{e_n}\dots v_{e_2}\widetilde{\rho}_{s(e_1)}(x_1).\eta_{s(e_1),f}(\widehat{1}\underset{f}{\ot} 1)\\
&=&\varepsilon_{r(e_1)}(x_0)\eta_{p_0,f}(\h{x}_{n} \underset{e_{n}}{\otimes} \dots \underset{e_{1}}{\otimes} 1)=\eta_{p_0,f}(x)\in\mathcal{V},
\end{eqnarray*}
where the last equality follows from Remark \ref{rem:lastspaceK} and the fact that $x_0\in B^r_{e_1}$.
\end{proof}

\subsubsection{End of the proof}

We are now ready to prove that $\h{\G}$ is K-amenable.

\begin{proof}[Proof of Theorem \ref{thm:kamenabilitygraph}]
Let $L_{t}$ and $\widetilde{K}_{t}$ be the Hilbert spaces $L$ and $\widetilde{K}$ endowed with the representations $\pi_{t}$ and $\widetilde{\rho}_{t}$ respectively. Note that since the unitaries $v_{e}^{t}$ and $w_{e}^{t}$ are finite-dimensional perturbations of $u_{e}$, the maps $\pi_{t}(a) - \pi(a)$ and $\widetilde{\rho}_{t}(a) - \rho(a)$ have finite rank for every reduced operator $a$ and every $t\in\R$. This implies that $F\circ\pi_{t}(x) - \widetilde{\rho}_{t}(x)\circ F$ is a compact operator for every $x\in P$, hence $(L_{t}, \widetilde{K}_{t}, F)$ is a KK-element, denoted $\widetilde{\gamma}_{t}$.

\vspace{0.2cm}

\noindent The maps $t\mapsto v_{e}^{t}$ and $t\mapsto w_{e}^{t}$ are norm continuous for every $e\in\E(\Gr)$ by construction, implying that the representation $\pi_{t}$ and $\widetilde{\rho}_{t}$ are pointwise norm continuous. Thus, $(\widetilde{\gamma}_{t})_{t}$ is an homotopy in KK-theory. Setting $t = 0$ gives $\widetilde{\gamma}$ by definition, and it was proven in Lemma \ref{lem:degenerate} that setting $t = 1$ gives a degenerate triple. The proof is therefore complete.
\end{proof}

\section{Appendix: graphs of von Neumann algebras}\label{sec:appendix}

\noindent All the results of Section \ref{sec:graphCstar} can also be worked out in the von Neumann algebraic setting. Since the techniques are the same, we only give a concise exposition of the main points.

\begin{de}
A \emph{graph of von Neumann algebras} is a tuple
$$(\Gr, (M_{q}, \Fi_{q})_{q}, (N_{e}, \Fi_{e})_{e}, (s_{e})_{e}),$$
where $\Gr$ is a connected graph and
\begin{itemize}
\item For every $q\in\V(\Gr)$ and every $e\in\E(\Gr)$, $M_{q}$ and $N_{e}$ are von Neumann algebras  with distinguished normal faithful states $\Fi_{q}$ and $\Fi_{e}$ respectively.
\item For every $e\in \E(\Gr)$, $N_{\rev{e}} = N_{e}$ and $\Fi_{\rev{e}} = \Fi_{e}$.
\item For every $e\in \E(\Gr)$, $s_{e}: N_{e} \rightarrow M_{s(e)}$ is a unital faithful normal state-preserving $*$-homomorphism such that $\sigma^{s(e)}_{t}\circ s_{e} = s_{e}\circ \sigma^{e}_{t}$ for all $t\in \mathbb{R}$, where $\sigma_{t}^{e}$ and $\sigma_{t}^{s(e)}$ are the modular automorphism groups of the states $\Fi_{e}$ and $\Fi_{s(e)}$ respectively.
\end{itemize}
For every $e\in \E(\Gr)$ we set $r_{e} = s_{\rev{e}}: N_{e} \rightarrow M_{r(e)}$, $N_{e}^{s} = s_{e}(N_{e})$ and $N_{e}^{r} = r_{e}(N_{e})$.
\end{de}

\noindent We will always use the shorthand notation $(\Gr, (M_{q})_{q}, (N_{e})_{e})$ for a graph of von Neumann algebras.

\subsection{Path bimodules}

As in the C*-algebra case, we will use paths in $\Gr$ to build a representation of the fundamental von Neumann algebra. The main difference is that we will work with bimodules over von Neumann algebras instead of Hilbert modules, which is in some sense more tractable. 
\vspace{0.2cm}

\noindent Let $\EE_{e}^{s}$ be the unique state-preserving (for the state $\varphi_e\circ s_e^{-1}$ on $N_e^s$) normal faithful conditional expectation from $M_{s(e)}$ onto $N_{e}^{s}$ and denote its kernel by $M_{s(e)}\ominus N_{e}^{s}$. We canonically identify ${\rm L}^{2}(N_{e}^{s})$ with a closed subspace of ${\rm L}^{2}(M_{s(e)})$.

\vspace{0.2cm}

\noindent For $n\geqslant 1$ and $w = (e_{1}, \dots, e_{n})$ a path in $\Gr$, we define a $M_{s(e_1)}$-$M_{r(e_n)}$-bimodule
$$H_{w} = K_{0} \underset{N_{e_1}}{\otimes} \dots \underset{N_{e_n}}{\otimes} K_{n},$$
where $K_{0} = \Lde(M_{s(e_{1})})$, $K_n=\Lde(M_{r(e_{n})})$ and, for $1\leq i\leq n-1$,
$$K_i=\left\{\begin{array}{lcl}
\Lde(M_{s(e_{i+1})})&\text{if}&e_{i+1}\neq \rev{e}_{i},\\
\Lde(M_{s(e_{i+1})})\ominus \Lde(N_{e_{i+1}}^{s})&\text{if}&e_{i+1}= \rev{e}_{i}.\end{array}\right.$$
We view  $K_{0} = \Lde(M_{s(e_{1})})$ as a $M_{s(e_1)}$-$N_{e_1}$-bimodule, where the left $M_{s(e_1)}$-action is the obvious one and the right $N_{e_1}$-action is given by $\xi\cdot x=\xi s_{e_1}(x)$. Similarly, we view $K_n=\Lde(M_{r(e_{n})})$ as a $N_{e_n}$-$M_{r(e_n)}$-bimodule, where the right $M_{r(e_n)}$-action is the obvious one and the left $N_{e_n}$-action is given by $x\cdot\xi=r_{e_n}(x)\xi$. Finally, for $1\leq i\leq n-1$, we view $K_i$ as a $N_{e_i}$-$N_{e_{i+1}}$-bimodule by setting $x\cdot\xi\cdot y=r_{e_i}(x)\xi s_{e_{i+1}}(y)$.

\vspace{0.2cm}

\noindent For any two vertices $p, q\in \V(\Gr)$, we now set
\begin{equation*}
H_{p, q} = \bigoplus_{w} H_{w},
\end{equation*}
where the sum runs over all paths $w$ in 
$\Gr$ from $p$ to $q$. By convention, the bimodule $H_{p, p}$ also contains the "empty path" $M_{p}$-$M_{p}$-bimodule $H_{\emptyset}=\Lde(M_{p})$.

\subsection{The fundamental von Neumann algebra}

We first have to define unitaries realizing the relations of the fundamental algebra. Let us fix a vertex $p_0\in \V(\Gr)$.

\vspace{0.2cm}

\noindent We define, for every $e\in \E(\Gr)$, an operator $u_{e}: H_{r(e), p_0}\rightarrow H_{s(e), p_0}$ in the following way. Let $w$ be a path from $r(e)$ to $p_0$.
\begin{itemize}
\item If $w$ is the empty path or does not begin with $\bar{e}$ and $\xi\in H_{w}$, we set
$$u_e\xi=\widehat{1}\ot\xi\in H_{(ew)}\subset  H_{s(e),p_0}.$$
\item If $w$ begins with $\bar{e}$ and $\xi\in H_{w}$, $\xi=\widehat{x}_0\ot\xi^{'}$ with $x_0\in M_{r(e)}$, we set
$$u_e\xi=\left\{\begin{array}{ll}
\widehat{1}\ot\xi & \text{if}\,\, x_0\in M_{r(e)}\ominus N_e^r \\
s_e\circ r_e^{-1}(x_0)\xi' & \text{if}\,\, x_0\in N_e^r
\end{array}\right.$$
\end{itemize}
It is easy to check that $u_e$ extends to a unitary operator and $u_e^{*}=u_{\bar{e}}$. Moreover we have,
\begin{eqnarray*}\label{relation}
u_{\bar{e}}s_e(b)u_e & = & r_e(b)\quad\text{for all}\quad b\in N_e
\end{eqnarray*}
Note that $u_e\,:\,H_{r(e),p_0}\rightarrow H_{s(e),p_0}$ commutes with the right actions of $M_{p_0}$.

\vspace{0.2cm}

\noindent The right version $v_e\,:\,H_{p_0,s(e)}\rightarrow H_{p_0,r(e)}$ is defined in a similar way. Let $w$ be a path from $p_0$ to $s(e)$,
\begin{itemize}
\item If $w$ is the empty path or does not end with $\bar{e}$ and $\xi\in H_{w}$, we set
\begin{equation*}
v_e\xi=\xi\ot\widehat{1}.
\end{equation*}
\item If $w$ ends with $\bar{e}$ and $\xi\in H_{w}$, $\xi=\xi^{'}\ot\widehat{x}_n$ with $x_n\in M_{s(e)}$, we set
$$v_e\xi=\left\{\begin{array}{ll}
\xi\ot\widehat{1} & \text{if}\,\, x_n\in M_{s(e)}\ominus N_e^s\\
\xi^{'}r_e\circ s_e^{-1}(x_n) & \text{if}\,\, x_n\in N_e^s
\end{array}\right.$$
\end{itemize}
Again, $v_e$ extends to a unitary operator, $v_e^*=v_{\bar{e}}$, $v_e\,:\,H_{p_0,s(e)}\rightarrow H_{p_0,r(e)}$ commutes with the left action of $M_{p_0}$ and we have, for every $b\in N_{e}$,
$$v_{\bar{e}}\rho(r_e(b))v_e = \rho(s_e(b)),$$
where $\rho$ denotes the right action of $M_{q}$ on $H_{p_0,q}$ for any vertex $q$.

\vspace{0.2cm}

\noindent For a path $w=(e_1,\ldots,e_n)$ we define unitaries
$$u_{w} = u_{e_{1}} \dots u_{e_{n}}\,:\,H_{r(e_n), p_0}\rightarrow H_{s(e_{1}), p_0}\quad\text{and}\quad v_{w}= v_{e_{n}}\dots v_{e_1}\,: \, H_{p_0,s(e_{1})}\rightarrow H_{p_0,r(e_{n})}.$$

\begin{rem}\label{rem:commutationvN}
We may also define, as in the C*-algebra case, the operators $u_e$ and $v_e$ relative to another base $p\in\V(\Gr)$ instead of the fixed base $p_0$. We get unitaries $u_e^{ p }\,:\,H_{r(e),p}\rightarrow H_{s(e),p}$ and $v^{ p }_e\,:\;H_{p, s(e)}\rightarrow H_{p,r(e)}$, satisfying the same relations as before and such that $u_e^{ p }$ commutes with the right $M_p$-actions and $v_e^{ p }$ commutes with the left $M_p$-actions. Moreover, one can easily check that for every $e, f\in \E(\Gr)$,
$$v_f^{s(e)}u_e^{s(f)}=u_e^{r(f)}v_f^{r(e)}.$$
When $w=(e_1,\dots,e_n)$ is a path, if we set $u_{w}^{p} = u^{p}_{e_{1}} \dots  u^{p}_{e_{n}}$ and $v_{w}^{p} = v^{p}_{e_{n}} \dots v^{p}_{e_{1}}$ then, for all paths $w, z$, we have $v_{z}^{s(w)}u_{w}^{s(z)} = u_{w}^{r(z)}v_{z}^{r(w)}$.
\end{rem}

\begin{de}
Let $(\Gr, (M_{q}), (N_{e}))$ be a graph of von Neumann algebras and let $p_{0}\in \V(\Gr)$.
\begin{itemize}
\item The \emph{(left) fundamental von Neumann algebra} of $(\Gr, (M_{q}), (N_{e}))$ in $p_{0}$ is
$$
\pi_{1}(\Gr, (M_{q}), (N_{e})) =\left\langle (u_{z})^{*}M_{q}u_{w} \vert q \in \V(\Gr),w,z \text{ paths from $q$ to $p_{0}$ }\right\rangle\subset\B(H_{p_{0}, p_{0}}).
$$
\item The \emph{right fundamental von Neumann algebra} of $(\Gr, (M_{q}), (N_{e}))$ in $p_{0}$ is
$$
\pi'_{1}(\Gr, (M_{q}), (N_{e})) = \left\langle (v_{z})^{*}\rho(M_{q})v_{w} \vert q \in \V(\Gr),w,z \text{ paths from $p_0$ to $q$ }\right\rangle \subset\B(H_{p_{0}, p_{0}}).
$$
\end{itemize}
\end{de}

\noindent By Remark \ref{rem:commutationvN} it is easy to check that the left and right fundamental von Neumann algebras commute with each other. From now on we denote by $M$ the fundamental von Neumann algebra of $(\Gr, (M_{q}), (N_{e}))$ in $p_{0}$ and by $M_{R}$ its right version.
\begin{rem}\label{rem:commutant}
The fundamental von Neumann algebras do not really depend on the choice of $p_{0}$. This follows from the same argument as in Remark \ref{rem:basepoint} for the C*-algebra case, using a maximal subtree $\T$ of $\Gr$.
\end{rem}

\begin{ex}
Using the graphs of Examples \ref{ex:freeproduct} and \ref{ex:hnn}, one recovers the amalgamated free product and HNN extension constructions for von Neumann algebras.
\end{ex}

\noindent In order to have a convenient way of defining both left and right reduced operators, we introduce the following notation inspired from \cite[Sec 5.2]{serre1977arbres}: let $w = (e_{1}, \dots, e_{n})$ be a path in $\Gr$ from $p$ to $q$ and let $\mu = (x_{0}, \dots, x_{n})$, where $x_0\in M_p$ and $x_{i}\in M_{r(e_{i})}$ for $1\leqslant i\leqslant n$. We say that $(w, \mu)$ is a \emph{reduced pair} (from $p$ to $q$) if, for $1\leqslant i\leqslant n-1$, $e_{i+1} = \rev{e}_{i}$ implies that $x_{i}\in M_{r(e_i)} \ominus N_{e_{i}}^{r}$. Whenever $(w, \mu)$ is a reduced pair, we define a \emph{left reduced operator} (from $p$ to $q$) $\vert w, \mu\vert: H_{q, p_0} \rightarrow H_{p, p_0}$ by
\begin{equation*}
\vert w,\mu\vert = x_{0}u_{e_{1}} \dots u_{e_{n}}x_{n}
\end{equation*}
and a \emph{right reduced operator} (from $p$ to $q$) $\vert w, \mu\vert': H_{p_{0},p} \rightarrow H_{p_0,q}$ by
\begin{equation*}
\vert w, \mu\vert' = \rho(x_{n})v_{e_{n}} \dots v_{e_{1}}\rho(x_{0}).
\end{equation*}
It is easy to check that, as in the C*-algebra case, the left reduced operators from $p_0$ to $p_0$ are in $M$. Moreover, the linear span of $M_{p_0}$ and the left reduced operators from $p_0$ to $p_0$ is a $\sigma$-weakly dense $*$-subalgebra of $M$. Similarly, the right reduced operators from $p_0$ to $p_0$ are in $M_R$ and the linear span of $\rho(M_{p_0})$ and the right reduced operators from $p_0$ to $p_0$ is a $\sigma$-weakly dense $*$-subalgebra of $M_{R}$.

\subsection{Modular theory}

Let $H= H_{p_{0}, p_{0}}$ and $\Omega = \h{1}_{p_{0}}\in \Lde(M_{p_{0}})\subset H$. We now build the fundamental state on $M$ and describe its modular theory. First note that, whenever $(w, \mu)$ is a reduced pair with $w = (e_{1}, \dots, e_{n})$ and $\mu = (x_{0}, \dots, x_{n})$, we have
\begin{equation*}
\vert w, \mu\vert.\Omega = \h{x}_{0} \otimes\dots\otimes \h{x}_{n} = \vert w, \mu\vert'.\Omega.
\end{equation*}
Hence, $\Omega$ is a cyclic vector for both $M$ and $M_{R}$. Since $M_{R}\subset M'$, $\Omega$ is also a separating vector for $M$. Hence, the normal state defined by
\begin{equation*}
\Fi(x) = \langle \Omega, x.\Omega\rangle
\end{equation*}
is faithful and $(H, \Omega)$ is its GNS construction. Note also that $\varphi(x) = 0$ for any reduced operator $x\in M$. From the GNS construction, one can easily compute the modular operators. To simplify notations, set
\begin{equation*}
\Sigma_{w}: \left\{\begin{array}{ccc}
H_{w} & \rightarrow & H_{\overline{w}} \\
\xi_{0} \otimes \dots \otimes \xi_{n} & \mapsto & \xi_{n} \otimes \dots \otimes \xi_{0}
\end{array}\right.
\end{equation*}
for any path $w=(e_1,\dots, e_n)$ in $\Gr$, where $\overline{w}=(\overline{e}_n,\dots,\overline{e}_1)$. The following is easy to check.

\begin{prop}\label{prop:modulargraph}
Let $J$ and $\nabla$ be respectively the modular conjugation and the modular operator of $\Fi$. We have
\begin{equation*}
J = \bigoplus_{w = (e_{1}, \dots, e_{n})}(J_{r(e_{n})}\otimes \dots \otimes J_{s(e_{1})})\Sigma_{w}\quad\text{and}\quad
\nabla=\bigoplus_{w = (e_{1}, \dots, e_{n})}\nabla_{s(e_{1})}\otimes \nabla_{r(e_{1})}\otimes \dots \otimes \nabla_{r(e_{n})},
\end{equation*}
where $J_{q}$ and $\nabla_{q}$ denote respectively the modular conjugation and modular operator of $\Fi_{q}$. In particular, $\varphi$ is a trace if and only if $\varphi_q$ is a trace for any $q\in\V(\Gr)$.
\end{prop}

\begin{rem}
Let $(w, \mu)$ be a reduced pair with $w$ a path in $\Gr$ from $p_{0}$ to $p_{0}$ and observe that
\begin{equation*}
J\vert w, \mu\vert^{*}J = \vert w,\mu\vert'.
\end{equation*}
This implies that $M_{R}$ is exactly the commutant $M'$ of $M$.
\end{rem}

\noindent Let us fix a maximal subtree $\T$ of $\Gr$. The embeddings $\pi^{\T}_{q}: M_{q} \rightarrow M$ given by $\T$ (defined as in Section \ref{quotient}) are state-preserving and commute with the modular automorphism groups by Proposition \ref{prop:modulargraph}. Hence, for every $q\in \V(\Gr)$ and $e\in \E(\Gr)$, there are conditional expectations $\EE_{q}: M\rightarrow  \pi_{q}(M_{q})$ and $\EE_{e}^{s}: M\rightarrow \pi_{s(e)}(N_{e}^{s})$ such that 
\begin{equation*}
\left\{\begin{array}{ccc}
\varphi_{q}\circ\pi_{q}^{-1}\circ \EE_{q} & = & \varphi \\
\varphi_{e}\circ\pi_{s(e)}^{-1}\circ \EE_{e}^{s} & = & \varphi_{e}
\end{array}\right.
\end{equation*}

\subsection{Universal property and unscrewing process}

We can now give a universal property for the fundamental von Neumann algebra. With this in hand, we will easily get an unscrewing process. Assume that we have
\begin{itemize}
\item For every $p\in\V(\Gr)$, a Hilbert space $K_p$ and a faithful normal unital $*$-homomorphism $\pi_p\,:\,M_p\rightarrow\mathcal{B}(K_p)$.
\item For every $e\in\E(\Gr)$, a unitary $w_{e}\in \B(K_{r(e)}, K_{s(e)})$ such that $w_{e}^{*} = w_{\rev{e}}$ and for every $e\in \E(\Gr)$ and $b\in N_{e}$,
\begin{equation*}
w_{\rev{e}}\pi_{s(e)}(s_{e}(b))w_{e} = \pi_{r(e)}(r_{e}(b)).
\end{equation*}
\end{itemize}
Let $L$ be the $\sigma$-weakly closed linear span of $\pi_{p_{0}}(M_{p_{0}})$ and all elements of the form
\begin{equation*}
\pi_{s(e_{1})}(a_{0})w_{e_{1}} \dots w_{e_{n}}\pi_{r(e_{n})}(a_{n})
\end{equation*}
in $\B(K_{p_{0}})$, where $n\geqslant 1$, $(e_{1}, \dots, e_{n})$ is a path in $\Gr$ from $p_{0}$ to $p_{0}$, $a_{k}\in M_{r(e_{k})}$, $1 \leqslant k \leqslant n$, and $a_{0}\in M_{p_{0}}$. Using the relations, we see that $L$ is a von Neumann algebra. We assume moreover the existence of a faithful normal state $\psi\in L_{*}$ such that $\psi\circ\pi_{p_0}=\varphi_{p_0}$ and, for every reduced pair $(w,\mu)$ from $p_0$ to $p_0$ we have,
$$\psi(\pi_{s(e_{1})}(x_{0})w_{e_{1}} \dots w_{e_{n}}\pi_{r(e_{n})}(x_{n})) = 0,$$
where $w=(e_1,\ldots,e_n)$ and $\mu=(x_0,\ldots,x_n)$.

\begin{prop}\label{prop:universalvonneumann}
With the hypothesis and notations above, there exists a unique normal $*$-isomorphism
\begin{equation*}
\pi: \pi_{1}(\Gr, (M_{q}), (N_{e}), p_{0})\rightarrow L
\end{equation*}
such that $\pi = \pi_{p_{0}}$ on $M_{p_{0}}$ and
\begin{equation*}
\pi(a_{0}u_{e_{1}} \dots u_{e_{n}} a_{n}) = \pi_{s(e_{1})}(a_{0})w_{e_{1}} \dots w_{e_{n}}\pi_{r(e_{n})}(a_{n})
\end{equation*}
for every reduced operator $a_{0}u_{e_{1}} \dots u_{e_{n}} a_{n}\in \pi_{1}(\Gr,(M_{q}), (N_{e}), p_{0})$. Moreover, $\varphi = \psi\circ \pi$.
\end{prop}

\noindent The proof is the same as the one of Proposition \ref{prop:universalreduced}. Using this and the universal properties of von Neumann amalgamated free product and von Neumann HNN extensions, we get the following straightforward von Neumann version of the unscrewing process.

\begin{prop}
Let $(\Gr, (M_{q})_{q}, (N_{e})_{e})$ be a graph of von Neumann algebras. Then, the fundamental von Neumann algebra $\pi_{1}(\Gr, (M_{q})_{q}, (N_{e})_{e})$ is isomorphic to an inductive limit of iterations of amalgamated free products and HNN extensions of vertex algebras amalgamated over edge algebras.
\end{prop}

\noindent As an application of the unscrewing process, we prove a permanence property. In the next statement, $\mathcal{R}$ denotes the hyperfinite $\II_{1}$ factor and $\mathcal{R}^{\omega}$ is an ultraproduct of $\mathcal{R}$.

\begin{cor}\label{cor:embedable}
Let $(\Gr, (M_{q})_{q}, (N_{e})_{e})$ be a graph of \emph{finite} von Neumann algebras.
\begin{enumerate}
\item If all the algebras $N_{e}$ are amenable, then the reduced fundamental von Neumann algebra embeds into $\mathcal{R}^{\omega}$ if and only if all the algebras $M_{p}$ embed into $\mathcal{R}^{\omega}$.
\item If all the algebras $N_{e}$ are finite-dimensional, then the fundamental von Neumann algebra has the Haagerup property if and only if all the algebras $M_{p}$ have the Haagerup property.
\end{enumerate}
\end{cor}

\begin{proof}
$(1)$. One implication is obvious. For the other one, observe that since the property of embeddability in $\mathcal{R}^{\omega}$ is stable under inductive limits, we may assume that the graph $\Gr$ is finite. Hence, by induction, it suffices to prove the corollary for amalgamated free products and HNN extensions. The case of an amalgamated free product was done in \cite{browndykemajung2008entropy}. Moreover, by a result of Ueda \cite{ueda2008hnn}, an HNN extension of von Neumann algebras is a corner in an amalgamated free product of von Neumann algebras. More precisely (see \cite{fimavaes2012hnn}), we have
\begin{equation*}
\HNN(M,N,\theta)\simeq e_{11}\left((M_{2}(\C)\otimes M)\ast_{N\oplus N}(M_{2}(\C)\otimes N)\right)e_{11}.
\end{equation*}
By the result of \cite{browndykemajung2008entropy}, $(M_{2}(\C)\otimes M)\ast_{N\oplus N}(M_{2}(\C)\otimes N)$ is embeddable in $\mathcal{R}^{\omega}$ whenever $N$ is amenable and $M$ is embeddable in $\mathcal{R}^{\omega}$. This proves $(1)$.

\vspace{0.2cm}

\noindent $(2)$. Again, one implication is obvious. For the other implication we see, using the same arguments (inductive limit, induction and reduction of the case of an HNN extension to an amalgamated free product), that it is suffices to treat the case of an amalgamated free product of von Neumanna algebras with the Haagerup property amalgamated over a finite-dimensional von Neumann algebra. This was done in \cite{boca1993method}.
\end{proof}

\subsection{Relation with the C*-algebraic construction}\label{section:relationcstarvn}

Let $(\Gr, (A_{q}, \varphi_{q})_{q}, (B_{e}, \varphi_{e})_{e})$ be a graph of C*-algebras with faithful states and let $P(p_{0})$ be the reduced fundamental C*-algebra in $p_{0}\in\V(\Gr)$. For $e\in\E(\Gr)$ (resp. $q\in\V(\Gr)$), let $N_{e}$ (resp. $M_{q}$) be the von Neuman algebra generated by $B_{e}$ (resp. $A_{q}$) in the GNS representation of $\varphi_{e}$ (resp. $\varphi_{q}$). Since the states are faithful, we may view $B_{e}$ and $A_{q}$ as subalgebras of $N_{e}$ and $M_{q}$ respectively. We still denote by $\varphi_{e}$ and $\varphi_{q}$ the unique normal extension of $\varphi_{e}$ and $\varphi_{q}$ to states on $N_{e}$ and $M_{q}$ respectively.

\vspace{0.2cm}

\noindent In the sequel, \emph{we do assume that the states $\varphi_{e}$ and $\varphi_{q}$ are still faithful on $N_{e}$ and $M_{q}$}. Since $s_{e}\,:\,B_{e}\rightarrow A_{s(e)}$ is state-preserving, it extends uniquely to a normal faithful and unital $*$-homomorphism, still denoted $s_{e}$, from $N_{e}$ to $M_{s(e)}$ which is again state-preserving. Also, because $\EE_{e}^{s}\,:\,A_{s(e)}\rightarrow B^{s}_{e}$ preserves the state $\varphi_{e}\circ s_{e}^{-1}$ on $B_{e}^{s}$ (and is faithful because $\varphi_{s(e)}$ is), it extends uniquely to a normal linear map, again denoted $\EE_{e}^{s}$, from $M_{s(e)}$ to $N^{s}_{e} = s_{e}(N_{e})$ which is a state-preserving conditional expectation, i.e. satisfies
\begin{equation*}
\varphi_{e}\circ s_{e}^{-1}\circ\EE_{e}^{s} = \varphi_{s(e)}.
\end{equation*}
By Takesaki's Theorem on conditional expectations \cite{takesaki1972conditional}, it follows that $\sigma_{t}^{s(e)}\circ s_{e} = s_{e}\circ\sigma_{t}^{e}$ for all $t\in\mathbb{R}$. Hence, we have produced a graph of von Neumann algebras and we denote by $M(p_{0})$ its fundamental von Neumann algebra in $p_{0}$.

\begin{prop}\label{prop:quantumvonneumann}
Under the above assumptions, the fundamental state on $P(p_{0})$ is faithful and $M(p_{0})$ is isomorphic to the von Neumann algebra generated by $P(p_{0})$ in the GNS representation of the fundamental state.
\end{prop}

\begin{proof}
Recall that $\varphi$ denote the fundamental state on $M(p_{0})$ and denote by $\psi$ the one on $A(p_{0})$. We use the notations of Proposition \ref{prop:base} in which the GNS construction $(H_{p_{0}, \varphi_{p_{0}}}, \pi_{p_{0}, \varphi_{p_{0}}}, \xi_{p_{0}, \varphi_{p_{0}}})$ of the state $\psi$ on $P(p_{0})$ is studied (see Remark \ref{rem:GNSstates}). Define $V\,:\, H_{p_{0}, \varphi_{p_{0}}}\rightarrow H_{p_{0}, p_{0}}$ in the following way: for $a\in A_{p_0}$ we set $V(\pi_{p_{0}, \varphi_{p_{0}}}(a)\xi_{p_{0}, \varphi_{p_{0}}}) = a\Omega$ and, for $a_{0}u_{e_{1}}^{p_{0}}\dots u_{e_{n}}^{p_{0}}a_{n}\in P(p_{0})$ a reduced operator, we set
$$V(\pi_{p_{0}, \varphi_{p_{0}}}(a_{0}u_{e_{1}}^{p_{0}}\dots u_{e_{n}}^{p_{0}}a_{n})\xi_{p_{0}, \varphi_{p_{0}}}) = a_{0}u_{e_{1}}\dots u_{e_{n}}a_{n}\Omega.$$
It is easy to check, as in the proof of Proposition \ref{prop:base}, that $V$ is a unitary. Moreover one has $V\pi_{p_{0}, \varphi_{p_{0}}}(a)V^{*} = a$ for every $a\in A_{p_{0}}$ and $V(\pi_{p_{0}, \varphi_{p_{0}}}(a_{0}u_{e_{1}}^{p_{0}}\dots u_{e_{n}}^{p_{0}}a_{n}))V^{*} = a_{0}u_{e_{1}}\dots u_{e_{n}}a_{n}$ for any reduced operator $a_{0}u_{e_{1}}^{p_{0}}\dots u_{e_{n}}^{p_{0}}a_{n}\in P(p_{0})$. It follows that $V\left(\pi_{p_{0}, \varphi_{p_{0}}}(P(p_{0}))\right)''V^{*} = M(p_{0})$. Hence, $\pi(x) = VxV^{*}$ is a $*$-isomorphism between the von Neumann algebra $\left(\pi_{p_{0}, \varphi_{p_{0}}}(A(p_{0}))\right)''$ and $M(p_{0})$ and, by construction, we have $\varphi\circ\pi\circ\pi_{p_{0}, \varphi_{p_{0}}} = \psi$. Since we proved that $\varphi$ is faithful and since the GNS map $\pi_{p_{0}, \varphi_{p_{0}}}$ is faithful whenever the state $\varphi_{p_{0}}$ on $A_{p_{0}}$ is faithful, this proves that the fundamental state on $P(p_{0})$ is faithful.
\end{proof}


\bibliographystyle{amsalpha}
\bibliography{quantum}

\end{document}